\documentclass[12pt]{article}
\usepackage[margin =1in]{geometry}
\usepackage{amssymb,amsthm}
\usepackage{amsmath}
\usepackage{enumerate}
\usepackage{enumitem}
\usepackage{cite}
\usepackage{color}
\usepackage{subcaption}

\numberwithin{equation}{section}

\newcommand{\localmu}{\mu}
\newcommand{\E}{{\bf E}}
\newcommand{\R}{{\bf R}}

\newcommand{\localconstant}{C}

\newcommand{\error}{E}

\newcommand{\xkball}{\ball{p}{\delta_p}}
\newcommand{\ykball}{\ball{p}{\epsilon_p/2}}

\newcommand{\cU}{\mathcal{U}}

\newcommand{\norm}[1] {\left \| #1 \right \|}
\newcommand{\inclu}[0] {\ar@{^{(}->}}

\newcommand{\gph}{{\rm gph}\,}

\newcommand{\dist}{{\rm dist}}

\newcommand{\EE}{\mathbb{E}}

\newcommand{\RR}{\mathbb{R}}

\newcommand{\cB}{\mathcal{B}}

\newcommand{\cX}{\mathcal{X}}
\newcommand{\cL}{\mathcal{L}}

\newcommand{\cY}{\mathcal{Y}}

\newcommand{\cM}{\mathcal{M}}

\newcommand{\epi}{\mathrm{epi}\,}

\newcommand{\range}{\mathrm{range}}
\newcommand{\cl}{\mathrm{cl}\,}

\newcommand{\cK}{\mathcal{K}}

\newcommand{\conv}{\mathrm{conv}}

\newcommand{\abs}[1]{\left| #1 \right|}

\newcommand{\proj}{P}
\newcommand{\dom}{\mathrm{dom}\,}
\newcommand{\argmin}{\operatornamewithlimits{argmin}}

\newcommand{\Diag}{{\rm Diag}}

\newcommand{\NN}{\mathbb{N}}

\newcommand{\expect}[1]{\mathbb{E}\left[#1\right]}

\newcommand{\dotp}[1]{\left\langle #1\right\rangle}

\newtheorem{thm}{Theorem}[section]
\newtheorem{lemma}{Lemma}[section]

\newtheorem{definition}[thm]{Definition}
\newtheorem{proposition}[thm]{Proposition}
\newtheorem{lem}[thm]{Lemma}
\newtheorem{cor}[thm]{Corollary}

\newtheorem{assumption}{Assumption}

\newcommand{\ub}{\texttt{ub}}

\newcommand{\cF}{\mathcal{F}}

\newtheorem{example}{Example}[section]

\newcommand{\paren}[1]{ \left( #1 \right) }
\theoremstyle{remark}
\newtheorem{claim}{Claim}

\newcommand{\sphere}{\mathbb{S}}
\newcommand{\unif}{\text{Unif}}
\newcommand{\discrete}{k}

\newcommand{\perturb}{\nu}
\newcommand{\stepsize}{\alpha}
\newcommand{\subgrad}{v}

\usepackage{mathtools}
\usepackage[boxruled]{algorithm2e}

\newcommand{\tangent}[2]{T_{#1}({#2})}
\newcommand{\tangentM}[1]{\tangent{\cM}{#1}}
\newcommand{\ball}[2]{B_{#2}(#1)}

\numberwithin{equation}{section}

\newcommand{\Nul}{{\rm Null}\,}
\newcommand{\Lin}{{\rm Lin}}
\newcommand{\M}{{\mathcal M}\,}

\newcommand{\ri}{{\rm ri}\,}
\newcommand{\cS}{{\bf S}}
\newcommand{\Y}{{\bf Y}}

\usepackage[boxruled]{algorithm2e}
\title{Active manifolds, stratifications, and convergence to local minima in nonsmooth optimization}

	\author{Damek Davis\thanks{School of ORIE, Cornell University,
Ithaca, NY 14850, USA;
\texttt{people.orie.cornell.edu/dsd95/}. Research of Davis supported by an Alfred P. Sloan research fellowship and NSF DMS award 2047637.}  \qquad Dmitriy Drusvyatskiy\thanks{Department of Mathematics, U. Washington,
Seattle, WA 98195; URL: \texttt{www.math.washington.edu/{\raise.17ex\hbox{$\scriptstyle\sim$}}ddrusv}, email: \texttt{ddrusv@uw.edu}. Research of Drusvyatskiy was supported by NSF DMS-1651851 and CCF-2023166 awards.}\qquad Liwei Jiang\thanks{School of ORIE, Cornell University. Ithaca, NY 14850, USA;
	\texttt{orie.cornell.edu/research/grad-students/liwei-jiang}}}

\date{}
\begin{document}
\maketitle

\begin{abstract}
	We show that the subgradient method converges only to local minimizers when applied to generic Lipschitz continuous and  subdifferentially regular functions that are definable in an o-minimal structure. At a high level, the argument we present is appealingly transparent: we  interpret the nonsmooth dynamics as an approximate Riemannian gradient method on a certain distinguished submanifold that captures the nonsmooth activity of the function. In the process, we develop new regularity conditions in nonsmooth analysis that parallel the stratification  conditions of Whitney, Kuo, and Verdier and extend stochastic processes techniques of Pemantle.
\end{abstract}
\bigskip

\noindent{\bf AMS Subject Classification.} Primary 49J52, 90C30; Secondary 60G07, 32B20

\section{Introduction}
The subgradient method is the workhorse procedure for finding minimizers of Lipschitz continuous functions $f$ on $\R^n$. One common variant, and the one we focus on here, proceeds using the update
\begin{equation}\label{eqn:SGD_intro}
x_{k+1}=x_k-\alpha_k \nabla f(x_k)+\alpha_k\nu_k,
\end{equation}
for some sequence $\alpha_k>0$ and a mean zero noise vector $\nu_k$ chosen by the user. 
As long as $\nu_k$ is absolutely continuous with respect to the Lebesgue measure, the algorithm will only encounter points at which $f$ is differentiable and therefore the recursion \eqref{eqn:SGD_intro} is well defined. 
The typical choice of $\alpha_k$, and one that is well-grounded in theory, is proportional to $k^{-\gamma}$ for  $\gamma\in (1/2,1)$. The subgradient method is core to a wide array of tasks in computational mathematics and applied sciences, such as in statistics, machine learning, control, and signal processing. Despite its ubiquity and the striking simplicity of the evolution equation \eqref{eqn:SGD_intro}, the following question remains  open.

\begin{quote}
	\centering
	Is there a broad class of nonsmooth and nonconvex functions for which the subgradient dynamics \eqref{eqn:SGD_intro} are sure to converge only to local minimizers?
\end{quote}

In order to better situate the question, let us look at the analogous question for smooth functions, where the answer is entirely classical. Indeed, the seminal work of Pemantle \cite{pemantle1990nonconvergence} shows that the subgradient method applied to a Morse function either diverges or converges to a local minimizer. Conceptually, the nondegeneracy of the Hessian stipulated by the Morse assumption ensures that around every extraneous critical point, the function admits a direction of negative curvature. Such directions ensure that the stochastic process \eqref{eqn:SGD_intro} locally escapes any neighborhood of the extraneous critical point.
Aside from being generic, the Morse assumption or rather the slightly  weaker strict saddle property is known to hold for a wealth of concrete statistical estimation and learning problems, as shown for example in \cite{ge2016matrix,sun2015nonconvex,bhojanapalli2016global,ge2017no,sun2018geometric}.
Going beyond smooth functions requires new tools. In particular, a positive answer is impossible for general Lipschitz functions, since generic (in  Baire sense) Lipschitz  functions may have highly oscillatory derivatives \cite{borwein2000lipschitz,rockafellar1981favorable}. 
Therefore one must isolate some well-behaved function class to make progress. In this work, we focus on Lipschitz functions that are semi-algebraic, or more generally definable in an o-minimal structure \cite{MR1404337}. The class of definable functions is virtually exhaustive in contemporary applications of optimization, and has been the subject of intensive research over the past decade. 
The following is an informal statement of one of our main results. 

\begin{thm}[Informal]\label{thm:inform}
Let $f$ be a function that is Lipschitz continuous, subdifferentially regular, and is definable in some o-minimal structure. Then for a full-measure set of vectors $v\in\R^n$, the subgradient method applied to the perturbed function $f_v(x)=f(x)-\langle v,x\rangle$ either diverges or converges to a local minimizer of $f_v$.
\end{thm}

Subdifferential regularity is a common assumption in nonsmooth analysis \cite{rockafellar2009variational,clarke2008nonsmooth} and is in particular valid for  weakly convex functions. Weakly convex functions are those for which the assignment $x\mapsto f(x)+\frac{\rho}{2}\|x\|^2$ is convex for some $\rho\in \R$; equivalently, these are exactly the functions whose epigraph has positive reach in the sense of Federer \cite{federer1959curvature}. This function class is broad and includes convex functions, smooth functions with Lipschitz continuous gradient, and any function of the form $f(x)=h(c(x))+r(x)$, where $h$ is a Lipschitz convex function, $r$ is a convex function taking values in $\R\cup\{\infty\}$, and $c(\cdot)$ is a smooth map with Lipschitz Jacobian. Classical literature highlights
the importance of such composite functions in optimization \cite{poliquin1996prox,poliquin1992amenable,rockafellar1981favorable, fletcher1982model,shapiroreducible}, while recent advances in statistical learning and signal processing have further reinvigorated the problem class. For
example, nonlinear least squares, phase retrieval \cite{davis2020nonsmooth,duchi2019solving,eldar2014phase}, robust principal component analysis \cite{candes2011robust, chandrasekaran2011rank}, and adversarial learning \cite{jin2020local,ostrovskii2021efficient} naturally lead to composite/weakly convex problems.
We refer the reader to the recent expository articles \cite{drusvyatskiy2020subgradient,drusvyatskiy2017proximal} for more details on this problem class and its numerous applications.

Though our arguments make heavy use of subdifferential regularity, we conjecture that the conclusion of Theorem~\ref{thm:inform} is valid without this assumption. We note in passing that in the smooth setting, the noiseless gradient method ($\nu_k\equiv 0$) applied to a Morse function is also known to converge only to local minimizers, as long as it is initialized outside of a certain Lebesgue null set \cite{gradient_descent_jason,Lee:2019:FMA:3349830.3349888}. It is unclear how to extend this class of results to the nonsmooth setting, without explicitly incorporating noise injection $\nu_t$ as we do here.

\subsection{Main ingredients of the proof.}
As the starting point, let us recall the baseline guarantee from \cite{davis2021proximal} for the subgradient method when applied to a semi-algebraic function $f$, or more generally one definable in an o-minimal structure. The main result of \cite{davis2020stochastic} shows that for such functions, almost surely,  every limit point $\bar x$ of the subgradient sequence $\{x_k\}$ is Clarke critical. Explicitly, this means that the zero vector lies in the  Clarke subdifferential   
$$\partial_c f(\bar x)=\conv\left\{\lim_{i\to\infty} \nabla f(y_i): y_i\in \dom (\nabla f), y_i\to \bar x \right\}.$$
Therefore, our task reduces to isolating geometric conditions around extraneous Clarke critical points which facilitate local escape of the subgradient sequence. 

The main difficulty in contrast to the smooth setting is that there is no simple analogue of the Morse lemma that can reduce a  nonsmooth function to a common functional form by a diffeomorphism. Instead, a fundamentally different idea is required. Our arguments focus on a certain smooth manifold that captures the ``nonsmooth activity'' of the function near a critical point. Formal models of such manifolds have appeared throughout the optimization literature, notably in \cite{wright1993identifiable,lewis2002active,lemarecha2000,mifflin2005algorithm,shapiroreducible,drusvyatskiy2014optimality}. 
Following \cite{drusvyatskiy2014optimality}, a smooth embedded submanifold $\mathcal{M}$ of $\R^d$ is called {\em active} for $f$ at $\bar x$ if $(i)$ the restriction of $f$ to $\cM$ is smooth near $\bar x$, and $(ii)$ the subgradients $w\in \partial_c f(x)$ are uniformly bounded away from zero at all points $x\in \R^n\setminus\mathcal{M}$ near $\bar x$. For subdifferentially regular functions, such manifolds are geometrically distinctive in that $f$ varies smoothly along $ \mathcal{M}$ and sharply in directions normal to $\mathcal{M}$. 
As an illustration, Figure~\ref{subsec:partlysmooth_man} depicts a nonsmooth function, having the $y$-axis as the active manifold around the the critical point (origin). A critical point $\bar x$ is called an {\em active strict saddle} if $f$ decreases quadratically along some smooth path in the active manifold $\mathcal{M}$ emanating from $\bar x$. Returning to Figure~\ref{fig:activemanifold}, the origin is indeed an active strict saddle since $f$ has negative curvature along the $y$-axis at the origin. 
Our focus on active manifolds and active saddles is justified because these structures are in a sense generic for definable functions. Indeed, the earlier work \cite{davis2021proximal,MR3461323} shows that for a definable function $f$, there exists a full-measure set of perturbations $v\in \mathcal{V}$ such that every critical point $\bar x$ of the tilted function $f_v(x)= f(x)-\langle v,x\rangle$ lies on a unique active manifold and is either a local minimizer or an active strict saddle.

\begin{figure*}[h!]
	\centering
	\begin{subfigure}{0.5\textwidth}
		\centering
		\includegraphics[width=\textwidth]{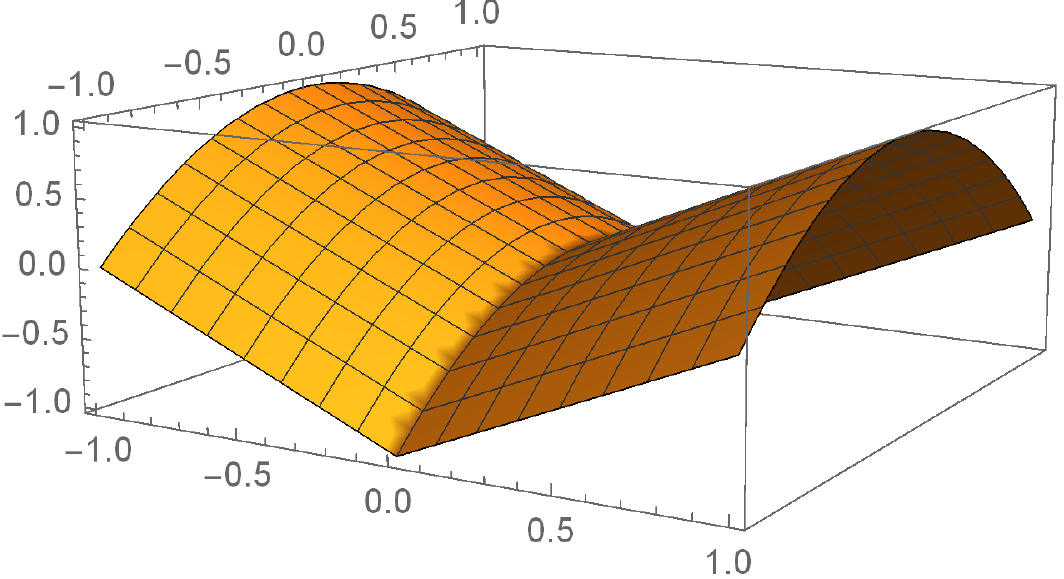}
		\caption{The function $f(x,y)$}\label{subsec:partlysmooth_man}
	\end{subfigure}%
	~ 
	\begin{subfigure}{0.5\textwidth}
		\centering
		\includegraphics[scale=0.5]{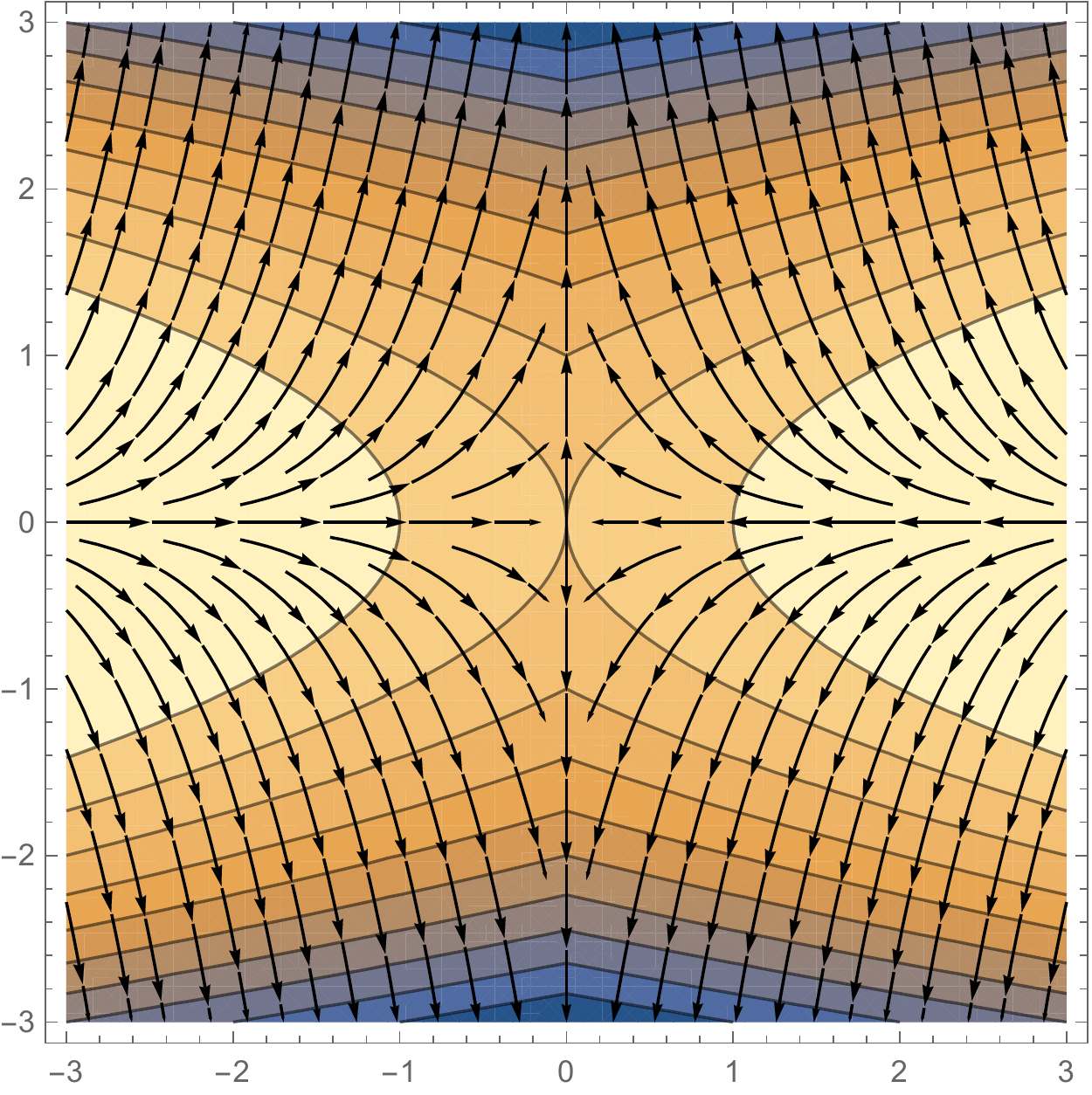}
		\caption{Subgradient flow $\dot{\gamma}\in -\partial_c f (\gamma)$}\label{subsec:partlysmooth_grad}
	\end{subfigure}
	\caption{The $y$-axis is an active manifold for the function $f(x,y)=|x|-y^2$ at the origin.}\label{fig:activemanifold}
\end{figure*}

The importance of the active manifold for subgradient dynamics is best illustrated in continuous time by looking at the trajectories of the differential inclusion $\dot{\gamma}\in -\partial_c f(\gamma)$. Returning to the running example, Figure~\ref{subsec:partlysmooth_grad} shows that the set of initial conditions that are attracted to the critical point by subgradient flow ($x$-axis) has zero measure.
It appears therefore that although the subgradient method never reaches the active manifold, it nonetheless inherits desirable properties from the function along the manifold, e.g., saddle point avoidance. In this work, we rigorously verify this general phenomenon.

Our central observation is that under two mild regularity conditions on $f$, which we will describe shortly, the subgradient dynamics can be understood as an inexact Riemannian gradient method on the restriction of $f$ to $\mathcal{M}$. Explicitly, we will find that the ``shadow sequence''
$
y_k = P_{\cM}(x_k),
$
 satisfies the recursion 
\begin{align}\label{eq:shadowintro}
y_{k+1} = y_k - \alpha_k \nabla_{\cM} f(y_k) + O(\alpha_k^2),
\end{align}
near $\bar x$, where $P_{\cM}(\cdot)$ is the nearest-point projection onto $\mathcal{M}$ and $\nabla_{\cM} f$ denotes the \emph{covariant gradient} of $f$ along $\cM$.\footnote{The covariant gradient $\nabla_{\cM} f(y)$ is the projection onto $T_{\cM}(y)$ of $\nabla \hat f(y)$ where $\hat f$ is any $C^1$ smooth function defined on a neighborhood $U$ of $\bar x$ and that agrees with $f$ on $U\cap \cM$.}
Notice that the error term in \eqref{eq:shadowintro} scales quadratically in the stepsize $\alpha_k$, and this will be crucially used in our arguments. 
The dynamic equation \eqref{eq:shadowintro} will allow us to prove that the subgradient iterates $x_k$ eventually escape from any small neighborhood around an active strict saddle $\bar x$ of $f$.

The validity of \eqref{eq:shadowintro} relies on two regularity properties of $f$ that we now describe. Reassuringly, we will see that both properties are generic in the sense that they hold along the active manifolds of almost every tilt perturbation of a definable function. \paragraph{Regularity property I: aiming towards the manifold.}{
The first condition we require is simply that near the critical point, subgradients are well aligned with directions pointing towards the nearest point on the manifold. Formally, we model this condition with the \emph{proximal aiming} inequality: 
\begin{equation}\label{eqn:aiming_intro}
\langle v,x-P_{\cM}(x)\rangle\geq c\cdot \dist(x,\cM) \qquad \text{for all $x$ near $\bar x$ and $v  \in \partial f_c(x).$}
\end{equation}
for some constant $c>0$.
It is not hard to see that if $f$ is subdifferentially regular and $\mathcal{M}$ is its active manifold, then proximal aiming \eqref{eqn:aiming_intro} is implied by the regularity condition:
\begin{equation}\label{eqn:b_reg_intro}
f(y) \geq f(x) +\langle v,y - x\rangle+o(\|y-x\|)\qquad \textrm{as }x\to \bar x,~ y\stackrel{\cM}{\to} \bar x, \textrm{ with }v\in \partial_c f(x).
\end{equation}
We refer to \eqref{eqn:b_reg_intro} as {\em (b)-regularity of $f$ along $\cM$ at $\bar x$}, for reasons that will be clear shortly.
This estimate stipulates that subgradients $v\in \partial_c f(x)$  yield affine minorants of $f$ up to first-order near $\bar x$, but only when comparing points $x$ and $y\in\cM$. This condition is automatically true for weakly convex functions, and holds in much broader settings as we will see.
 }

\paragraph{Regularity property II: subgradients on and off the manifold.}{
The second regularity property posits that subgradients on and off the manifold are aligned in tangent directions up to a linear error, that is, there exists $C> 0$ satisfying
\begin{equation}\label{eqn:subgrad_connect}
\|P_{T_{\cM}(y)}(\partial_c f(x)- \nabla_{\cM} f(y))\| \leq C\cdot\|x-y\| \qquad \text{for all $x\in {\bf R}^d$ and $y\in \cM$ near $\bar x$}.
\end{equation}
Whenever \eqref{eqn:subgrad_connect} holds, we say that $f$ is {\em strongly (a)-regular along $\cM$}, for reasons that will become apparent shortly.
}

\bigskip
\bigskip

The analytic conditions $(b)$ and strong $(a)$ play a central role in our work. Upon interpreting these conditions geometrically in terms of normals to the epigraph of $f$, a striking resemblance emerges to the classical regularity conditions in stratification theory due to Whitney \cite{whitney1992elementary,whitney1992local,whitney1992tangents}, Kuo \cite{kuo1972characterizations}, and Verdier \cite{verdier1976stratifications}. There is an important distinction, however, that is worth emphasizing. Regularity conditions in stratification theory deal with compatibility between two smooth manifolds. In contrast, we will be concerned with compatibility between a {\em specific nonsmooth set}---the epigraph of $f$---and the {\em specific manifold}---the graph of the restriction of $f$ to the active manifold $\cM$. 
Consequently, a significant part of the paper develops conditions $(b)$ and strong $(a)$ in this more general setting. Some highlights include a thorough calculus, genericity results under linear perturbations, and a proof that strong $(a)$  implies $(b)$ for definable functions. We moreover argue that the two conditions are common in eigenvalue problems because they satisfy the so-called transfer principle. Namely, any orthogonally invariant function of symmetric matrices will satisfy the regularity condition, as long as its restriction to diagonal matrices satisfies the analogous property.  Summarizing, typical functions, whether built from concrete structured examples or from unstructured linear perturbations, admit an active manifold around each critical point along which the objective function is both $(b)$ and strongly $(a)$ regular.

In the final stages of completing this manuscript, we became aware of the concurrent and
independent work \cite{bianchi2021stochastic}. The two papers, share similar core ideas, rooted in strong (a) regularity and proximal aiming. However, the proof of the main result in \cite{bianchi2021stochastic}---avoidance of saddle points---fundamentally relies on a claimed equivalence in \cite[Theorem 4.1]{benaim2005stochastic}, which is known to be false. The most recent draft on arxiv takes a different approach that does not rely on \cite[Theorem 4.1]{benaim2005stochastic}. The same equivalence was used in the follow up preprint \cite{schechtman2021stochastic} by a subset of the authors; this paper has subsequently been withdrawn from arxiv.

\subsection{Outline of the paper.}

The remainder of the paper is organized as follows. Section~\ref{sec:notation} introduces all the necessary preliminaries that will be used in the paper: smooth manifolds~\S\ref{sec:manifolds}, normal cones~\S\ref{sec:normalcones}, subdifferentials~\S\ref{sec:subdifferentials}, and active manifolds~\S\ref{sec:activemanifolds}. Section~\ref{sec:fund_reg_cond} introduces regularity properties of (nonsmooth) sets and functions generalizing the ``compatibility" conditions used in stratification theory. The section closes with a theorem asserting that conditions $(b)$ and strong $(a)$ hold along the active manifold around any limiting critical of generic semialgebraic problems. Section~\ref{sec:algos_main} introduces the algorithms that we study in the paper and the relevant assumptions. Section~\ref{sec:twopillarsmainresults} discusses the two pillars of our algorithmic development (aiming and strong (a)-regularity) and the dynamics of the shadow iteration. Section~\ref{sec:avoidingsaddlepointschapter3} presents the main results of the paper on saddle-point avoidance.
Most of the technical proofs from Section~\ref{sec:twopillarsmainresults} and \ref{sec:avoidingsaddlepointschapter3} appear as Sections~\ref{sec:proofoftwopillars} and \ref{sec:proofsofstochastic}, respectively.

\section{Notation and basic constructions}\label{sec:notation}
We follow standard terminology and notation of nonsmooth and variational analysis, following mostly closely the monograph of Rockafellar-Wets \cite{rockafellar2009variational}. Other influential treatments of the subject include \cite{mordukhovich2006variational,penot2012calculus,clarke2008nonsmooth,borwein2010convex}. Throughout, we let $\E$ and $\bf Y$ denote Euclidean spaces with 
inner products denoted by $\langle\cdot,\cdot \rangle$ and the induced norm $\|x\|=\sqrt{\langle x,x\rangle}$. The symbol ${\bf B}$ will stand for the closed unit ball in $\E$, while $B_r(x)$ will denote the closed ball of radius $r$ around a point $x$. The closure of any set $Q\subset\E$ will be denoted by $\cl Q$, while its convex hull will be denoted by $\conv\, Q$. The relative interior of a convex set $Q$ will be written as $\ri Q$. The lineality space of any convex cone is the linear subspace $\Lin(Q):=Q\cap -Q$.

For any function $f\colon\E\to\R\cup\{+\infty\}$,  the {\em domain}, {\em graph}, and {\em epigraph} are defined as
\begin{align*}
\dom\, f&:=\{x\in \E: f(x)<\infty\},\\
\gph f&:=\{(x,f(x))\in \E\times \R: x\in\dom\, f\},\\
\epi\, f&:=\{(x,r)\in \E\times \R: r\geq f(x)\},
\end{align*}
respectively. We say that $f$ is closed if $\epi f$ is a closed set, or equivalently if $f$ is lower-semicontinuous at every point in its domain.
 If $\cM$ is some subset of $\E$, the symbol $f\big|_{\cM}$ denotes the restriction of $f$ to $\cM$ and we set $\gph f\big|_{\cM}:=(\gph f)\cap (\cM\times \R)$. We say that $f$ is {\em sublinear} if its epigraph is a convex cone, and we then define the {\em lineality space} of $h$  to be $\Lin(h):=\{x: h(x)=-h(-x)\}$. The graph of $h$ restricted to $\Lin(h)$ is precisely the lineality space of $\epi h$.

 The {\em distance} and the {\em projection} of a point $x\in\E$ onto a set $Q\subset\E$ are 
\begin{align*}
d(x,Q):=\inf_{y\in Q}\|y-x\|\qquad\textrm{and}\qquad P_Q(x):=\argmin_{y\in Q}\|y-x\|,
\end{align*}
respectively. The indicator function of a set $Q$, denoted by $\delta_Q\colon\E\to\R\cup\{\infty\}$, is defined to be zero on $Q$ and $+\infty$ off it. 
The {\em gap} between any two closed  cones $U, V\subset\E$ is defined as
$$\Delta(U,V):=\sup\{\dist(u,V): u\in U,\,\|u\|=1\}.$$

\subsection{Manifolds}\label{sec:manifolds}
We next set forth some basic notation when dealing with smooth embedded submanifolds of $\E$. Throughout the paper, all smooth manifolds $\mathcal{M}$ are assumed to be embedded in $\E$ and we consider the tangent and normal spaces to $\mathcal{M}$ as subspaces of $\E$. Thus, a set $\M\subset\E$ is a {\em $C^p$ manifold} (with $p\geq 1$)  if around any point $x\in \M$ there exists an open neighborhood $U\subset\E$ and a $C^p$-smooth map $F$ from $U$ to some Euclidean space $\bf{Y}$ such that the Jacobian $\nabla F(x)$ is surjective and equality $\M\cap U=F^{-1}(0)$ holds. Then the tangent and normal spaces to $\M$ at $x$ are simply $T_{\M}(x):=\Nul(\nabla F(x))$ and $N_{\M}(x):=(T_{\M}(x))^{\perp}$, respectively.  Note that for $C^p$ manifolds $\M$ with $p\geq 1$, the projection $P_{\M}$ is $C^{p-1}$-smooth on a neighborhood of each point $x$ in $\mathcal{M}$, and is $C^p$ smooth on the tangent space $T_{\M}(x)$ \cite{miller2005newton}. Moreover, the inclusion $\range(\nabla P_{\cM}(x)) \subseteq \tangentM{x}$ holds for all $x$ near $\cM$ and the equality $\nabla P_{\cM}(x) = P_{\tangentM{x}}$ holds for all $x \in \cM$.

Let $\M\subset\E$ be a $C^p$-manifold for some $p\geq 1$. Then a function $f\colon \M\to\R$ is called $C^p$-smooth around a point $x\in \M$ if there exists a $C^p$ function $\hat f\colon U\to \R$ defined on an open neighborhood $U$ of $x$ and that agrees with $f$ on $U\cap \M$. Then the {\em covariant gradient of $f$ at $x$} is  defined to be the vector 
$\nabla_{\M} f(x):=P_{T_{\M}(x)}(\nabla \hat f(x)).$
When $f$ and $\M$ are $C^2$-smooth, the {\em covariant Hessian of $f$ at $x$} is  defined to be the unique self-adjoint 
bilinear form $\nabla_{\M}^2 f(x)\colon T_{\M}(x)\times T_{\M}(x)\to \R$ satisfying
$$\langle \nabla_{\M}^2 f(x)u, u\rangle=\frac{d^2}{dt^2}f(P_{\M}(x+tu))\mid_{t=0}\qquad \textrm{for all }u\in T_{\M}(x).$$
If $\cM$ is $C^3$-smooth, then we can identify $\nabla_{\M}^2 f(x)$ with the matrix $P_{T_{\cM}( x)}\nabla^2 \hat f(x)P_{T_{\cM}( x)}$.

\subsection{Normal cones and prox-regularity}\label{sec:normalcones}
The symbol ``$o(h)$ as $h\to 0$'' stands for any univariate function $o(\cdot)$ satisfying $o(h)/h\to 0$ as $h\searrow 0$. The {\em Fr\'{e}chet normal cone} to a set  $Q\subset\E$ at a point $x\in \E$, denoted $\hat N_Q(x)$, consists of all vectors $v\in \E$ satisfying 
\begin{equation}\label{eqn:frechet}
\langle v,y-x\rangle\leq o(\|y-x\|)\quad\textrm{as}\quad y\to x\textrm{ in }Q.
\end{equation}
The {\em limiting normal cone} to $Q$ at $ x\in Q$, denoted by $N_Q(x)$, consists of all vectors $v\in\E$ for which there exist sequences $x_i\in Q$ and $v_i\in \hat N_Q(x_i)$ satisfying $(x_i,v_i)\to (x,v)$. The {\em Clarke normal cone} is the closed convex hull $N^c_Q(x)=\cl\conv\, N_Q(x)$. Thus the inclusions
\begin{equation}\label{eqn:includ_normal_vec}
\hat N_Q(x)\subset N_Q(x)\subset N_Q^c(x),
\end{equation}
hold for all  $x\in Q$. The set $Q$ is called {\em Clarke regular} at $\bar x\in Q$ if $Q$ is locally closed around $\bar x$ and equality $N^c_Q(\bar x)=\hat N_Q(\bar x)$ holds. In this case, all  inclusions in \eqref{eqn:includ_normal_vec} hold as equalities.

A particularly large class of Clarke regular sets consists of those called prox-regular. Following \cite{poliquin2000local,clarke1995proximal}, a locally closed set $Q\subset \E$ is called {\em prox-regular at $\bar x\in Q$} if the projection $P_Q(x)$ is a singleton set for all points $x$ near  $\bar x$. Equivalently \cite[Theorem 1.3]{poliquin2000local}, a locally closed set $Q$ is prox-regular 
at $\bar x\in Q$ if and only if there exist constants $\epsilon,\rho>0$ satisfying 
$$\langle v,y-x\rangle\leq \frac{\rho}{2}\|y-x\|^2,$$
for all $y,x\in Q\cap B_{\epsilon}(\bar x)$ and all normal vectors $v\in N_Q(x)\cap \epsilon\bf{B}$.
If $Q$ is prox-regular at $\bar x$, then the projection $P_Q(\cdot)$ is automatically locally Lipschitz continuous around $\bar x$ \cite[Theorem 1.3]{poliquin2000local}. Common examples of prox-regular sets are convex sets and $C^2$ manifolds, as well as sets cut out by finitely many $C^2$ inequalities under transversality conditions \cite{poliquin2010calculus}. Prox-regular sets are closely related to proximally smooth sets \cite{clarke1995proximal} and sets with positive reach  \cite{federer1959curvature}.

\subsection{Subdifferentials and weak-convexity}\label{sec:subdifferentials}
Generalized gradients of functions can be defined through the normal cones to epigraphs. 
Namely, consider a function $f\colon\E\to\R\cup\{\infty\}$ and a point $x\in \dom\, f$. The  {\em Fr\'{e}chet}, {\em limiting}, and {\em Clarke} subdifferentials of $f$ at $x$ are defined, respectively, as
\begin{equation}\label{eqn:allsubdif}
\begin{aligned}
\hat \partial f(x)&:=\{v\in \E: (v,-1)\in \hat N_{\epi f}(x,f(x))\},\\
 \partial f(x)&:=\{v\in \E: (v,-1)\in N_{\epi f}(x,f(x))\},\\
\partial_c f(x)&:=\{v\in \E: (v,-1)\in N^c_{\epi f}(x,f(x))\}.
\end{aligned}
\end{equation}
Explicitly, the inclusion $v\in \hat \partial f(x)$ amounts to requiring the lower-approximation property:
$$f(y)\geq f(x)+\langle v,y-x\rangle+o(\|y-x\|)\quad \textrm{as}\quad y\to x.$$
Moreover, a vector $v$ lies in $\partial f(x)$ if and only if there exist sequences $x_i\in \E$ and Fr\'{e}chet subgradients $v_i\in \hat \partial f(x_i)$ satisfying $(x_i,f(x_i),v_i)\to (x,f(x),v)$ as $i\to\infty$. If $f$ is locally Lipschitz continuous around $x$, then equality $\partial_c f(x)=\conv\, \partial f(x)$ holds. A point $\bar x$ satisfying $0\in\partial f(x)$ is called {\em critical} for $f$, while 
a point satisfying $0\in\partial_c f(x)$ is called {\em Clarke critical}. The distinction disappears for subdifferentially regular functions. We say that $f$ is {\em subdifferentially regular} at $ x\in \dom f$ if the epigraph of $f$ is Clarke regular at $(x,f(x))$. 

 The three subdifferentials defined in \eqref{eqn:allsubdif} fail to capture the horizontal normals to the epigraph---meaning those of the form $(v,0)$. Such horizontal normals play an important role in variational analysis, in particular for developing subdifferential  calculus rules.
Consequently, we define the {\em limiting} and {\em Clarke horizon subdifferentials}, respectively, by:
\begin{equation}\label{eqn:horiz_subdiff}
\begin{aligned}
\partial^{\infty} f(x)&:=\{v\in \E: (v,0)\in N_{\epi f}(x,f(x))\},\\
\partial^{\infty}_c f(x)&:=\{v\in \E: (v,0)\in N^c_{\epi f}(x,f(x))\}.
\end{aligned}
\end{equation}

A function $f\colon\E\to\R\cup\{\infty\}$ is called {\em $\rho$-weakly convex} if the quadratically perturbed function $x\mapsto f(x)+\frac{\rho}{2}\|x\|^2$ is convex. Weakly convex functions are subdifferentially regular. Indeed, the subgradients of a $\rho$-weakly convex function yield quadratic minorants, meaning
$$f(y)\geq f(x)+\langle v,y-x\rangle-\frac{\rho}{2}\|y-x\|^2$$
all points $x,y\in \dom\, f$ and all subgradients $v\in \partial f(x)$. 
The epigraph of any weakly convex function is a prox-regular set at each of its points.
A primary example of weakly convex functions consists of compositions of Lipschitz convex functions with smooth maps \cite{drusvyatskiy2017proximal,davis2019stochastic}.

\subsection{Active manifolds and active strict saddles}\label{sec:activemanifolds}
Critical points of typical nonsmooth functions lie on a certain manifold that captures the activity of the problem in the sense that critical points of slight linear tilts of the function do not leave the manifold. Such active manifolds have been modeled in a variety of ways, including identifiable surfaces \cite{wright1993identifiable}, partly smooth manifolds \cite{lewis2002active}, $\mathcal{UV}$-structures \cite{lemarecha2000,mifflin2005algorithm}, $g\circ F$ decomposable functions \cite{shapiroreducible}, and minimal identifiable sets \cite{drusvyatskiy2014optimality}.

In this work, we adopt the following formal model of activity,  explicitly used in\cite{drusvyatskiy2014optimality}, where the only difference is that we focus on the Clarke subdifferential instead of the limiting one.

\begin{definition}[Active manifold]\label{defn:ident_man}{\rm 
Consider a function $f\colon\R^d\to\R\cup\{\infty\}$ 
	and 
fix a set $\mathcal{M} \subseteq \RR^d$ 
	containing 
	a point $\bar x$ satisfying $0\in \partial_c f(\bar x)$. 
Then $\mathcal{M}$ 
	is called an
{\em  active} $C^p${\em-manifold around} $\bar x$  
	if 
		there exists 
	a constant $\epsilon>0$ satisfying the following.
\begin{itemize}
\item {\bf (smoothness)}  
The set $\mathcal{M}$ 
	is 
a $C^p$-smooth manifold near $\bar x$ and  the restriction 
	of 	$f$ 
	to $\mathcal{M}$ 
	is $C^p$-smooth near $\bar x$.
\item {\bf (sharpness)} 
The lower bound holds:
$$\inf \{\|v\|: v\in \partial_c f(x),~x\in U\setminus \cM\}>0,$$
where we set $U=\{x\in B_{\epsilon}(\bar x):|f(x)-f(\bar x)|<\epsilon\}$.
\end{itemize}}
\end{definition}

The sharpness condition simply means that the subgradients of $f$ must be uniformly bounded away from zero at points off the manifold that are sufficiently close to $\bar x$ in distance and in function value. The localization in function value can be omitted for example if $f$ is weakly convex or if $f$ is continuous on its domain; see \cite{drusvyatskiy2014optimality} for details.

Intuitively, the active manifold has the distinctive feature that the the function grows linearly in normal directions to the manifold; see Figure~\ref{subsec:partlysmooth_man} for an illustration. This is summarized by the following theorem from  \cite[Theorem D.2]{davis2019stochastic_geo}. 
\begin{proposition}[Identification implies sharpness]\label{prop: sharpness}
	Suppose that a closed function $f\colon\E\to\R\cup\{\infty\}$ admits an active manifold $\cM$ at  a point $\bar x$ satisfying $0\in \hat \partial f(\bar x)$. 
	Then there exist constants $c, \epsilon > 0 $ such that 
	\begin{equation}\label{eqn:growth_cond}
		f(x) - f(\proj_{\cM}(x)) \ge  c\cdot \dist(x, \cM), \qquad \forall x \in B_{\epsilon}(\bar x).
	\end{equation}
\end{proposition}

Notice that there is a nontrivial assumption $0\in \hat \partial f(\bar x)$ at play in Proposition~\ref{prop: sharpness}. Indeed, under the weaker inclusion $0\in  \partial_c f(\bar x)$ the growth condition \eqref{eqn:growth_cond} may easily fail, as the univariate example $f(x)=-|x|$ shows. It is worthwhile to note that under the assumption $0\in \hat \partial f(\bar x)$, the active manifold is locally unique around $\bar x$ \cite[Proposition~8.2]{drusvyatskiy2014optimality}.

Active manifolds are useful  because they allow to reduce many questions about nonsmooth functions to a smooth setting. In particular, the notion of a strict saddle point of smooth functions naturally extends to a nonsmooth setting. The following definition is taken from~\cite{davis2019active}. See Figure~\ref{fig:activemanifold} for an illustration.

\begin{definition}[Active strict saddle]
{\rm
Fix an integer $p\geq 2$ and consider a closed function $f\colon\E\to\R\cup\{\infty\}$ and a point $\bar x$ satisfying $0\in \partial_c f(\bar x)$. We say that $\bar x$ is a {\em $C^p$ strict active saddle point of $f$} if $f$ admits a $C^p$ active manifold $\cM$ at $\bar x$ such that the inequality $\langle \nabla^2_{\cM} f(\bar x)u,u\rangle<0$ holds for some  $u\in T_{\cM}(\bar x)$.}
\end{definition}

It is often convenient to think about active manifolds of slightly tilted functions. Therefore, we say that $\cM$ is an {\em active $C^p$  manifold of $f$ at $\bar x$ for $v\in \partial_c f(\bar x)$} if  $\cM$ is an active $C^p$  manifold for the tilted function $x\mapsto f(x)-\langle v,x\rangle$ at $\bar x$. Active manifolds for sets are defined through their indicator functions. Namely a set $\cM\subset Q$ is {\em  an active $C^p$ manifold of $Q$ at $\bar x\in Q$ for  $v\in N^c_Q (\bar x)$} if it is an active $C^p$ manifold of the indicator function $\delta_Q$ at $\bar x$ for $v$.

\section{The four fundamental regularity conditions}\label{sec:fund_reg_cond}
This section introduces compatibility conditions between two sets, motivated by the works of Whitney \cite{whitney1992elementary,whitney1992local,whitney1992tangents}, Kuo \cite{kuo1972characterizations}, and Verdier \cite{verdier1976stratifications}. Our discussion builds on the recent survey of Trotman \cite{trotman2020stratification}. We illustrate the definitions with examples and prove basic relations between them. It is important to note that these classical works focused on compatibility conditions between smooth manifolds, wherein primal (tangent) and dual (normal) based characterizations are equivalent. In contrast, it will be more expedient for us to base definitions on normal vectors instead of tangents.  The reason is that when applied to epigraphs, such conditions naturally imply some regularity properties for the subgradients, which in turn underpin all algorithmic consequences in the paper.

Throughout this section, we fix two  sets $\cX$ and $\cY$ and a point $\bar x\in \cY$. The reader should keep in mind the most important setting when $\cY$ is a smooth manifold contained in the closure of $\cX$. The phenomena we study are naturally one-sided, and therefore we 
will deal with variational conditions that differ only in the choice of the orientation of the inequalities. With this in mind, in order to simplify notation,  we let $\diamond$ stand for any of the symbols in $\{\leq,=,\geq\}$.
 We begin with the extensions of the two classical conditions of Whitney \cite{whitney1992local,whitney1992tangents}.

\begin{definition}[Whitney conditions]{\rm 
Fix two sets $\cX,\cY\subset\E$. 
\begin{enumerate}
\item We say that $\cX$ is {\em $(a)$-regular along} $\cY$ if for any sequence $x_i\in \cX$ converging to a point $y\in \cY$ and any sequence of normals  $v_i\in N_{\cX}(x_i)$, every limit point of $v_i$  lies in $N_{\cY}(y)$.
\item  We say that $\cX$ is  {\em $(b_{\leq})$-regular along} $\cY$ if the estimate
\begin{equation}\label{eqn:eqn_b}
\langle v,y-x\rangle\leq o(\|y-x\|)
\end{equation}
holds for all $x\in\cX$, $y\in \cY$, and all $v\in N_{\cX}(x)\cap \bf{B}$.
Properties $(b_{\geq})$ and $(b_{=})$ are defined analogously with the inequality in \eqref{eqn:eqn_b} replaced by $\geq$ and $=$, respectively.
\end{enumerate}
More generally, we say that {\em $\cX$ is regular along $\cY$ near a point $\bar x\in \cY$},  in any of the above senses, if there exists a neighborhood $U$ of $\bar x$ such that $\cX\cap U$ is regular along $\cY\cap U$.}
\end{definition}

Both conditions $(a)$ and $(b_\diamond)$ are geometrically transparent. Condition $(a)$ simply asserts that ``limits of normals to $\cX$ are normal to $\cY$''---clearly a desirable property. Figure~\ref{fig:cartan} illustrates how condition $(a)$ may fail using the classical example of the Cartan umbrella $\cX=\{(x,y,z): z(x^2+y^2)=x^3\}$, which is not $(a)$-regular along the $z$-axis near the origin.
Explicitly, condition $(b_{\leq})$ means that for any sequences $x_i\in \cX$ and $y_i\in \cY$ converging to the same point, the condition  
$$\limsup_{i\to\infty}~ \left\langle v_i,\frac{y_i-x_i}{\|y_i-x_i\|}\right\rangle\leq 0,$$ 
holds,  where $v_i\in  N_{\cX}(x_i)$ are arbitrary unit normal vectors.
That is, the angle between the rays spanned by $x_i-y_i$ and any normal vector $v_i\in N_{\cX}(x_i)$ becomes obtuse  in the limit as $x_i\in \cX$ and $y_i\in \cY$ tend to the same point. Conditions $(b_{=})$ and $(b_{\geq})$ have analogous interpretations, with the word obtuse replaced by acute and ninety degrees, respectively. Note that when $\cX$ is a smooth manifold, the normal cone $N_{\cX}(x)$ is a linear subspace, and therefore all three versions of property $(b_{\diamond})$ are equivalent. On the other hand, a prox-regular set $\cX$ is $(b_{\leq})$-regular along any subset $\mathcal{Y}$. Moreover, semismooth sets $\cX$ in the sense of \cite{gfrerer2021semismooth,mifflin1977semismooth} are $(b_{=})$-regular along any singleton set $\cY:=\{\bar x\}$ contained in $\cX$.

\begin{figure}[t!]
	\centering
	\begin{subfigure}{0.3\textwidth}
		\centering
		\includegraphics[width=1\linewidth]{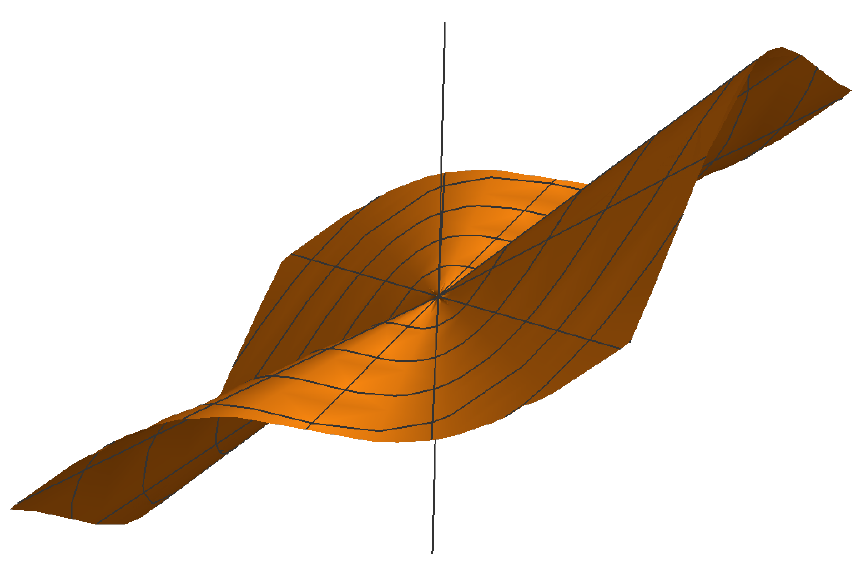}
		\caption{ $x^3=z(x^2+y^2)$\label{fig:cartan}}
\end{subfigure}\qquad\qquad
	\begin{subfigure}{0.3\textwidth}
		\centering
		\includegraphics[width=0.85\linewidth]{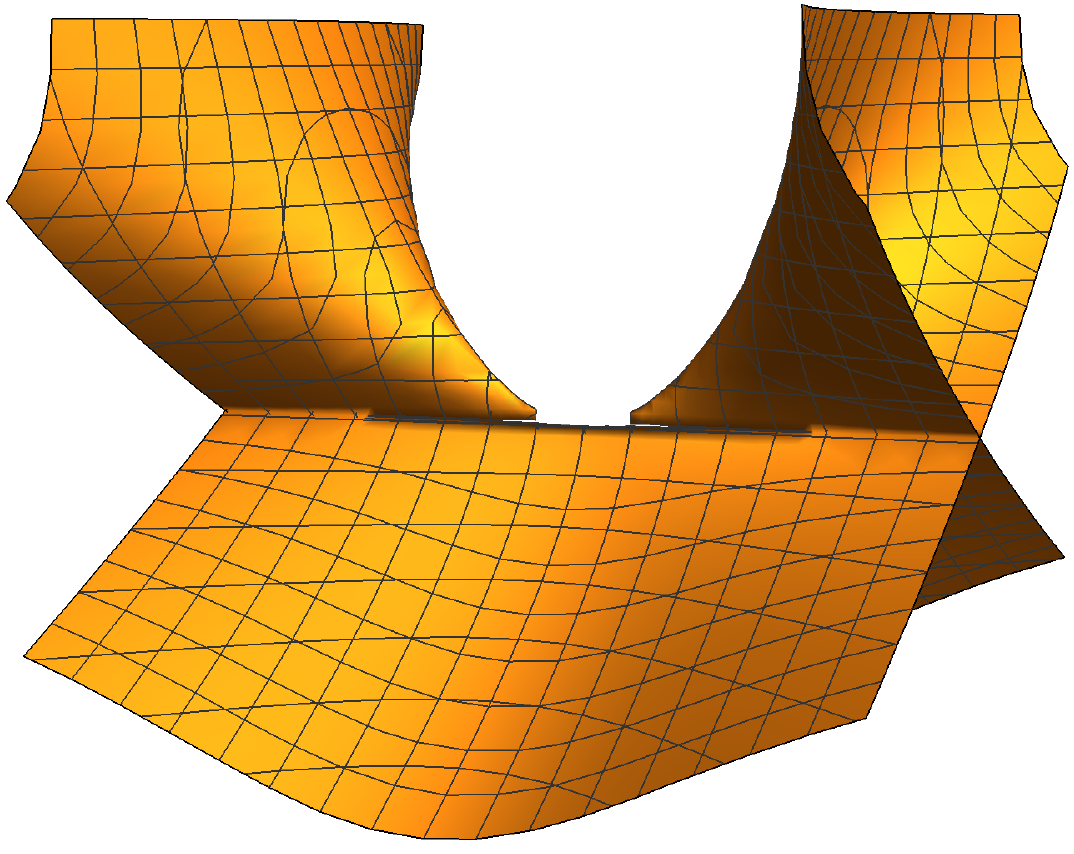}
		\caption{ $y^2=x^2z^2-z^3$\label{fig:hat}}
	\end{subfigure}%
	\caption{Illustrations of conditions $(a)$ and $(b)$.}
\end{figure}

We will use the following simple lemma frequently. It states that whenever $\cY$ is contained in $\cX$, condition $(a)$ simply amounts to the inclusion of normal cones, $N_{\cX}(\bar x)\subseteq N_{\cY}(\bar x)$.

\begin{lemma}[Inclusion of normal cones]\label{lemma:inclusion_normal}
Consider two sets $\cY\subseteq\cX\subseteq\E$. Then $\cX$ is $(a)$-regular along $\cY$ at $\bar x$ if and only if the inclusion $N_{\cX}(y)\subseteq N_{\cY}(y)$ holds for all $y\in \cY$. 
\end{lemma}
\begin{proof}
Suppose first that the inclusion $N_{\cX}(y)\subseteq N_{\cY}(y)$ holds for all $y\in \cY$. Consider a sequence $x_i\stackrel{\cX}{\to} y$
and vectors $v_i\in N_{\cX}(x_i)$ converging to some vector $v$. Then we deduce $v\in N_{\cX}(y)\subseteq N_{\cY}(y)$, as claimed. Conversely, suppose that $\cX$ is $(a)$-regular along $\cY$. Note that the inclusion $\hat N_{\cX}(y)\subset \hat N_{\cY}(y)$ holds trivially for any $y\in \cY$. For any vector $v\in N_{\cX}(y)$, by definition, there exists a sequence  $x_i\stackrel{\cX}{\to}y$ and vectors $v_i\in \hat N_{\cX}(x_i)$ converging to $\bar v$. Condition (a) therefore guarantees $\bar v\in N_{\cY}(y)$, as claimed.
\end{proof}

The following  lemma shows that condition $(b_{\leq})$ implies condition $(a)$ for any sets $\cX$ and $\cY$. Moreover, it is classically known that there exist  smooth manifolds $\cX$ and $\cY$ that satisfy condition $(b_{\leq})$ but not $(a)$; see e.g. \cite{trotman2020stratification}. Therefore $(b_{\leq})$ is strictly stronger than $(a)$.

\begin{lemma}\label{lem:bima}
The   implication $(b_{\leq})~\Rightarrow~(a)$ holds for any sets $\cX$ and $\cY$. Moreover, the implication $(b_{\geq})~\Rightarrow~(a)$ holds if $\hat{N}_{\cY}(\bar x)$ is a linear subspace.
\end{lemma}
	\begin{proof}
 		Suppose  that $\cX$ is $(b_{\leq})$-regular along $\cY$. Consider a sequence $x_i\in\cX$ converging to a point $y\in\cY$ 
 		and vectors $v_i\in N_{\cX}(x_i)$ converging to some vector $v$. It suffices to argue that the inclusion 
 		$v \in \hat N_{\cY}(y)$ holds. To this end, consider an arbitrary sequence $y_j\in \cY\setminus\{y\}$ converging to $y$. Passing to a subsequence, we may suppose that the unit vectors $\frac{y_j-y}{\|y_j-y\|}$ converge.	
		For each $j$ we may choose an index $i_j$ satisfying 
		$\|x_{i_j}-y\|\leq \frac{\|y_j-y\|}{j}$. Straightforward algebraic manipulations directly  imply
		$$\lim_{j\to \infty }\frac{\langle v, y_{j}-y\rangle}{\|y_j-y\|}\leq \limsup_{j\to \infty }\frac{\langle v_{i_j}, y_{j}-x_{i_j}\rangle}{\|y_{j}-x_{i_j}\|}\leq 0,$$
where the last inequality follows from $(b_{\leq})$-regularity. Thus,  $v$  lies in $\hat N_{\cY}(\bar x)$, as claimed. The proof of the implication $(b_{\geq})~\Rightarrow~(a)$
when  $\hat{N}_{\cY}(\bar x)$ is a linear subspace is  analogues.
			\end{proof}

Notice that condition $(a)$ does not specify the rate at which the gap $\Delta(N_{\cX}(x_i),N_{\cY}(y))$ tends to zero as $x_i\in \cX$ tends to $y$. A natural strengthening of the condition, introduced by Verdier \cite{verdier1976stratifications} in the smooth category, requires the gap to be linearly bounded by  $\|x_i-y\|$, with a coefficient that is uniform over all $y\in \cY$.\footnote{What we call strong $(a)$ is often called condition $(w)$, the Verdier condition, or the Kuo-Verdier $(kw)$ condition in the stratification literature.} Condition $(b)$ can be similarly strengthened. The following definition records the resulting two properties.

\begin{definition}[Strong $(a)$ and strong $(b)$]
{\rm
Consider two sets $\cX,\cY$ in $\E$.
\begin{enumerate}
\item 	We say that $\cX$ is {\em strongly (a)-regular along }$\cY$  if there exists a constant $C>0$ satisfying 
\begin{equation}\label{eqn:strong_a}
\Delta(N_{\cX}(x),N_{\cY}(y))\leq C\cdot\|x-y\|,
\end{equation}
		for all $x\in \cX$ and $y\in \cY$.
		\item We say that $\cX$ is {\em strongly $(b_{\leq})$-regular along} $\cY$ if there exists a constant $C>0$ satisfying 
		\begin{equation}\label{eqn:strong_b}
		\langle v,y-x\rangle\leq C \|x-y\|^2,
		\end{equation}
		for all $x\in \cX$, $y\in \cY$, and all vectors $v\in N_{\cX}(x)\cap {\bf B}$. Properties strong $(b_{\geq})$ and strong $(b_{=})$ are defined analogously with the inequality in \eqref{eqn:strong_b} replaced by $\geq$ and $=$, respectively. 
\end{enumerate}
More generally, we say that {\em $\cX$ is regular along $\cY$ near a point $\bar x\in \cY$},  in any of the above senses, if there exists a neighborhood $U$ of $\bar x$ such that $\cX\cap U$ is regular along $\cY\cap U$.}
\end{definition}
 
Summarizing, we have defined four fundamental regularity conditions quantifying the compatibility of two sets $\cX$ and $\cY$. The most important situation for our purposes is when $\cY$ is a smooth manifold contained in $\cX$. The algorithmic importance of these conditions becomes clear when we interpret what they mean for epigraphs of functions. With this in mind, for the rest of the section, we fix a closed function $f\colon\E\to\R\cup\{\infty\}$, a set $\cM\subset \dom f$, and a point $\bar x \in \cM$.

\begin{definition}[Condition $(b_{\diamond})$ for functions]
{\rm
	 We say that $f$ is {\em $(a)$-regular along $\cM$ near $\bar x$} if the epigraph of $f$ is  $(a)$-regular along $\gph f\big|_{\cM}$ near $(\bar x,f(\bar x))$. Conditions $(b_{\diamond})$, strong $(a)$, and strong $(b_{\diamond})$ are defined similarly.
}
\end{definition}

Our immediate goal is to interpret regularity of a function $f$ along $\cM$ in purely analytic terms. We begin with conditions $(b_{\diamond})$ and  strong $(b_{\diamond})$. To this end, we will need the following simple lemma.

\begin{lemma}[Regularity of the domain]\label{lem:reg_domain(b)}
Suppose that $f$ is locally Lipschitz continuous on its domain.
If  $f$ is $(b_{\diamond})$-regular along $\cM$ near $\bar x$, then the domain of $f$ is $(b_{\diamond})$-regular along $\cM$ near $\bar x$. Analogous statements hold for strong $(b_{\diamond})$.
\end{lemma}
\begin{proof}
	Suppose that $f$ is $(b_{\diamond})$-regular along $\cM$ near $\bar x$. For any $x\in \dom f$ and $y\in \cM$ set $X=(x,f(x))$ and $Y=(y,f(y))$. Then for any unit vector $v\in N_{\dom f}(x)$, the vector $V=(v,0)$ satisfies the inclusion $V\in N_{\epi f}(X)$ and therefore we may write
	$$\left\langle v, \frac{y-x}{\|y-x\|}\right\rangle=\left\langle V, \frac{Y- X}{\|Y-X\|}\right\rangle\cdot \frac{\|Y-X\|}{\|y-x\|}.$$
	Using $(b_{\diamond})$-regularity of $\epi f$ along $\cY$ and local Lipschitz continuity of $f$ on its domain immediately guarantees that $\dom f$ is $(b_{\diamond})$-regular along $\cM$ near $\bar x$. The analogous statement for strong $(b_{\diamond})$ follows from the same argument.
\end{proof}

The following result interprets $(b_{\diamond})$-regularity of a function in purely analytic terms.

\begin{thm}[From geometry to analysis]\label{thm:reinterp0}
	Suppose that $f$ is locally Lipschitz continuous on its domain. Then the following are true.
	\begin{enumerate}	
		\item~ {\bf (condition $(b)$)}  $f$ is $(b_{\leq})$-regular along $\cM$ near $\bar x$ if and only if there exists $\epsilon>0$ such that the estimates 
		\begin{align}
			\frac{f(x)+\langle v, y-x\rangle-f(y)}{\sqrt{1+\|v\|^2}}&\leq o( \|y-x\|),\label{eqn:bd1}\\
			\left\langle \frac{w}{\|w\|},y-x\right\rangle &\leq  o(\|y-x\|) ,\label{eqn:bd2}
		\end{align}
		hold for all $x\in \dom f\cap B_{\epsilon}(\bar x)$, $y\in \cM\cap B_{\epsilon}(\bar x)$,  $v\in \partial f(x)$, and $w\in \partial^{\infty} f(x)$.
		
		\item~ {\bf (strong $(b)$)}  $f$ is strongly $(b_{\leq})$-regular along $\cM$ near $\bar x$ if and only if there exists a constant $\epsilon0$ such that the estimate holds: 
		\begin{align}
			\frac{f(x)+\langle v, y-x\rangle-f(y)}{\sqrt{1+\|v\|^2}}&\leq O( \|y-x\|^2),\label{eqn:bd12}\\
				\left\langle \frac{w}{\|w\|},y-x\right\rangle &\leq O(\|y-x\|^2) ,\label{eqn:bd22}
		\end{align}
		hold for all $x\in \dom f\cap B_{\epsilon}(\bar x)$, $y\in \cM\cap B_{\epsilon}(\bar x)$, $v\in \partial f(x)$, and $w\in \partial^{\infty} f(x)$.
	\end{enumerate}
Analogous equivalences hold for $(b_{=})$ and $(b_{\geq})$, along with their strong variants, by replacing the inequalities in  \eqref{eqn:bd1}-\eqref{eqn:bd22} by $=$ and $\leq$, respectively.
\end{thm}
\begin{proof}
	Throughout the proof, set $\cX=\epi f$ and $\cY:=\gph f\big|_{\cM}$. We will use capital letters  $X$, $Y$, and $\bar X$ to denote the lifted points $(x,f(x))$, $(y,f(y))$, and $(\bar x,f(\bar x))$, respectively. We will use the relationship for any point $x\in \dom f$
	\cite[Theorem 8.9]{rockafellar2009variational}:
	\begin{equation}\label{eqn:equiv_norm_epi_notused}
		N_{\cX}(X) \textrm{ coincides with the union of}\quad\R_{++}(\partial f(x)\times \{-1\})\quad \textrm{and}\quad \partial^{\infty} f(x)\times\{0\}.
	\end{equation}
	By definition, $f$ is $(b_{\leq})$-regular along $\cM$ near $\bar x$ if and only if  for any the estimate	
	\begin{equation}\label{eq:bbb}
		\langle V,Y-(x,r)\rangle \leq   o(\|(x,r)-Y\|)\cdot\|V\|,
	\end{equation}
	holds for all  $x\in \dom f$ and $y\in \cM$ with $(x,r)$ and $Y$ sufficiently close to $\bar X$, and for all $V\in N_{\cX}(x,r)$.  Let us look at the two cases $r=f(x)$ and $r>f(x)$. In the former case $r=f(x)$, condition \eqref{eq:bbb} is formally equivalent to the two conditions \eqref{eqn:bd1} and \eqref{eqn:bd2}.
 In the latter case $r>f(x)$, the expression \eqref{eq:bbb}  becomes
	$$\langle w,y-x\rangle\leq o(\|(x,r)-Y\|)\cdot\|w\|$$
	for all $w\in N_{\dom f}(x)$.
	Clearly, this is implied by $(b_{\leq})$ regularity of $\dom f$ along $\cM$ near $\bar x$. The claimed equivalence for $(b_{\leq})$-regularity now follows immediately from Lemma~\ref{lem:reg_domain(b)}. The rest of the equivalence follow from an analogous argument.
\end{proof}

The conditions in Theorem~\ref{thm:reinterp0} are particularly transparent when $f$ is Lipschitz continuous near $\bar x$. Then  $\partial^{\infty} f(\bar x)$ consists only of the zero vector and $\partial f(x)$ is nonempty and uniformly bounded near $\bar x$. Therefore, conditions $(b_{\leq})$  and strong  $(b_{\leq})$, respectively, are equivalent to the two properties
\begin{align*}
f(y) &\geq f(x)+\langle v, y-x\rangle+o(\|y-x\|)\\
f(y) &\geq f(x)+\langle v, y-x\rangle+O(\|y-x\|^2)
\end{align*}
as $x$ and $y\in \cM$ tend to $\bar x$ and $v\in \partial f(x)$ is arbitrary.
In words, condition $(b_{\leq})$ ensures a restricted lower Taylor approximation property 
as $x$ and $y\in \cM$ tend to $\bar x$ and $v\in \partial f(x)$ are arbitrary.  Strong $(b)$-regularity, in turn, replaces the little-o term with the squared norm $O(\|x-y\|^2)$. In particular, this holds automatically if $f$ is weakly convex. When $\cM=\{\bar x\}$ is a single point, condition $(b_=)$ reduces to generalized differentiability in the sense of Norkin \cite{norkin1980generalized} and is closely related to the semismoothness property of Mifflin \cite{mifflin1977semismooth}.

Condition $(b_{\leq})$ becomes particularly useful algorithmically when the inclusion $0\in \hat\partial f(\bar x)$ holds and  $\cM$ is a $C^1$ active manifold of $f$ around $\bar{x}$.  Indeed, condition $(b_{\leq})$ along with the sharp growth guarantee of Theorem~\ref{prop: sharpness} then imply that there exists a constant $\mu>0$ such that the estimate 
\begin{equation}\label{eqn:angle}
\langle v,x-P_{\cM}(x)\rangle\geq \mu\cdot \dist(x,\cM),
\end{equation}
holds for all  $x\in \dom f$  near $\bar x$ and for all $v\in \partial f(x)$. In words, this means that negative subgradients of $f$ at $x$ always point towards the active manifold. The angle condition \eqref{eqn:angle} together with strong $(a)$ regularity will form the core of the algorithmic developments. 
For ease of reference, we record a slight generalization of the angle condition \eqref{eqn:angle}  when $f$ is not necessarily locally Lipschitz around $\bar x$ and can even be infinite-valued.

\begin{cor}[Proximal aiming]\label{cor:prox-aiming_gen}
	Consider a closed function $f\colon\E\to\R\cup\{\infty\}$ that admits an active $C^1$-manifold $\cM$ at a point $\bar x$ satisfying $0\in \hat\partial f(\bar x)$. Suppose that $f$ is locally Lipschitz continuous on its domain and that $f$ is $(b_{\leq})$-regular along $\cM$ near $\bar x$. Then, there exists a constant $\mu>0$ such that the estimate
	\begin{equation}\label{eqn:prox_aim_super}
	\langle v,x-P_{\cM}(x)\rangle\geq \mu \cdot\dist(x,\cM) -\sqrt{1+\|v\|^2}\cdot o(\dist(x,\cM)),
	\end{equation}
	holds for all $x \in \dom f$ near $\bar x$ and for all $v\in \partial f(x)$. Moreover,  if $f$ is locally Lipschitz around $\bar x$, the same statement holds with $\partial f(x)$ replaced by $\partial_c f(x)$ and with the negative term omitted in \eqref{eqn:prox_aim_super}.\footnote{The last claim follows immediately from \eqref{eqn:prox_aim_super} by possibly increasing $\mu>0$ and taking convex combinations of limiting subgradients, all of which are uniformly bounded.}
\end{cor}

Next, we move on to interpreting conditions $(a)$ and strong $(a)$ in analytic terms. We will focus on the most interesting setting when $\cM$ is a smooth manifold and the restriction of $f$ to $\cM$ is smooth near $\bar x$.  In particular, we will make use of the following observation in our arguments:  the tangent space to $\cY:=\gph f\big|_{\cM}$ at $Y:=(y,f(y))$ is:
\begin{equation}\label{eqn:simple_tangent}
	T_{\cY}(Y)=\{(u,\langle \nabla_{\cM} f(y),u\rangle):u\in T_{\cM}(y)\}.
\end{equation}

\begin{lemma}[Regularity of the domain]\label{lem:reg_domain(a)}
Suppose that $f$ is locally Lipschitz continuous on its domain, $\cM$ is a $C^1$ manifold around $\bar x$, and the restriction of $f$ to $\cM$ is $C^1$-smooth near $\bar x$.
If  $f$ is $(a)$-regular along $\cM$ near $\bar x$, then the domain of $f$ is $(a)$-regular along $\cM$ near $\bar x$. Analogous statement holds for strong $(a)$-regularity.
\end{lemma}

\begin{proof}
	Throughout the proof, set $\cY=\gph f\big|_{\cM}$. Suppose first that $f$ is $(a)$-regular along $\cM$ near $\bar x$. Note the inclusion $N_{\dom f}(y)\times \{0\}\subseteq N_{\epi f}(y, f(y))$ for all $y$ near $\bar x$. 
	Using Lemma~\ref{lemma:inclusion_normal}, we therefore conclude $N_{\dom f}(y)\times \{0\}\subseteq N_{\cY}(y,f(y))$. The desired inclusion $N_{\dom f}(y)\subset N_{\cM}(y)$ now follows immediately from \eqref{eqn:simple_tangent}.

	Finally, suppose that $f$ is strongly $(a)$-regular along $\cM$ near $\bar x$. Fix points $x\in \dom f$ and $y\in \cM$ near $\bar x$ and as before define $X=(x,f(x))$ and $Y=(y,f(y))$.	Then condition (a) implies that there exists a constant $C>0$ such that for any $v\in N_{\dom f}(\bar x)$ there is a vector $(w_1,w_2)\in N_{\cY}(Y)$ satisfying $\|(v,0)-(w_1,w_2)\|\leq C\|X-Y\|$.
	It follows easily from the description  \eqref{eqn:simple_tangent} that the inclusion $w_1+w_2\nabla_{\cM} f(y)\in N_{\cM}(y)$ holds, and therefore 
$$\dist(v, N_{\cM}(y))\leq \|v-w_1-w_2\nabla_{\cM} f(y)\|\leq C(1+\|\nabla_{\cM} f(y)\|)\|X-Y\|.$$ 
Since $f$ is locally Lipschitz continuous on its domain, there exists $C'>0$ satisfying   $\|\nabla_{\cM} f(y)\|\leq C'$ and $\|X-Y\|\leq C'\|x-y\|$ for all $x\in \dom f$ and $y\in \cM$ near $\bar x$. Thus $\dom f$ is strongly $(a)$-regular along $\cM$ at $\bar x$, as claimed.
\end{proof}

The following theorem reinterprets conditions conditions $(a)$ and strong $(a)$ in entirely analytic terms.

\begin{thm}[From geometry to analysis]\label{thm:reinterp}
Suppose that $f$ is locally Lipschitz continuous on its domain, $\cM$ is a $C^1$ manifold around $\bar x$, and the restriction of $f$ to $\cM$ is $C^1$-smooth near $\bar x$.
The following claims are true.
	\begin{enumerate}	
		\item~ {\bf (condition $(a)$)} $f$ is $(a)$-regular along $\cM$ near $\bar x$ if and only if the inclusions hold:
		\begin{equation}\label{eqn:a_getit}
		P_{T_{\cM}(x)}(\partial f(x))\subseteq \{\nabla_{ \cM} f(x)\}\qquad \textrm{and}\qquad \partial^{\infty} f(x)\subseteq N_{\cM}(x).
		\end{equation}
	for all $x\in \cM$ near $\bar x$.
		
		\item~ {\bf (strong $(a)$)} $f$ is strongly $(a)$-regular along $\cM$ near $\bar x$ if and only if there exist constants $C,\epsilon>0$ satisfying: 
		\begin{align}
		\|P_{T_{\cM}(y)}(v-\nabla_{\cM} f(y))\|&\leq C\sqrt{1+\|v\|^2} \|x-y\|,\label{eqn:str_a1}\\
		\|P_{T_{\cM}(y)}(w)\|&\leq C \|w\|\cdot\|x-y\|,\label{eqn:str_a2}
		\end{align}
		 for all $x\in \dom f\cap B_{\epsilon}(\bar x)$ and $y\in \cM\cap B_{\epsilon}(\bar x)$, $v\in \partial f(x)$, and $w\in \partial^{\infty} f(x)$.
	\end{enumerate}
\end{thm}
\begin{proof}
	The proof is similar to that of Theorem~\ref{thm:reinterp0}. Throughout, set $\cX=\epi f$ and $\cY:=\gph f\big|_{\cM}$. We will use capital letters  $X$, $Y$, and $\bar X$ to denote the lifted points $(x,f(x))$, $(y,f(y))$, and $(\bar x,f(\bar x))$, respectively. We also recall the relationship for any point $x\in \dom f$
 \cite[Theorem 8.9]{rockafellar2009variational}:
	\begin{equation}\label{eqn:equiv_norm_epi}
 N_{\cX}(X) \textrm{ coincides with the union of}\quad\R_{++}(\partial f(x)\times \{-1\})\quad \textrm{and}\quad \partial^{\infty} f(x)\times\{0\}.
 \end{equation}
 	Lemma~\ref{lemma:inclusion_normal} implies that $f$ is $(a)$-regular along $\cM$ near $\bar x$ if and only if the inclusion $N_{\cX}(X)\subset N_{\cY}(X)$ holds for all $x$ near $\bar x$, or equivalently  $\langle N_{\cX}(X), V\rangle =\{0\}$ for all $V\in T_{\cY}(X)$.
	 In light of \eqref{eqn:simple_tangent} and \eqref{eqn:equiv_norm_epi}, this happens if and only if 
	$$\langle \partial f(x)-\nabla_\cM f(x),u\rangle\subseteq\{0\}\qquad \textrm{and}\qquad \langle \partial^{\infty} f(x),u \rangle=\{0\}\qquad \forall u\in T_{\cM}(x),
	$$
	which is clearly  equivalent to \eqref{eqn:a_getit}.

		Next, by definition  $f$ is strongly $(a)$-regular along $\cM$ near $\bar x$ if and only if 
		there exists a constant $C$ such that
		\begin{equation}\label{eqn:key_inner_prod}
		\langle U,V \rangle\leq C\|U\|\cdot\|V\|\cdot\|(x,r)-Y\|
		\end{equation}
		for all  $(x,r)\in\cX$ and $Y\in \cY$ sufficiently close to $\bar X$, and for all  $U\in N_{\cX}((x,r))$ and $V\in T_{\cY}( Y)$.
		Let us interpret \eqref{eqn:key_inner_prod} in two cases,  $r=f(x)$ and $r> f(x)$. In the former case $r=f(x)$, in light of \eqref{eqn:equiv_norm_epi} and local Lipschitz continuity of $f$ on its domain,  condition \eqref{eqn:key_inner_prod} simplifies to
		\begin{align}
		\langle v-\nabla_{\cM} f(y),u \rangle&\leq C'  \sqrt{\|v\|^2+1}\cdot \|u\|\cdot \|x-y\|,\\
		\langle w,u\rangle&\leq C'\|w\|\cdot  \|u\| \cdot \|x-y\|.\label{eqn:sec_needit}
		\end{align}
		holding for some constant $C'$, for all $x\in \dom f$ and $y\in \cM$ sufficiently close to $\bar x$, and for all $u\in T_{\cM}(y)$, $v\in \partial f(x)$, and $w\in \partial^{\infty} f(x)$. In the case $r>f(x)$, taking into account the equality $N_{\cX}(x,r)= N_{\dom f}\times\{0\}$, we see that \eqref{eqn:key_inner_prod} reduces to 
		$$\langle w,u\rangle\leq C'\|w\|\cdot  \|u\| \cdot \sqrt{\|x-y\|^2+(r-f(y))^2}
		$$		
		holding for all $w\in N_{\dom f}(x)$.
		Clearly, this is implied by $\dom f$ being strongly $(a)$-regular along $\cM$ at $\bar x$. In particular, taking into account Lemma~\ref{lem:reg_domain(a)} we see that this condition holds automatically if  $f$ is strongly $(a)$ regular along $\cM$ at $\bar x$. The claimed equivalence for strong (a) regularity follows immediately.  	
\end{proof}

Again the conditions in Theorem~\ref{thm:reinterp} become particularly transparent when $f$ is Lipschitz continuous near $\bar x$. Then conditions $(a)$ and strong $(a)$, respectively, are equivalent to
\begin{align*}
P_{T_{\cM}(y)}(\partial f(y))&= \{\nabla_{ \cM} f(y)\}\\
\|P_{T_{\cM}(y)}(\partial f(x)-\nabla_{\cM} f(y))\|&=O(\|x-y\|)
\end{align*}
holding as $x\to\bar x$ and $y\in \cM$ tend to $\bar x$.
In words, condition $(a)$ is equivalent to the projection  $P_{T_{\cM}(y)}(\partial f(y))$ reducing to a a single point---the covariant gradient $\nabla_{ \cM} f(y)$. This type of property is called the projection formula in \cite{doi:10.1137/060670080}. Strong $(a)$ provides a ``stable improvement'' over the projection formula wherein the 
deviation $\partial f(x)-\nabla_{ \cM} f(y)$ in tangent directions $T_{\cM}(y)$ is linearly bounded by $\|x-y\|$, for points  $x\in\E$ and $y\in \cM$ near $\bar x$.

The rest of the chapter is devoted to exploring the relationship between the four basic regularity conditions, presenting examples, proving calculus rules, and justifying that these conditions hold ``generically'' along active manifolds. Section~\ref{sec:algos_main} will in turn use these conditions to analyze subgradient type algorithms.

\subsection{Relation between the four conditions}
The goal of this section is to explore the relationship between the four regularity conditions. Recall that Lemma~\ref{lem:bima} already established the implication $(b_{\leq})\Rightarrow (a)$. More generally, the goal of this section is to show in reasonable settings the string of implications:
\begin{equation}\label{eqn:main_implications}
\boxed{(a)\quad \Leftarrow \quad (b_=) \quad \Leftarrow\quad {\rm strong }\,(a) \quad \Leftarrow \quad {\rm strong }\, (b_=)}.
\end{equation}
Before passing to formal statements, we require some preparation. Namely, the task of verifying conditions $(b_{\diamond})$, strong $(a)$, and strong $(b_{\diamond})$ requires considering arbitrary points $x \in \cX$ and $y\in \cY$, which are a priori unrelated. We now show that it essentially suffices to set $y$ to be the projection of $x$ onto $\cY$, or more generally a retraction of $x$ onto $\cY$. In this way, we may remove one degree of flexibility for the question of verification. We begin by defining the projected variants of conditions $(b_{\diamond})$, strong $(a)$, and strong $(b_{\diamond})$.

We begin by defining retractions onto a set $\cY$, with the nearest point projection being the primary example. The added flexibility will be useful once we pass to functions.

\begin{definition}[Retractions]
{\rm 
A map $\pi\colon\E\to\E$ is a {\em retraction} onto a set $\cY\subset \E$ near a point $\bar x\in \cY$ if 
\begin{enumerate}
\item  the inclusion $\pi(x)\in \cX$ holds for all $x$ near $\bar x$,
\item there exists a constant $C\geq 0$ such that the inequality $\|x-\pi(x)\|\leq C\cdot \dist(x,\cY)$ holds for all $x$ near $\bar x$.
\end{enumerate}
If $\pi$ is $C^p$-smooth near $\bar x$, we call $\pi$ a {\em $C^p$-smooth retraction}. }
\end{definition}

Next, we define the projected conditions.

\begin{definition}[Projected conditions]\label{defn:project_condt}
{\rm
Fix two sets  $\cX,\cY\subset\E$, a point $\bar x\in \cY$, and a retraction $\pi$ onto $\cY$. We say that $\cX$ is {\em $(b^{\pi}_{\diamond})$-regular} along $\cY$ at $\bar x$ if it satisfies condition $(b_{\diamond})$ in the restricted setting $y_i=\pi(x_i)$. 
 Conditions  {\em strong} $(a^{\pi})$ and {\em strong $(b^{\pi}_{\diamond})$} are defined analogously.
		}
\end{definition}

The following theorem allows one to reduce the question of verifying regularity conditions to the setting $y\in \pi(x)$. 

\begin{thm}\label{thm:atob}
Fix two sets $\cX,\cY\subset\E$, a point $\bar x\in \cY$, and a $C^1$-smooth retraction $\pi$ onto $\cY$. Suppose moreover that $\cY$ is a $C^1$-smooth manifold near $\bar x$. Then the equivalences hold:
\begin{enumerate}
\item strong $(a)~\Leftrightarrow~\textrm{strong }(a^{\pi})$
\item  $(a)$ and $(b^{\pi}_{\diamond})$$~\Leftrightarrow~$ $(b_{\diamond})$
\end{enumerate} 
Moreover, if $\pi$ is $C^2$-smooth, then the implication holds:
 $$\textrm{strong }(a^{\pi}) \textrm{ and strong }(b^{\pi}_{\diamond})~\quad\Rightarrow\quad~ \textrm{strong }(b_{\diamond}).$$
\end{thm}
\begin{proof}
		
	Suppose that $\cX$ is strongly $(a^\pi)$-regular regular along $\cY$ near $\bar x$. Thus there exists a constant $C_1>0$ such that
\begin{equation}\label{eqn:1del}
 		\Delta \left(N_{\cX}(x), N_{\cY}(\pi(x))\right) \le  C_1 \|x-\pi(x)\|,
		\end{equation}
 		for all $x\in \cX$ sufficiently close to $\bar x$. On the other hand, since $\pi$ is a retraction onto $\cY$, there exists some constant $C'>0$ satisfying $\|x-\pi(x)\|\leq C'\cdot \|x-y\|$ for all $\in\cX$ and $y\in\cY$ near $\bar x$. Moreover, since $\cY$ is a $C^1$-smooth manifold, there exists a constant $C_2>0$ such that
		\begin{equation}\label{eqn:2del}
 		\begin{aligned}
 		\Delta\left(N_{\cY}(\pi(x)), N_{\cY}(y)\right) &\le C_2 \|\pi(x) - y\|\\
 		&\le C_2\left(\|\pi(x) - x\| + \|x-y\|\right)\\
 		&\le (1+C') C_2 \|x-y\|.
 		\end{aligned}
		\end{equation}
 	Combining \eqref{eqn:1del} and \eqref{eqn:2del}, and using the triangle inequality, we conclude
 	$\Delta(N_{\cX}(x), N_{\cY}(y))\le 
 	 (C_1C' +(1+C')C_2)\norm{x-y}$,
  for all $x \in \cX, y \in \cY$ sufficiently close to $\bar x$. Thus $\cX$ is strongly $(a)$-regular along $\cY$ at $\bar x$ as claimed.

 		Next, suppose that $\cX$ is  both $(a)$ and $(b^\pi_{\diamond})$ regular along $\cY$ near $\bar x$. Let $x_i \in \cX$ and $y_i \in \cY$ be sequences converging to some point $y$ near $\bar x$ and let $v_i\in N_{\cX}(x_i)$ be arbitrary. Let us write 		 \begin{equation}\label{eqn:decomp_er}
 			\dotp{v_i, \frac{y_i-x_i}{\|y_i-x_i\|}} =  \dotp{v_i, \frac{\pi(x_i)-x_i}{\|y_i-x_i\|}} + \dotp{v_i, \frac{y_i-\pi(x_i)}{\|y_i-x_i\|}}.
 		\end{equation}
		We analyze each term on the right side separately. To this end, observe
		 $$\left\langle v_i, \frac{\pi(x_i)-x_i}{\|y_i-x_i\|}\right\rangle=\left\langle v_i, \frac{\pi(x_i)-x_i}{\|\pi(x_i)-x_i\|}\right\rangle\cdot\frac{\|\pi(x_i)-x_i\|}{\|y_i-x_i\|}.$$
		Therefore, the accumulation points of $\left\langle v_i, \frac{\pi(x_i)-x_i}{\|y_i-x_i\|}\right\rangle$ inherit the sign of the accumulation points of $\left\langle v_i, \frac{\pi(x_i)-x_i}{\|\pi(x_i)-x_i\|}\right\rangle$.

	Next, moving on since the retraction $\pi$ is $C^1$-smooth near $\bar x$, we deduce 
\begin{equation}\label{eqn:tempo}
\limsup_{i\to\infty}	\left|\dotp{v_i, \frac{y_i-\pi(x_i)}{\|y_i-x_i\|}}\right|\leq \limsup_{i\to\infty}	\left|\dotp{v_i, \frac{\nabla \pi(x_i)(y_i-x_i)}{\|y_i-x_i\|}}\right|.
\end{equation}
Passing to a subsequence, we may assume $\frac{y_i-x_i}{\|y_i-x_i\|}$ tends to some vector $w\in \E$ and that $v_i$ converge to some vector $v$. Observe that since $\pi$ maps points into $\cY$, the range of $\nabla \pi(y)$ is contained in the tangent space $T_{\cY}(y)$. Noting  that condition $(a)$ guarantees $v\in N_{\cY}(y)$, we deduce that the right-side of \eqref{eqn:tempo} is zero. Thus condition $(b_{\diamond})$ holds.

Next, suppose that $\pi$ is $C^2$-smooth and that $\cX$ is both strongly $(a^{\pi})$-regular and strongly $(b^\pi_{\diamond})$-regular along $\cY$ near $\bar x$. Note that we already proved that strong $(a^{\pi})$ implies strong $(a)$.
We return to the decomposition:
 \begin{equation}
 			\dotp{v_i, \frac{y_i-x_i}{\|y_i-x_i\|^2}} =  \dotp{v_i, \frac{\pi(x_i)-x_i}{\|y_i-x_i\|^2}} + \dotp{v_i, \frac{y_i-\pi(x_i)}{\|y_i-x_i\|^2}}.
 		\end{equation}
and analyze each term separately. To this end, we may write
$$\dotp{v_i, \frac{\pi(x_i)-x_i}{\|y_i-x_i\|^2}}=\dotp{v_i, \frac{\pi(x_i)-x_i}{\|\pi(x_i)-x_i\|^2}}\frac{\|\pi(x_i)-x_i\|^2}{\|y_i-x_i\|^2}.$$
Therefore, the accumulation points of $\dotp{v_i, \frac{\pi(x_i)-x_i}{\|y_i-x_i\|^2}}$ inherit the sign of the accumulation points of $\dotp{v_i, \frac{\pi(x_i)-x_i}{\|\pi(x_i)-x_i\|^2}}$.
Next, since $\pi$ is $C^2$ smooth, we compute
\begin{equation}\label{eqn:tempo-randon}
\limsup_{i\to\infty}	\left|\dotp{v_i, \frac{y_i-\pi(x_i)}{\|y_i-x_i\|^2}}\right|\leq \limsup_{i\to\infty}	\frac{1}{\|y_i-x_i\|}\cdot\left|\dotp{v_i, \frac{\nabla \pi(y_i)(y_i-x_i)}{\|y_i-x_i\|}}\right|.
\end{equation}
Since $w_i:=\frac{\nabla \pi(y_i)(y_i-x_i)}{\|y_i-x_i\|}$ is tangent to $\cY$ at $y_i$, strong $(a)$ regularity implies that the right side of \eqref{eqn:tempo-randon} is finite.
 We thus conclude that $\cX$ is strongly $(b_{\diamond})$ regular along $\cY$ near $\bar x$, as claimed.
  \end{proof}

With Theorem~\ref{thm:atob} at hand, we may now establish the remaining implications in \eqref{eqn:main_implications}, beginning with  strong $(b_{\geq})$ implies strong $(a)$.

\begin{proposition}[Strong $(b_{\geq})$ implies strong $(a)$]\label{pro:getirdone}
Consider a $C^3$ manifold $\cY$ that is contained in a set $\cX\subset \E$. Suppose that $\cX$ is prox-regular at  a point $\bar x\in \cY$.
Then the following implication holds:
	$$\textrm{\rm strong } (b_{\geq})~~\Rightarrow ~~\textrm{\rm strong } (a).$$
\end{proposition}
\begin{proof}
Suppose that $\cX$ is strongly $(b_{\geq})$-regular along $\cY$ near $\bar x$.
In light of Theorem~\ref{thm:atob}, it suffices to prove that the strong $(a^\pi)$ condition holds for $C^2$-smooth retraction. We will use the projection $\pi:=P_{\cY}$, which is indeed a $C^2$-smooth retraction onto $\cY$ since   $\cY$ is a $C^3$ manifold. Thus, there exist constants $\epsilon,L>0$ satisfying
	\begin{equation}\label{eqn:proj_lip_main}
	\|P_{\cY}(y+h)-P_{\cY}(y)-\nabla P_{\cY}(y)h\|\leq L\|h\|^2,
	\end{equation}
	for all $y\in B_{\epsilon}(\bar x)$ and $h\in \epsilon  {\bf B}$. 	Fix now two points $x \in \cX$ and $y \in \cY$ and a unit vector  $v \in N_{\cX}(x)$. Clearly, we may suppose $v\notin N_{\cY}(y)$, since otherwise the claim is trivially true. Define the normalized vector $w := -\frac{P_{\tangent{\cY}{y}}( v)}{\|P_{\tangent{\cY}{y}}( v)\|}$. 
Noting the equality $\nabla P_{\cY}(y)=P_{T_{\cY}(y)}$ and appealing to \eqref{eqn:proj_lip_main}, we deduce the estimate
	\begin{align*}
	\|P_{\cY}(y - \alpha w) -  (y - \alpha w)\| \leq L\| \alpha w\|^2 = L\alpha^2,
	\end{align*}
	for all $y\in B_{\epsilon}(\bar x)$ and $\alpha\in (0,\epsilon)$.
 	Shrinking $\epsilon>0$, prox-regularity yields the estimate
	\begin{align*}
	\dotp{v, P_{\cY}(y - \alpha w)-x} \leq \frac{\rho}{2}\|x-P_{\cY}(y - \alpha w)\|^2,
	\end{align*}
	for some constant $\rho>0$. Therefore, we conclude 
	\begin{align*}
	\alpha\|P_{\tangent{\cY}{y}}v\| = -\alpha\dotp{v, w}  
	&= \dotp{v,  x-y} + \dotp{v, P_{\cY}(y - \alpha w)-x} + \dotp{v,  (y - \alpha w)- P_{\cY}(y - \alpha w)} \\
	&\leq C\|x - y\|^2 + \frac{\rho}{2}\|x-P_{\cY}(x - \alpha w)\|^2 + L\alpha^2,
	\end{align*}
	where the last inequality follows from the strong $(b_{\geq})$ condition.
	Note that the middle term is small: 
	$$
	\| P_{\cY}(y - \alpha w)-x\|^2 \leq 2\|P_{\cY}(y - \alpha w) - (y - \alpha w) \|^2 + 2\|y - \alpha w - x\|^2 \leq 2L^2\alpha^4 + 4\|y - x\|^2 + 4\alpha^2.
	$$
Thus, we have 
	\begin{align*}
	\alpha \|P_{\tangent{\cY}{y}}v\| &\leq  C\|x - y\|^2 + \rho L^2\alpha^4 + 2\rho \|x - y\|^2 + 2\rho\alpha^2 + L\alpha^2.
	\end{align*}
	Dividing both sides by $\alpha$ and setting $\alpha = \|x - y\|$ completes the proof.
\end{proof}

Next we prove the last implication, strong $(a)$ $\Rightarrow$ $(b_{=})$, in the definable category. This result thus generalizes the theorems of Kuo \cite{kuo1971ratio}, Verdier \cite{verdier1976stratifications}, and Ta Le Loi \cite{le1998verdier}. The proof technique we present is different from those in the earlier works on the subject and will be based on an application of the Kurdyka-{\L}ojasiewicz inequality \cite{doi:10.1137/060670080}.

\begin{thm}[Strong $(a)$ implies $(b)$]
Fix two definable sets $\cX,\cY\subset\E$ and a point $\bar x\in \cY$. Suppose in addition that $\cY$ is a $C^2$-smooth manifold around $\bar x$ and that $\cX$ is a locally closed set. Then the following implication holds: 
$$\textrm{\rm strong }(a)~~\Rightarrow~~(b_{=}).$$
\end{thm}

We note that the theorem may easily fail for general $C^{\infty}$-manifolds $\cX$ and $\cY$, without some extra ``tameness'' assumption such as definability. See the discussion in \cite{le1998verdier} for details. 

\begin{proof}

 		Suppose that $\cX$ is strongly $(a)$-regular along $\cY$  near $\bar x$. In light of Theorem~\ref{thm:atob}, it suffices to show that $\cX$ is $(b^{\pi})$-regular along $\cY$ near $\bar x$. 
To this end, define the function 
		$$
 		g(x, v) = |\dotp{v, x-\proj_{\cY}(x)}|+\delta_{\cl \cX}(x).
 		$$
		Fix a compact neighborhood $U$ of $\{\bar x\}\times \mathbb{B}$. Then the KL-inequality \cite[Theorem 11]{doi:10.1137/060670080} ensures that there exists $\eta>0$ and a continuous function $\psi\colon[0,\eta)\to \R$ satisfying $\psi(0)=0$ and $\psi'(0)=0$ such that 		
 		\begin{align}\label{eq:semismoothwhatIneed}
 			g(x, v) \leq \psi(\dist(0, \partial g(x, v)))).
 		\end{align}
		for any $(x,v)\in U$ with $g(x,v)\leq \eta$. It suffices now to show that $\dist(0,\partial g(x,v))$ is linearly upper bounded by $\dist(x,\cY)$ for all $x\in \cX$ near $\bar x$ and all unit vectors $v\in N_{\cX}(x)$.
	To this end, fix  any point $(x,v)$. Clearly, we may assume  $g(x,v)\neq 0$, since otherwise there is nothing to prove. 
We compute 
		$$\partial g(x,v)=\left\{(I-\nabla \proj_{Y}(x))v+N_{\cl \cX}(x)\right\}\times \{x-P_{\cY}(x)\}.$$
		Therefore as long as $v\in N_{\cX}(x)$ we have
		\begin{equation}\label{eqn:intermediate}
		 \dist(0,\partial g(x,v))\leq \|\nabla P_{\cY}(x)v\|+\dist(x,\cY).
		 \end{equation}
 		Since $\cY$ is a $C^2$-manifold near $\bar x$, there exists a constant $L>0$ such that the inequality $\|\nabla P_{\cY}(x)\|\leq L$ holds for all $x$ near $\bar x$. Further, let $C>0$ be the constant from the defining property \eqref{eqn:strong_a} of strong $(a)$ regularity. Thus, as long as $x\in \cX$ is sufficiently close to $\bar x$, there exists a vector $w\in N_{\cY}(P_{\cY}(x))$ satisfying 
$\|v-w\|\leq C\dist(x,\cY)$. Therefore, continuing with~\eqref{eqn:intermediate} we deduce 
$$\dist(0,\partial g(x,v))\leq \|\nabla P_{\cY}(x)w\|+(1+CL)\dist(x,\cY).$$	
 To complete the proof, note that $\nabla P_{\cY}(x)w=0$ since $\range(\nabla P_{\cY}(x)) \subseteq T_{\cY}(P_{\cY}(x))$. 
  	 	\end{proof}

\subsection{Basic examples}
Having a clear understanding of how the four regularity conditions are related, we now present a few interesting examples of sets that are regular along a distinguished submanifold. More interesting examples can be constructed with the help of calculus rule, discussed at the end of the section.
We begin with the following simple example showing that any convex cone is regular along its lineality space. 
\begin{proposition}[Cones along the lineality space]\label{prop:CONES}
	Let $\cX\subset \E$ be a convex cone and let $\cY=\Lin(X)$ denote its lineality space. Then  $\cX$ is both strongly $(a)$ and strongly $(b_{=})$ regular along $\cY$.
\end{proposition}
\begin{proof}
Strong $(a)$ regularity follows from the inclusion $N_{\cX}(x)\subset N_{\cY}(y)$ holding for all $x\in \cX$ and $y\in \cY$.  Next, fix any points $x\in \cX$ and $y\in \cY$ and a vector $v\in N_{\cX}(x)$. Strong $(b_{=})$ regularity follows from the equality $\langle v,x-y\rangle=0$, which is straightforward to verify. 
\end{proof}

More interesting examples may be constructed as diffeomorphic images of cones around points in the lineality space. Following \cite{shapiroreducible}, a set $\cX\subset\E$ is said to be {\em $C^p$-cone reducible around a point $\bar x\in\cX$} if there exist a closed convex cone $K$ in some Euclidean space $\bf Y$, open neighborhoods $U$ of $\bar x$ and $V$ of the origin in  $\bf Y$,  and a diffeomorphism $\varphi\colon U\to V$ satisfying  $\varphi(\bar x)=0$ and $\cX\cap U=\varphi^{-1}(K\cap V)$. In this case, it follows from \cite[Theorem 4.2]{lewis2002active} that the set $\cM=\varphi^{-1}(\Lin(K)\cap V )$ is an active manifold for $\cX$ at $\bar x$ for any $v\in\ri N_{\cX}(\bar x)$. Common examples of sets that are cone reducible around each of their points are polyhedral sets, the cone of positive semidefinite matrices, the Lorentz cone, and any set cut out by smooth nonlinear inequalities with linearly independent gradients.
It is straightforward to see that conditions $(a)$ and $(b_{\diamond})$ are preserved under $C^1$ diffeomorphisms, while strong $(a)$ and strong  $(b_{\diamond})$ are preserved under $C^2$ diffeomorphisms.
  The following is therefore an immediate consequence of Proposition~\ref{prop:CONES}.

\begin{cor}[Cone reducible sets are regular along the active manifold]\label{cor:cone_red}
Suppose that a set $\cX$ is $C^2$ cone reducible to $K$ by $\varphi\colon U\to V$ around $\bar x$. Then $\cX$ is  strongly $(a)$ and strongly $(b_=)$-regular along $\varphi^{-1}(\Lin(K)\cap V)$ near $\bar x$.
\end{cor}

The next proposition shows that any convex set is strongly (a)-regular along any affine space contained in it.
\begin{proposition}[Affine subsets of convex sets]\label{prop:concl_affine_sub}
Consider a convex set $\cX\subset\E$ and a subset $\cY\subset \cX$ that is locally affine around a point $\bar x\in \cY$. Then $\cX$ is strongly $(a)$-regular along $\cY$ near $\bar x$.
\end{proposition}
\begin{proof}
Translating the sets we may suppose $\bar x=0$ and therefore that $\cY$ coincides with a linear subspace near the origin.
Fix now points $x \in \cX$ and $y \in \cY$ and a unit vector  $v \in N_{\cX}(x)$. Clearly, we may suppose $v\notin N_{\cY}(y)$, since otherwise the claim is trivially true. Define the normalized vector $w := -\frac{P_{\cY}( v)}{\|P_{\cY}( v)\|}$. The for all $y\in \cY$ near $\bar x$ and all small $\alpha>0$, using the linearity of the projection $P_{\cY}$ we compute
	\begin{align*}
	\alpha\|P_{\tangent{\cY}{y}}v\|=\alpha\|P_{\cY}v\| =- \alpha\dotp{v, w}  
	&= \dotp{v,  x-y} + \dotp{v, P_{\cY}(y - \alpha w)-x} \leq \|x - y\|,
	\end{align*}
	where the last inequality follows from convexity of $\cX$. This completes the proof.
\end{proof}

Not surprisingly, the conclusion of Theorem~\ref{prop:concl_affine_sub} can easily fail if $\cX$ is prox-regular (instead of convex) or if $\cY$ is a smooth manifold (instead of affine). This is the content of the following example.

\begin{example}[Failure of strong (a)-regularity]
{\rm
Define $\cX$ to be the epigraph of the function  $f(x,y)= \max\{0,y-x^2\}$ and set $\cY$ to be the $x$-axis $Y=\R\times\{0\}\times \{0\}$.  Consider the sequence $y_k=(1/k,0,0))$ in $\cY$ and $x_k=(1/k,1/k^2,0)$ in $\cX$ converging to the origin. Fix the sequence of normal vectors $v_k=(-2/k,1,-1)\in N_X(x_k)$ and note $N_{\cY}(y_k)=\{0\}\times \R\times\R$. A quick computation shows
$$\Delta\left(\frac{v_k}{\|v_k\|},N_{\cY}(y_k)\right)=\frac{2/k}{\sqrt{2+4/k^2}}\geq \frac{2}{k\sqrt{6}}= \frac{2}{\sqrt{6}}\sqrt{\|x_k-y_k\|}.$$ Therefore $\cX$ is not strongly $(a)$-regular along $\cY$ near $\bar x$.
}
\end{example}

Strong $(a)$-regularity fails in the above example ``by a square root factor in the distance to $\cY$.'' The following theorem shows a surprising fact: the estimate \eqref{eqn:strong_a} is guaranteed to hold up to a square root for any prox-regular set along a smooth submanifold. Since we will not use this result and the proof is very similar to that of Proposition~\ref{pro:getirdone}, we have placed the argument in the appendix.

\begin{proposition}[Strong $(a)$ up to square root]\label{prop:square_root}
	Consider a $C^3$ manifold $\cY$ that is contained in a set $\cX\subset \E$. Suppose that $\cX$ is prox-regular around a point $\bar x\in \cY$. Then there exists a constant $C>0$ satisfying
\begin{equation}\label{eqn:sqrt_estimate}
	\Delta(N_{\cX}(x),N_{\cY}(y))\leq C\cdot\sqrt{\|x-y\|},
	\end{equation}
	for all $x\in \cX$ and $y\in \cY$ sufficiently close to $\bar x$. \end{proposition}

The following example connects $(b_=)$-regularity to inner-semicontinuity of the normal cone map. Recall that a set-valued map $F\colon\E\rightrightarrows\Y$ is an assignment of points $x\in \E$ to subsets $F(x)\subset\Y.$ The map $F$ is called {\em inner-semicontinuous} at  $\bar x\in\E$ if for any vector $\bar y\in F(\bar x)$ and any sequence $x_i\to\bar x$, there exists a sequence $y_i\in F(x_i)$ converging to $\bar y$.

\begin{proposition}[Condition $(b)$ and inner semicontinuity]\label{prop:inner_semi_b}
	Consider a set $\cX$ and a subset $\cY\subset \cX$. Suppose that $\cX$ is prox-regular at some point $\bar x\in\cY$ and that that the normal cone map $N_{\cX}$ is inner-semicontinuous on $\cY$ near $\bar x$. Then $\cX$ is $(b_=)$-regular along $\cY$ near $\bar x$.
\end{proposition}
\begin{proof}
	Consider sequences $x_i\in \cX$ and $y_i\in \cY$ converging to a point $y\in \cY$ near $\bar x$. Let $v_i\in N_{\cX}(x_i)$ be arbitrary unit normal vectors. Passing to a subsequence we may assume that  $v_i$ converge to some unit normal vector $\bar v\in N_{\cX}(y)$. By inner semicontinuity, there exist unit vectors $w_i\in N_{\cX}(y_i)$ converging to $\bar v$. Define the unit vectors $u_i:=\frac{x_i-y_i}{\|x_i-y_i\|}$. Prox-regularity of $\cX$ therefore guarantees
	$\langle v_i, u_i\rangle\geq -\frac{\rho}{2}\|x_i-y_i\|$ and $\langle w_i, u_i\rangle\leq \frac{\rho}{2}\|x_i-y_i\|.$
	We conclude 
	$$-\frac{\rho}{2}\|x_i-y_i\|\leq \langle v_i, u_i\rangle=\langle w_i, u_i\rangle+\langle v_i-w_i,u_i\rangle\leq \frac{\rho}{2}\|x_i-y_i\|+ \|v_i-w_i\|.$$
	Noting that the left and right sides both tend to zero completes the proof.
\end{proof}

 In particular, any proximally smooth set is $(b_{=})$-regular along any of its partly smooth submanifolds in the sense of Lewis \cite{lewis2002active}.

\subsection{Preservation of regularity under preimages by  transversal maps}

More interesting examples may be constructed through calculus rules. The next theorem shows that the four regularity conditions are preserved by taking preimages of smooth maps under a transversality condition.

\begin{thm}[Smooth preimages]\label{thm:preim}
	Consider a  $C^{1}$-map $F\colon\Y\to\E$ and an arbitrary point $\bar x\in \Y$.   Let $\cX,\cY\subset\E$ be two locally closed sets with $\cY$ Clarke regular and containing $F(\bar x)$. Suppose that the transversality condition holds:
	\begin{equation}\label{eqn:transv_calc}
	N_{\cY}(F(\bar x))\cap \Nul(\nabla F(\bar x)^*)=\{0\}.
	\end{equation}
	Then the following are true.
	\begin{enumerate}
	\item If $\cX$ is $(a)$-regular along $\cY$ at $F(\bar x)$  then $F^{-1}(\cX)$ is $(a)$-regular along $F^{-1}(\cY)$ at $\bar x$.
	\item If $\cX$ is $(a)$-regular and $(b_{\diamond})$-regular along $\cY$ at $F(\bar x)$,  then $F^{-1}(\cX)$ is $(b_{\diamond})$-regular along $F^{-1}(\cY)$ at $\bar x$.
	\end{enumerate}
If in addition $F$ is $C^2$-smooth, then the following are true.
\begin{enumerate}
	\item[3] If $\cX$ is strongly $(a)$-regular along $\cY$, then $F^{-1}(\cX)$ is strongly $(a)$-regular along $F^{-1}(\cY)$ at $\bar x$.
	\item[4]  If $\cX$ is both $(a)$-regular and strongly $(b_{\diamond})$-regular along $\cY$ at $F(\bar x)$,  then $F^{-1}(\cX)$ is strongly $(b_{\diamond})$-regular along $F^{-1}(\cY)$ at $\bar x$.
\end{enumerate}
\end{thm}
\begin{proof}
Notice that the transversality condition \eqref{eqn:transv_calc} is stable under perturbation of $\bar x$. In particular, it straightforward to see that there exists a constant $\tau>0$ and a neighborhood $U$ of $\bar x$ satisfying 
	$$\|\nabla F(y)^*v\|\geq \tau \|v\|\qquad \textrm{for all }y\in F^{-1}(\cY)\cap U,~v\in N_{\cY}(F(y)).$$
	Moreover, shrinking $U$, we may assume that $F$ is $\ell$-Lipschitz continuous on $U$.
	We prove the theorem in the order: $(1), (3), (2), (4)$.
	
	\bigskip
	\noindent{Claim $1$:}{
	Suppose that $\cX$ is $(a)$-regular along $\cY$ is at $F(\bar x)$. 
	Then, shrinking $\eta,\tau>0$ and $U$, we may ensure: 
	\begin{equation}\label{eqn:unif_trans}
	\|\nabla F(x)^*v\|\geq \tau \|v\|\qquad \textrm{for all }x\in F^{-1}(\cX)\cap U,~v\in N_{\cX}(F(x)).
	\end{equation}
	Transversality and Clarke regularity of $\cY$ imply  \cite[Theorem 10.6]{rockafellar2009variational}
	\begin{equation}\label{eqn:uni_equil_subd}
	N_{F^{-1}(\cY)}(y)=\nabla F(y)^*N_{\cY}(F(y))\qquad \textrm{and}\qquad N_{F^{-1}(\cX)}(x)\subset\nabla F(x)^*N_{\cX}(F(x))
	\end{equation}
	for all $y\in F^{-1}(\cY)$ and $x\in F^{-1}(\cX)$ sufficiently close to $\bar x$.

	Consider now a sequence $x_i\in F^{-1}(\cX)$ converging to a point $y\in F^{-1}(\cY)$ near $\bar x$ and a sequence of unit normal vectors $w_i\in N_{F^{-1}(\cX)}(x_i)$ converging to some vector $ w$. Using \eqref{eqn:uni_equil_subd}, we may write 
	$w_i=\nabla F(x_i)^*v_i$ for some vectors $v_i\in N_{\cX}(F(x_i))$. Note that due to \eqref{eqn:unif_trans}, the sequence $v_i$ is bounded. Indeed, the norm of $v_i$ is upper bounded by a constant that is independent of $x_i$ and $y_i$. Therefore passing to a subsequence we may suppose $v_i$ converges to some vector $v$. Since $\cX$ is $(a)$-regular  along $\cY$  at $F(\bar x)$, the inclusion $v\in N_{\cY}(F(y))$ holds. Therefore using \eqref{eqn:uni_equil_subd} we deduce
	$w=\lim_{i\to\infty} \nabla F(x_i)^*v_i=\nabla F(y)^*v\in N_{F^{-1}(\cY)}(F(y))$. Thus  $F^{-1}(\cX)$ is $(a)$-regular along $F^{-1}(\cY)$  near $\bar x$.}

	\bigskip
	
Before moving on to the next three claims, note that each of them implies condition $(a)$  and therefore we can be sure that the expressions  \eqref{eqn:unif_trans} and \eqref{eqn:uni_equil_subd} hold. Therefore for the rest of the proof, we will fix sequences $x_i$, $v_i$, and $w_i$  as in the proof of condition $(a)$, and we let $y_i\in F^{-1}(\cY)$ be an arbitrary sequence near $\bar x$.
 		
	\bigskip	
	\noindent{Claim $3$:} Suppose that $F$ is $C^2$-smooth and that $\cX$ is strongly $(a)$-regular along $\cY$  at $F(\bar x)$. Let $C>0$ be the corresponding constant in \eqref{eqn:strong_a}. 
	  Shrinking $U$ we may assume $\nabla F$ is $L$-Lipschitz continuous on $U$. 	We successively compute
	\begin{align}
	\dist(w_i,N_{F^{-1}(\cY)}(y_i))&=\dist(\nabla F(x_i)^*v_i, N_{F^{-1}(\cY)}(y_i))\notag\\
	&\leq \|\nabla F(x_i)-\nabla F(y_i)\|_{\rm op}\|v_i\|+\dist(\nabla F(y_i)^*v_i, N_{F^{-1}(\cY)}(y_i))\label{eqn:dumb1}\\
	&=\|\nabla F(x_i)-\nabla F(y_i)\|_{\rm op}\|v_i\|+ \dist(\nabla F(y_i)^*v_i, \nabla F(y_i)^*N_{\cY}(F(y_i)))\label{eqn:dumb2}\\
	&\leq \|\nabla F(x_i)-\nabla F(y_i)\|_{\rm op}\|v_i\|+ \|\nabla F(y_i)\|_{\rm op}\dist(v_i, N_{\cY}(F(y_i)))\\
	&\leq L\|v_i\|\|x_i-y_i\|+C\ell\|v_i\|\|F(x_i)-F(y_i)\|\label{eqn:dumb3}\\
	&\leq (L+C\ell^2)\|v_i\|\|x_i-y_i\|\notag\\
	&\leq (L+C\ell^2)\tau^{-1}\|\nabla F(x_i)^*v_i\|\|x_i-y_i\|\label{eqn:dumb4}\\
	&=(L+C\ell^2)\tau^{-1}\|w_i\|\|x_i-y_i\|,\notag
	\end{align}
	where \eqref{eqn:dumb1} follows from the triangle inequality, \eqref{eqn:dumb2} follows from \eqref{eqn:uni_equil_subd}, the estimate \eqref{eqn:dumb3} follows from strong $(a)$-regularity, and \eqref{eqn:dumb4} follows from \eqref{eqn:unif_trans}. Thus  $F^{-1}(\cX)$ is strongly $(a)$-regular along $F^{-1}(\cY)$  near $\bar x$.

\bigskip

	Setting the stage for the remainder of the proof, we compute
	\begin{align}
	\langle w_i,y_i-x_i\rangle=\langle v_i, F(y_i)-F(x_i)\rangle -\langle v_i, F(y_i)-F(x_i)-\nabla F(x_i)(y_i-x_i)\rangle.\label{eqn:da}
	\end{align}
	
	\noindent{Claim $2$: }{ 
	Suppose that $\cX$ is $(a)$-regular and $(b_{\diamond})$-regular along $\cY$ near $F(\bar x)$. Dividing \eqref{eqn:da} though by $\|x_i-y_i\|$ and taking into account that $F$ is $C^1$-smooth, we deduce that the limit points of $\langle w_i,\frac{y_i-x_i}{\|y_i-x_i\|}\rangle$ inherit the sign from the limit points of $\langle v_i, \frac{F(y_i)-F(x_i)}{\|F(y_i)-F(x_i)\|}\rangle$. Thus  $F^{-1}(\cX)$ is $(b_{\diamond})$-regular along $F^{-1}(\cY)$  near $\bar x$.
	}
	
	\bigskip	
	\noindent{Claim $4$: }{This is completely analogous to the proof of $(b_{\diamond})$-regularity, except we divide \eqref{eqn:da} though by $\|x_i-y_i\|^2$ and pass to the limit.}
\end{proof}

\subsection{Preservation of regularity under spectral lifts}\label{sec:preser_spec_set}
In this section, we study the prevalence of the four regularity conditions in eigenvalue problems. 
We begin with some notation. The symbol $\cS^{n}$ will denote the Euclidean space of symmetric matrices, endowed with the trace inner product $\langle A,B\rangle={\rm tr}(AB)$ and the induced Frobenius norm $\|A\|=\sqrt{{\rm tr}(A^2)}$. 
The symbol $O(n)$ will denote the set of $n\times n$ orthogonal matrices. The eigenvalue map $\lambda\colon\cS^n\to\R^n$ assigns to every matrix $X$ its ordered list of eigenvalues
$$\lambda_1(X)\geq \lambda_2(X)\geq \ldots \geq \lambda_n(X).$$
The following class of sets will be the subject of the study.

\begin{definition}{\rm A set $\cX\subset \R^{n} \rightarrow \overline{\R}$ is called {\em symmetric}  if it satisfies
$$
\pi \cX\subset \cX \qquad\textrm{for all }\pi \in \Pi(n).
$$
}
\end{definition}

\begin{definition}{\rm A set $\mathcal{Q}\subset \cS^{n}$ is called {\em spectral} if it satisfies
$$
U\mathcal{Q} U^T \subset \mathcal{Q} \qquad\textrm{for all }U \in O(n).
$$
}
\end{definition}

Thus a set in $\R^n$ is symmetric if it is invariant under reordering of the coordinates. For example, all $\ell_p$-norm balls, the nonnegative orthant, and the unit simplex are symmetric.
 A set in $\cS^{n}$ is spectral if it is invariant under conjugation of its argument by orthogonal matrices. Spectral sets are precisely those that can be written as $\lambda^{-1}(\cX)$ for some symmetric set $\cX\subset \R^n$.  See figure~\ref{fig:lp_balls_sch} for an illustration.

\begin{figure}[ht]
\centering
\begin{subfigure}{.19\textwidth}
  \centering
  \includegraphics[width=\linewidth]{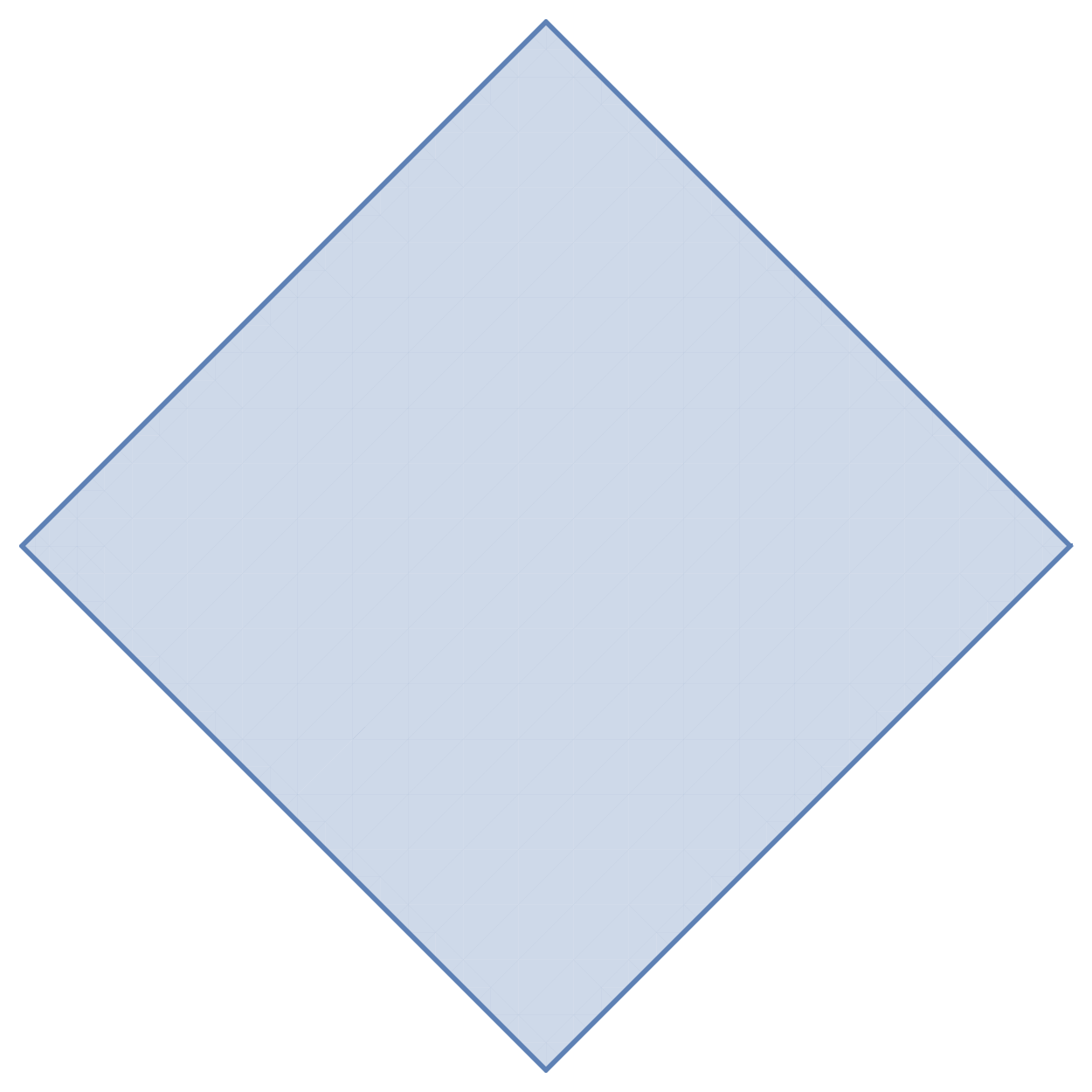}  
  \label{fig:sub-first_first}
\end{subfigure}
\begin{subfigure}{.19\textwidth}
  \centering
  \includegraphics[width=\linewidth]{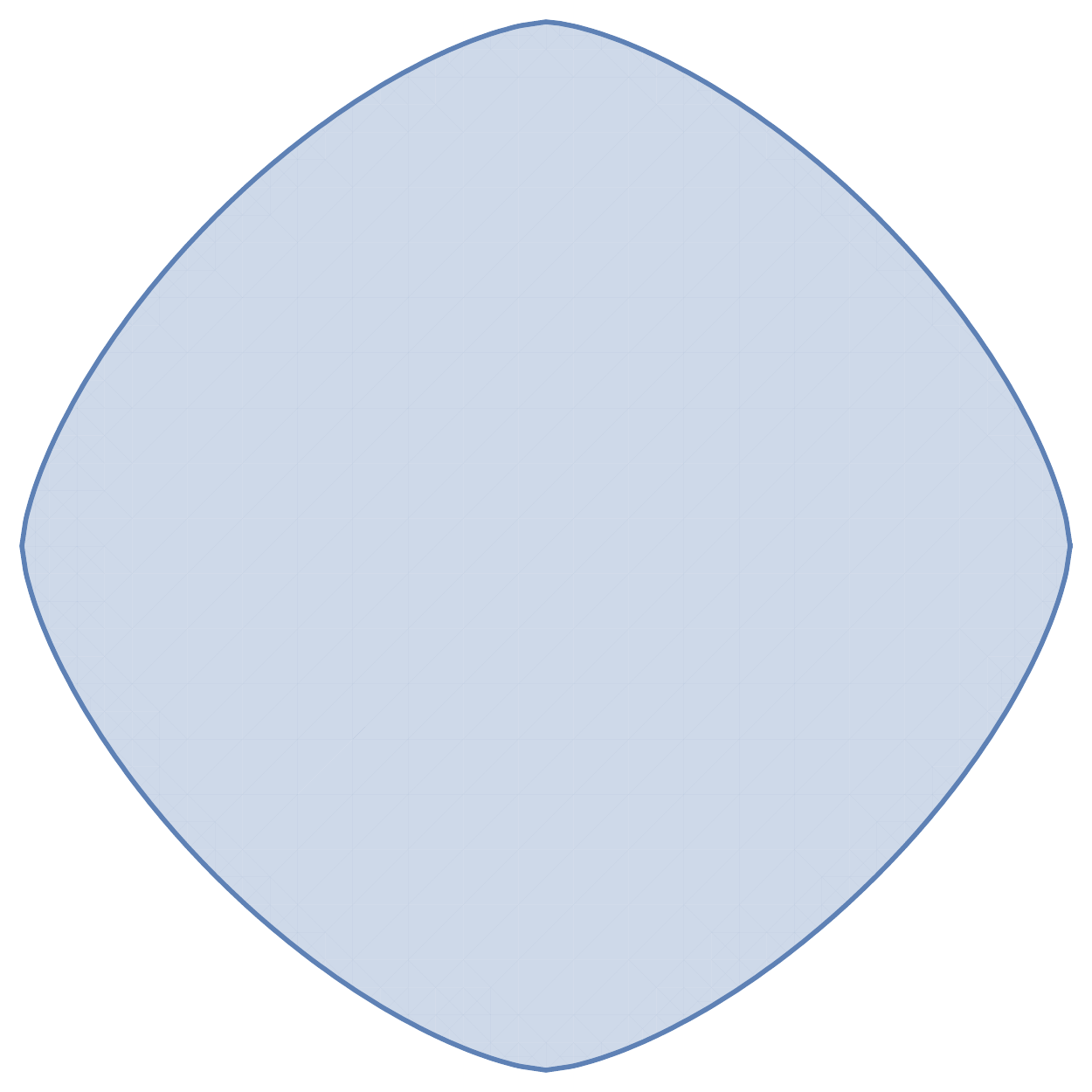}  
\end{subfigure}
\begin{subfigure}{.19\textwidth}
  \centering
  \includegraphics[width=\linewidth]{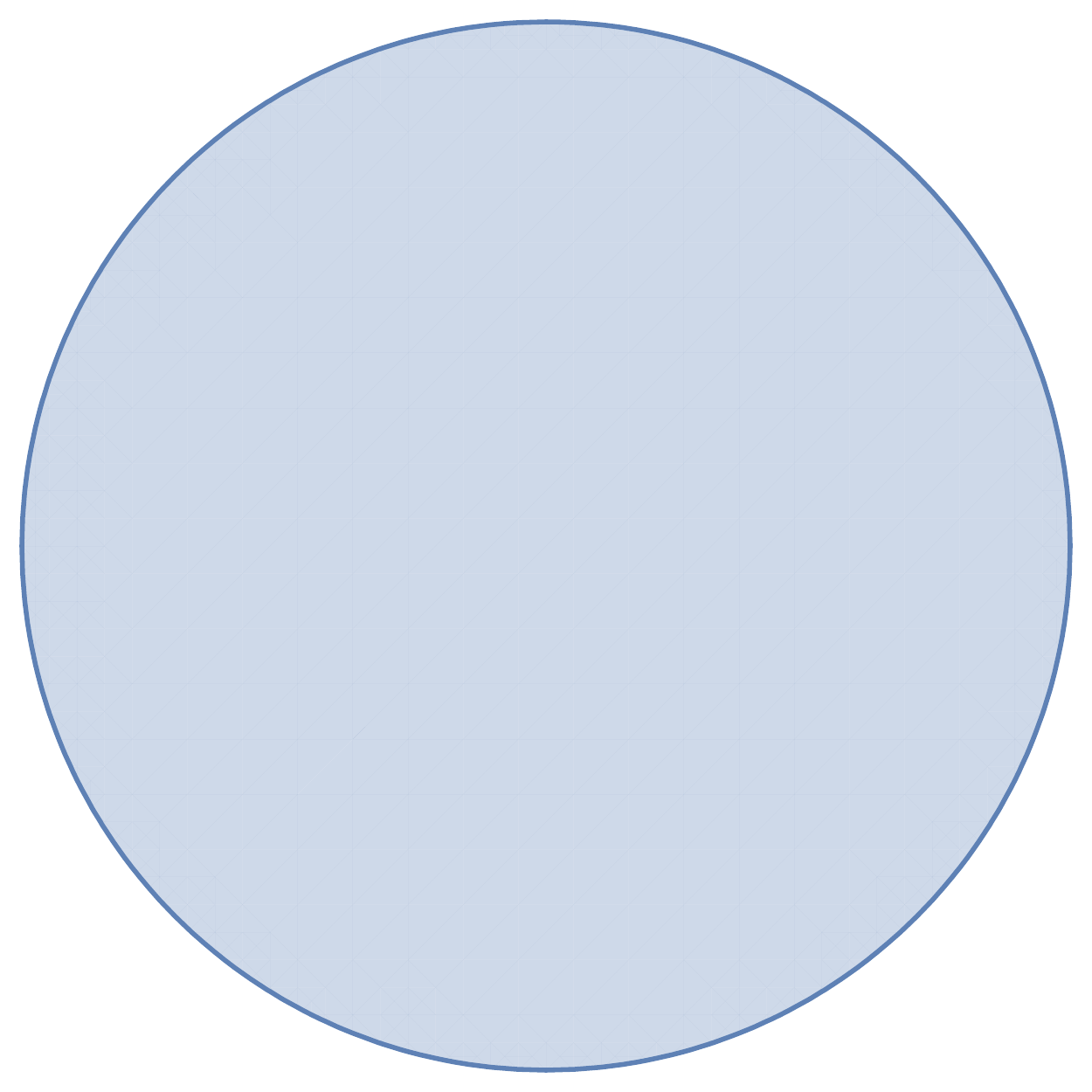}  
\end{subfigure}
\begin{subfigure}{.19\textwidth}
  \centering
  \includegraphics[width=\linewidth]{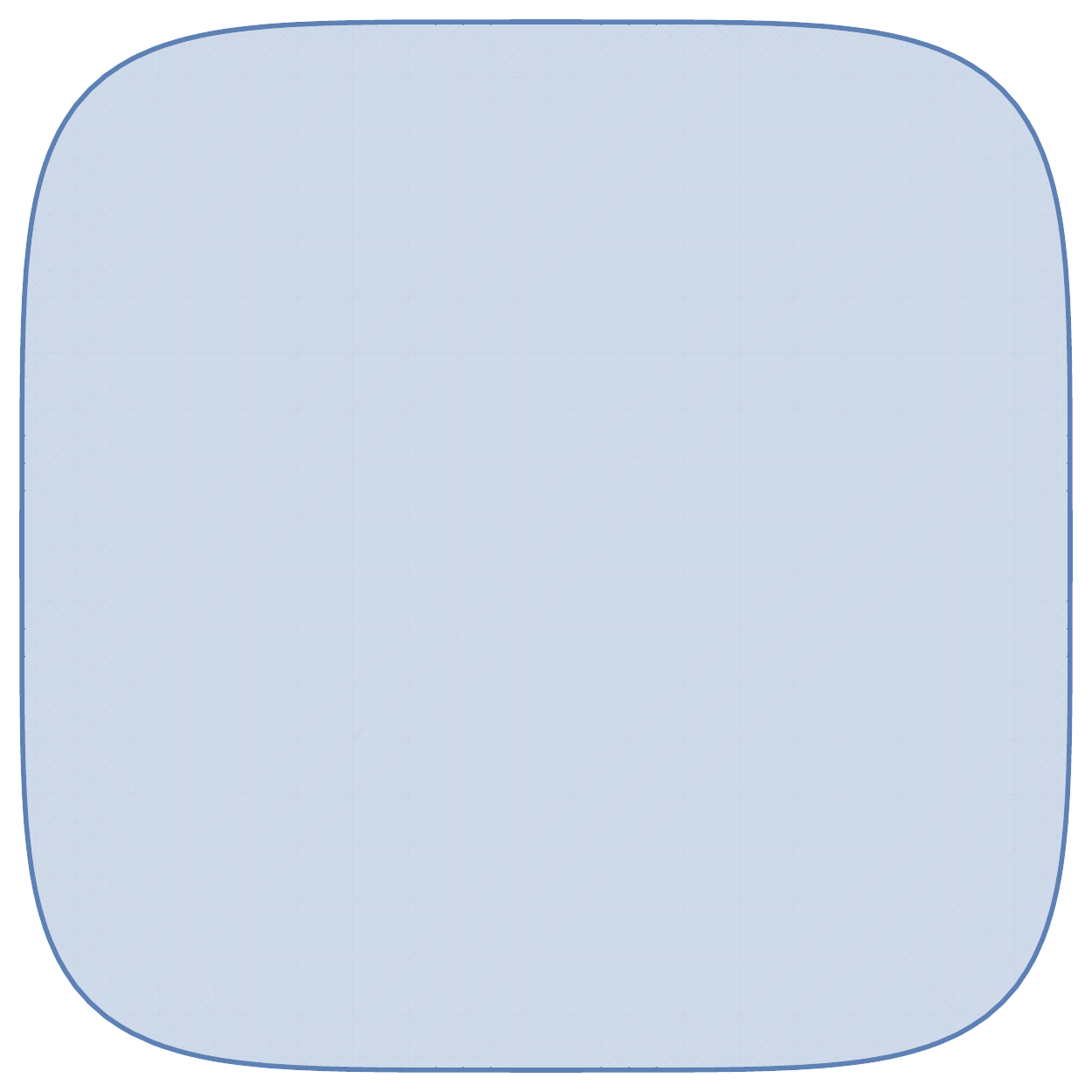}  
\end{subfigure}
\begin{subfigure}{.19\textwidth}
  \centering
  \includegraphics[width=\linewidth]{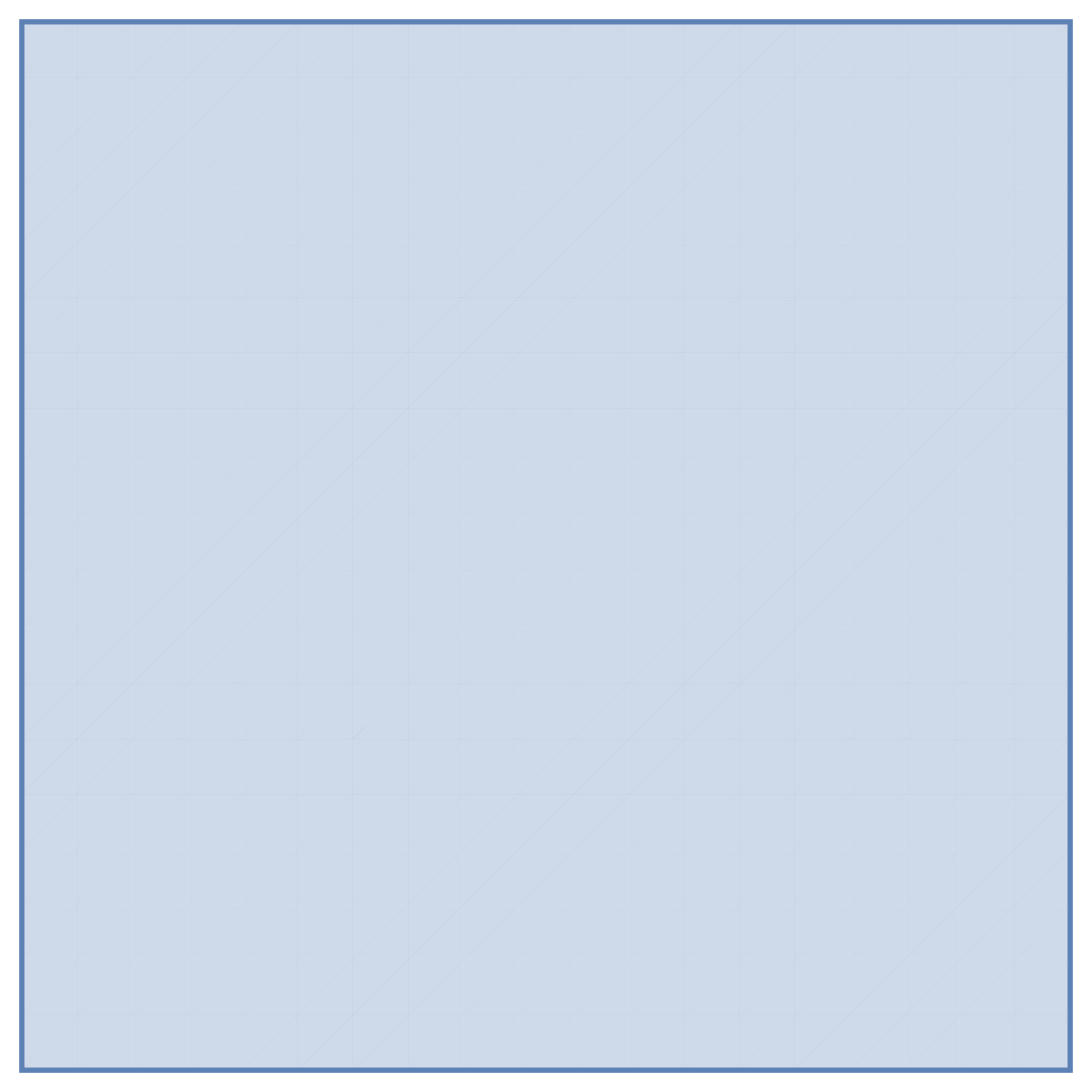}  
\end{subfigure}
\\

\begin{subfigure}{.19\textwidth}
  \centering
  \includegraphics[width=\linewidth]{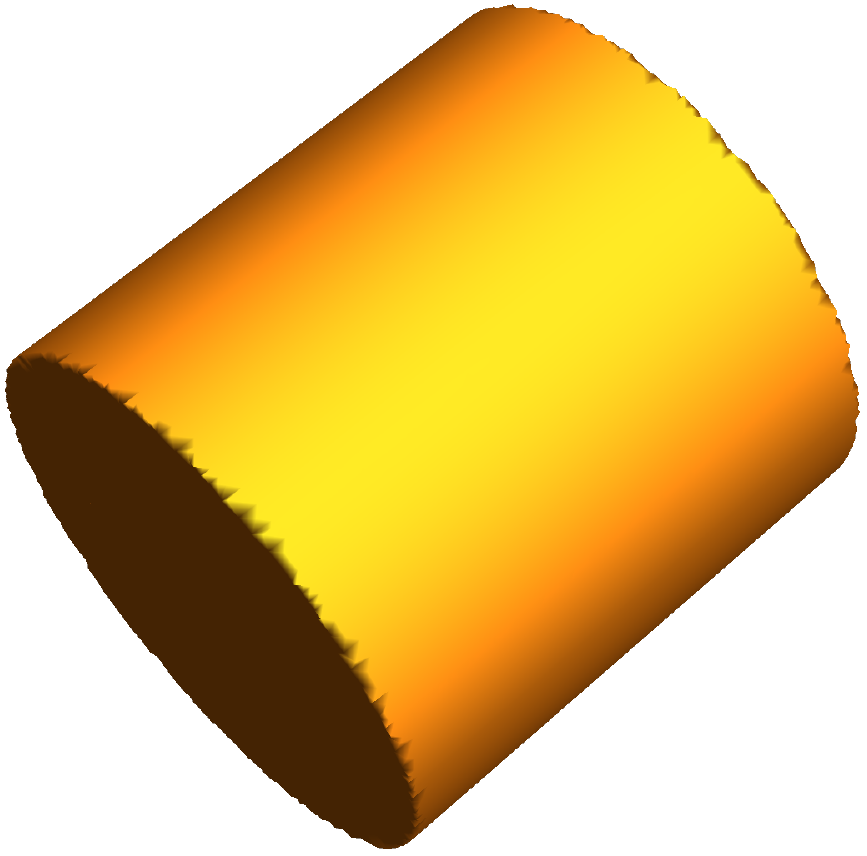}  
  \caption{$p=1$}
  \label{fig:sub-first_sec}
\end{subfigure}
\begin{subfigure}{.19\textwidth}
  \centering
  \includegraphics[width=\linewidth]{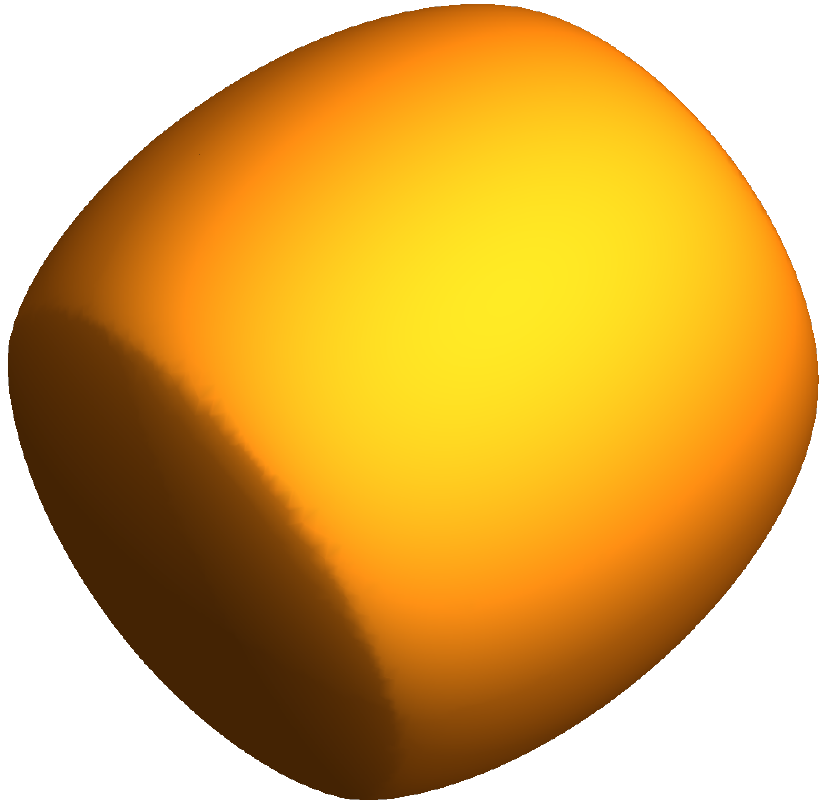}  
  \caption{$p=1.5$}
\end{subfigure}
\begin{subfigure}{.19\textwidth}
  \centering
  \includegraphics[width=\linewidth]{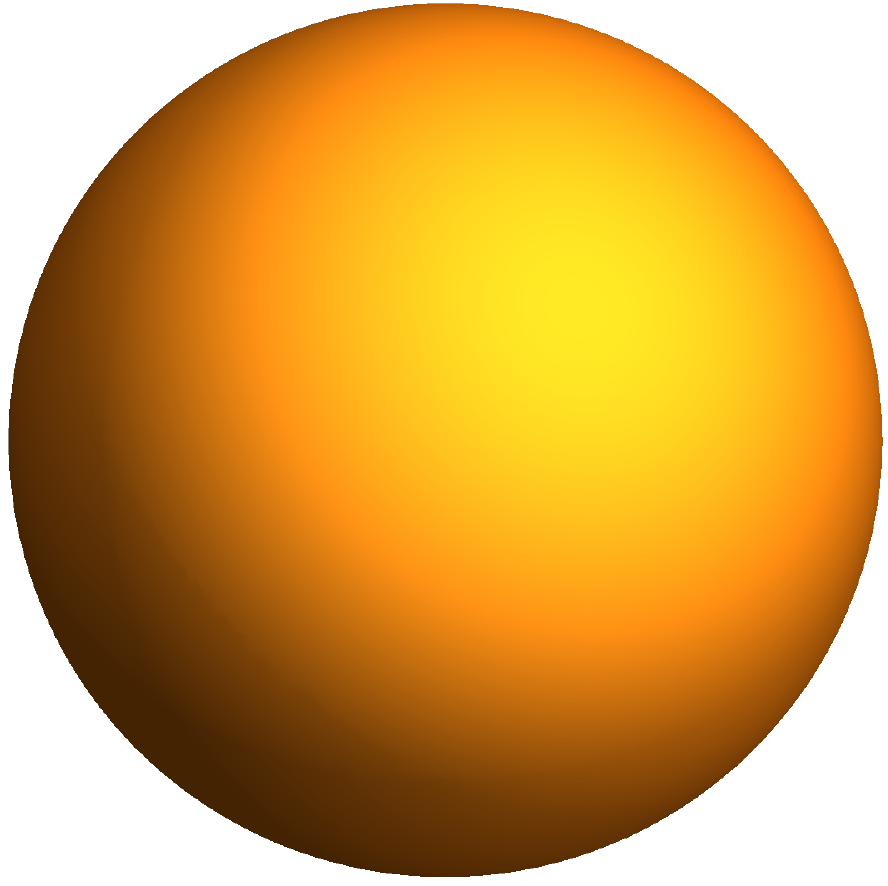}  
  \caption{$p=2$}
\end{subfigure}
\begin{subfigure}{.19\textwidth}
  \centering
  \includegraphics[width=\linewidth]{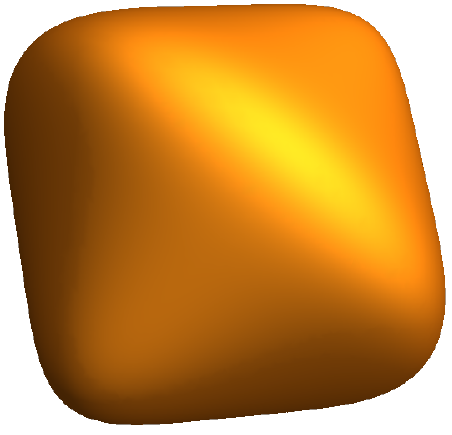}  
  \caption{$p=5$}
\end{subfigure}
\begin{subfigure}{.19\textwidth}
  \centering
  \includegraphics[width=\linewidth]{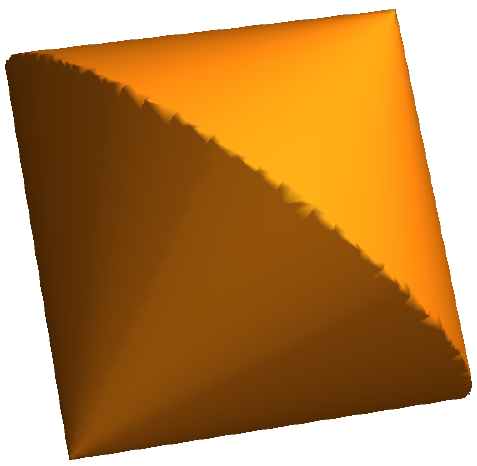}  
  \caption{$p=\infty$}
\end{subfigure}
\caption{Unit $\ell_p$ balls in $\R^2$ (top row) and unit balls of Schatten $\ell_p$-norms $\|A\|_{p}=\|\lambda(A)\|_p$ over $\cS^2$ (bottom row).}
\label{fig:lp_balls_sch}
\end{figure}

A prevalent theme in variational analysis is that a variety of geometric properties of a symmetric set $\cX$ and those of its induced spectral set $\lambda^{-1}(\cX)$ are in one-to-one correspondence.
Notable examples include convexity \cite{lewis1996convex,davis1957all}, smoothness \cite{lewis1996derivatives,lewis2005nonsmooth}, prox-regularity \cite{daniilidis2008prox}, and partial smoothness \cite{daniilidis2014orthogonal}. In this section, we add to this list the four regularity conditions. The key idea of the arguments is to pass through the projected conditions (Definition~\ref{defn:project_condt}) and then invoke Theorem~\ref{thm:atob}.

We will use the following expressions for the normal cone and the projection map to spectral sets $\lambda^{-1}(\cX)$: 
\begin{equation}\label{eqn:spectral_formulas}
\begin{aligned}
\proj_{\lambda^{-1}(\cX)}(X)&=\left\{U\Diag(w)U^T: w\in \proj_{\cX}(\lambda(X)),~ U\in O_X\right\}\\
N_{\lambda^{-1}(\cX)} (X)&=\left\{U\Diag(y)U^T: y\in N_{\cX} (\lambda(X)),~U\in O_X\right\}
\end{aligned}.
\end{equation}
where for any matrix $X$, we define the set of diagonalizing matrices 
$$O_X:=\{U\in  O(n): X=U\Diag(\lambda(X))U^T\}.$$
The expression for the proximal map was established in \cite{MR3756929} while the normal cone formula was proved in \cite{lewis1999nonsmooth}. An elementary proof of the subdifferential formula appears in \cite{MR3756929}.

\begin{thm}[Spectral preservation of projected regularity]\label{thm:spec_proj}
	Let $\bar X \in \cS^{n}$ be a symmetric matrix and set $\bar x=\lambda(\bar X)$. Consider two locally closed symmetric sets $\cX, \cY \subseteq \R^n$ such that $\cY$ contains $\bar x$. Let $\pi$ and $\Pi$ be the nearest-point projections onto $\cY$ and $\lambda^{-1}(\cY)$, respectively.
	 Then the following are true.
	\begin{enumerate}
		\item 	If $\cX$ is $(a)$-regular along $\cY$ near $\bar x$, then $\lambda^{-1}(\cX)$ is  $(a)$-regular along  $\lambda^{-1}(\cY)$ near $\bar X$.
		\item If $\cY$ is prox-regular at $\bar x$ and $\cX$ is strongly $(a^{\pi})$-regular along $\cY$ near $\bar x$, then $\lambda^{-1}(\cY)$ is prox-regular at $\bar X$ and $\cX$ is strongly $(a^{\Pi})$-regular along  $\lambda^{-1}(\cX)$ near $\bar X$.
		The analogous statement holds for $(b^{\pi}_{\diamond})$ and strong  $(b^{\pi}_{\diamond})$ conditions.
	\end{enumerate}	
\end{thm}
\begin{proof}
The result for $(a)$-regularity holds trivially from \eqref{eqn:spectral_formulas}. Suppose now that $\cY$ is prox-regular at $\bar x$. Then  the work \cite{daniilidis2008prox}  guarantees  that $\lambda^{-1}(\cY)$ is prox-regular at $\bar X$. As preparation for the rest of the proof, consider an arbitrary matrix $X\in \lambda^{-1}(\cX)$ near $\bar X$ and a normal vector $V\in N_{\lambda^{-1}(\cX)}(X)$ with unit Frobenius length. We may then write 
$$V = U\Diag(v) U^T,$$
for some unit vector $v \in N_{\cX}(\lambda(X))$ and orthogonal matrix $U\in O_X$.
Setting $Y:=\Pi(X)$  and using \eqref{eqn:spectral_formulas}, we may write 
$$Y=U\Diag(\pi(\lambda(X)))U^T.$$
Notice that because the coordinates of $\lambda(X)$ are decreasing and $\cY$ is symmetric, the coordinates of $\pi(\lambda(X))$ are also decreasing; otherwise, one may reorder $\pi(\lambda(X))$ and find a vector closer to $\lambda(X)$ in $\cY$. Consequently, we have
\begin{equation}\label{eqn:sim_order}
\lambda(Y)=\pi(\lambda(X))\qquad \textrm{and}\qquad U\in O_Y.
\end{equation}

Suppose now that $\cX$ is  strongly $(a^{\pi})$-regular along $\cY$ near $\lambda(\bar X)$ and let $C$ be the corresponding constant in \eqref{eqn:strong_a}. Thus there exists $w \in N_{\cY}(\pi(\lambda(X)))$ satisfying
	$$
	\|v-w\| = \dist(v,  N_{\cY}(P_{\cY}(\lambda(X)))\leq C\|\lambda (X)-\pi(\lambda(X))\|=C\|X-Y\|,
	$$
	where the last equation follows $X$ and $Y$ being simultaneously diagonalizable.	
	Taking into account \eqref{eqn:spectral_formulas} and \eqref{eqn:sim_order}, we deduce that  $W:=U\Diag(w)U^T$ lies in $N_{\lambda^{-1}(\cY)}(Y)$. Therefore we compute 
\begin{align*}
	\dist(V,N_{\lambda^{-1}(\cY)}(Y))\leq \|V - W\|= \|v-w\| \leq C\|X-Y\|.
	\end{align*}
Thus $\lambda^{-1}(\cX)$ is strongly $(a^{\pi})$-regular along $\lambda^{-1}(\cY)$ near $\bar X$, as claimed.

Next moving onto conditions  $(b^{\pi}_{\diamond})$ and strong $(b^{\pi}_{\diamond})$, we compute
$$\langle V,X-Y\rangle=\langle v,\lambda(X)-\pi(\lambda(X))\rangle.$$
The claimed results now follow immediately by noting $\|\lambda(X)-\pi(\lambda(X))\|=\|X-Y\|$.
\end{proof}

Combining Theorems~\ref{thm:spec_proj}, \ref{thm:atob}, and spectral preservation of smoothness \cite{daniilidis2014orthogonal} yields the main result of the section.

\begin{proposition}[Spectral Lifts]
		Let $\bar X \in \cS^{n}$ be a symmetric matrix and set $\bar x=\lambda(\bar X)$. Consider two locally closed symmetric sets $\cX,\cY\subseteq \R^n$ such that $\cY$ contains $\bar x$. 
	Then the following are true.
	\begin{enumerate}
		\item If $\cY$ is a $C^2$-smooth manifold at $\bar x$ and $\cX$ is strongly $(a)$-regular along $\cY$ near $\bar x$, then $\lambda^{-1}(\cY)$ is a $C^2$-smooth manifold at $\bar X$ and $\cX$ is strongly $(a)$-regular along  $\lambda^{-1}(\cX)$ near $\bar X$. The analogous statement holds for $(b_{\diamond})$.
		\item If $\cY$ is a $C^3$-smooth manifold at $\bar x$ and $\cX$ is both strongly $(a)$ and strongly $(b)$ regular along $\cY$ near $\bar x$, then $\lambda^{-1}(\cY)$ is a $C^3$-smooth manifold at $\bar X$ and $\cX$ is both strongly $(a)$ and strongly $(b)$ regular along  $\lambda^{-1}(\cX)$ near $\bar X$. 
	\end{enumerate}	
	\end{proposition}
\begin{proof}
This follows directly by combining Theorems~\ref{thm:spec_proj}, \ref{thm:atob}, and spectral preservation of smoothness \cite[Theorem 2.7]{daniilidis2014orthogonal} yields the main result of the section. 
\end{proof}

All the results in this section extend in a standard way (e.g. \cite{lewis2005nonsmooth}) to orthogonally invariant sets of {\em rectangular matrices} $X\in \R^{m\times n}$. Namely, one only needs to replace (i) eigenvalues $\lambda_i(X)$ with singular values $\sigma_i(X)$, (ii) symmetric sets $\cX$ with absolutely symmetric sets (i.e. those invariant under all {\em signed permutations} of coordinates), and (iii) spectral sets $\mathcal{Q}$ with those that are in variant under the map $X\mapsto UXV^{\top}$ for any orthogonal matrices $U\in O(m)$ and $V\in O(n)$.

\subsection{Regularity of functions along manifolds}
The previous sections developed basic examples and calculus rules for the four basic regularity conditions.
In this section we interpret these results for functions through their epigraphs.
We begin with the following lemma, which follows directly from Propositions~\ref{prop:CONES}, \ref{prop:concl_affine_sub}, and \ref{prop:inner_semi_b}.

\begin{lem}[Basic examples]\label{lem:func_exa}
	Consider a function $f\colon\E\to\R\cup\{\infty\}$, a set $\cM\subset\dom\, f$, and a point $\bar x\in \cM$. The following statements are true.
	\begin{enumerate}
		\item If $f$ is a sublinear function and $\cM=\{x: f(x)=-f(-x)\}$ is its lineality space, then $f$ is both strongly $(a)$ and strongly $(b_{=})$ regular along $\cM$ near $\bar x$.
		\item If $f$ is convex, $\cM$ is locally affine near $\bar x$, and $f$ restricted to $\cM$ is an affine function near $\bar x$, then $f$ is strongly $(a)$-regular along $\cM$ near $\bar x$.
		\item If $f$ is weakly convex and locally Lipschitz  near $\bar x$ and the subdifferential map $x\mapsto \partial f(x)$ is inner-semicontinuous on $\cM$ near $\bar x$, then $f$ is $(b_{=})$-regular along $\cM$ near $\bar x$. 
	\end{enumerate}
\end{lem}

The baic calculus rule established in Theorem~\ref{thm:preim} yields the following chain rule.

\begin{thm}[Chain rule]\label{thm:chain_rule}
	Consider a  $C^p$-smooth map $c\colon\bf{Y}\to\E$ and a closed function $h\colon\E\to\R\cup\{\infty\}$. Fix a set $\cM\subset \E$ and a point $\bar x$ with $c(\bar x)\in \cM$. Suppose that $\cM$ is a $C^1$ manifold around $c(\bar x)$, the restriction $h\big|_{\cM}$ is $C^1$-smooth near $\bar x$, and transversality holds:
\begin{equation}\label{eqn:transv_calc_full}
	N_{\cM} (c(\bar x))\cap \Nul(\nabla c(\bar x)^*)=\{0\}.
\end{equation}
Define the composition $f(x)=h(c(x))$ and the set $\mathcal{L}:=c^{-1}(\cM)$. The following are true.
	\begin{enumerate}
	\item If $h$ is $(a)$-regular along $\cM$ near $c(\bar x)$  then $f$ is $(a)$-regular along $\mathcal{L}$ near $\bar x$.
	\item If $h$ is $(a)$-regular and $(b_{\diamond})$-regular along $\cM$ near $c(\bar x)$,  then $f$ is $(b_{\diamond})$-regular along $\cL$ near $\bar x$.
\end{enumerate}
If in addition $\cM$ is a $C^2$ manifold around $c(\bar x)$ and the restriction $h\big|_{\cM}$ is $C^2$-smooth near $\bar x$, then the following are true.
\begin{enumerate}
	\item[3] If $h$ is strongly $(a)$-regular along $\cM$, then $f$ is strongly $(a)$-regular along $\cL$ near $\bar x$.
	\item[4]  If $h$ is both $(a)$-regular and strongly $(b_{\diamond})$-regular along $\cM$ at $c(\bar x)$,  then $f$ is strongly $(b_{\diamond})$-regular along $\cL$ near $\bar x$.
\end{enumerate}
\end{thm}
\begin{proof}
First, the transversality condition \eqref{eqn:transv_calc_full} classically guarantees that $\mathcal{L}$ is a smooth manifold around $\bar x$ with the same order of smoothness as $\cM$. Moreover, for any $x\in \mathcal{L}$, we may write $f(x)=h(c(x))=(h\big|_{\cM}\circ c)(x)$. Therefore the restriction of $f$ to $\mathcal{L}$ has the same order of smoothness as $h\big|_{\cM}$.
Next, observe that we may write $\epi f=\{(x,r): (c(x),r)\in \epi h\}$. 
Thus in the notation of Theorem~\ref{thm:preim}, setting $\cX=\epi h$, $\cY=\gph h\big|_{\cM}$, and $F(x,r)=(c(x),r))$, we may write 
$$\epi f=F^{-1}(\cX)\qquad \textrm{ and }\qquad \gph f\big|_{\cL}=F^{-1}(\cY).$$ 
A quick computation shows that the transversality condition~\eqref{eqn:transv_calc} follows from \eqref{eqn:transv_calc_full}.
An application of Theorem~\ref{thm:preim} completes the proof.
\end{proof}

An interesting class of examples where the chain rule is useful consists of decomposable functions \cite{shapiroreducible}, which serve as functional analogues of cone reducible sets.
A function $f\colon\E\to\R\cup\{\infty\}$ is called {\em properly $C^p$ decomposable at $\bar x$ as }$h\circ c$ if on a neighborhood of $\bar x$ it can be written as 
		$$f(x)=f(\bar x)+h(c(x)),$$
		for some $C^p$-smooth mapping $c\colon\E\to{\bf Y}$ satisfying $c(\bar x)=0$ and some proper, closed sublinear function $h\colon{\bf Y}\to\R$ satisfying the transversality condition:
		$$\Lin(h)+ {\rm Range}(\nabla c(\bar x))={\bf Y}.$$

It is shown in \cite[p 683]{shapiroreducible} that if $f\colon\E\to\R\cup\{\infty\}$ is properly $C^p$ decomposable at $\bar x$ as $h\circ c$, then the set 
$\cL=c^{-1}(\Lin(h))$
is a $C^p$-active manifold around $\bar x$ for any subgradient $v\in \ri \partial f(\bar x)$. The following is immediate from Lemma~\ref{lem:func_exa} and Theorem~\ref{thm:chain_rule}.

\begin{cor}[Decomposable functions are regular]
	Suppose that a function $f$ is properly $C^1$ decomposable as $h\circ c$ around $\bar x$ and define $\mathcal{L}=c^{-1}(\Lin(h))$. Then $f$ is both $(a)$ and $(b_{=})$ regular along $\cL$ near $\bar x$. Moreover, if $f$ is properly $C^2$-decomposable  as $h\circ c$ around $\bar x$, then $f$ is strongly $(a)$ and strongly $(b_=)$ regular along $\cL$ near $\bar x$.
\end{cor}

The chain rule can be used to obtain a variety of other calculus rules, including the sum rule. To see this, note that regularity of functions $f_i$ along sets $\cM_i$ directly implies regularity
of the separable function $f(y_1,\ldots, y_k)=\sum_{i=1}^k f_i(y_i)$ along the product set $\prod_{i=1}^k\cM_i$.
 Then a general sum rule for $f(x)=\sum_{i=1}^k f_i(x)$ follows from applying the chain rule (Theorem~\ref{thm:chain_rule}) to the decomposition $f(x)=h(c(x))$ with the linear map $c(x)=(x,\ldots, x)$ and the separable function $h(y_1,\ldots, y_k)=\sum_{i=1}^n f_i(y_i)$. For the sake of brevity, we leave details for the reader.

We end the section with an extension of the material in Section~\ref{sec:preser_spec_set} to the functional setting. Namely, a function $f\colon\R^n\to\R\cup\{\infty\}$ is called {\em symmetric} if equality $f(\pi x)=f(x)$ holds for all $x\in \R^n$ and all $\pi\in \Pi(n)$. A function  $f\colon{\bf S}^n\to\R\cup\{\infty\}$ is called {\em spectral} if it satisfies $F(UXU^{\top})=F(X)$ for all $X\in {\bf S}^n$ and all $U\in O(n)$. It is straightforward to see that any spectral function $F$ decomposes as $F=f\circ \lambda$ for some symmetric function $f$. Explicitly, we may take $f$ as the diagonal restriction $f(x)=(F\circ \Diag)(x)$. The subdifferentials of $F$ and $f$ are related by the expressions \cite{lewis1999nonsmooth}:
\begin{equation}\label{eqn:spectral_formula2}
	\begin{aligned}
	\partial F(X)&=\left\{U\Diag(y)U^T: y\in\partial f(\lambda(X)),~U\in O_X\right\}
	\end{aligned}.
\end{equation}
where for any matrix $X$, we define the set of diagonalizing matrices 
$$O_X:=\{U\in  O(n): X=U\Diag(\lambda(X))U^T\}.$$
The following theorem shows that the regularity of a symmetric function $f$ is inherited by the spectral function $F=f\circ \lambda$.

\begin{thm}[Spectral Lifts]
Consider a symmetric function $f\colon\R^d\to\R\cup\{\infty\}$  and let $\cM$ be a symmetric $C^2$-manifold containing $\bar x$. Suppose that $f$ is locally Lipschitz continuous around $\bar x$ and the restriction $f\big|_{\cM}$ is $C^2$-smooth near $\bar x$. Fix now a matrix $\bar X$ satisfying $\bar x:=\lambda(\bar X)$. Then if $f$ is $(a)$-regular along $\cM$ around $\bar x$, then $f\circ \lambda$ is $(a)$-regular along  $\lambda^{-1}(\cM)$ near $\bar X$. The analogous statement holds for strong $(a)$-regularity and $(b_{\diamond})$-regularity. If $\cM$ is in addition $C^3$ smooth, then the analogous statement holds for strong $(b_{\diamond})$-regularity.
\end{thm}
\begin{proof}
	First, \cite[Theorem 2.7]{daniilidis2014orthogonal} shows that $\cM$ is a $C^p$ manifold (with $p\geq 2$) around $\bar x$ if and only if $\lambda^{-1}(\cM)$ is a $C^p$ manifold around $\bar X$. The analogous statement is true for the restriction of $f$ to $\cM$ and for the restriction of $f\circ\lambda$ to $\lambda^{-1}(\cM)$. 
	
	The claim about $(a)$-regularity follows immediately from  Lemma~\ref{lemma:inclusion_normal}.	
 The main idea for verifying the rest of the properties is to instead focus on the analogous conditions with respect to the retraction
 $\pi$ onto $\gph F\big|_{\lambda^{-1}(\cM)}$ defined by the expression
 $$\pi(X,r)=(P_{\lambda^{-1}(\cM)}(X),F(P_{\lambda^{-1}(\cM)}(X))).$$
To this end, suppose that $f$ is strongly $(a)$-regular near $\bar x$. 
  We claim that $\epi F$ is strongly $(a^{\pi})$ regular along $\gph f\big|_{\cM}$ near $(\bar X, F(\bar X))$. To see this, consider a matrix $X\in{\bf S}^n$ near $\bar X$ and set $Y=P_{\lambda^{-1}(\cM)}(X)$. Let $Z\in \partial F(X)$ be arbitrary. Exactly the same argument as in the proof of Theorem~\ref{thm:spec_proj} shows that there exists a matrix $W\in \partial F(Y)$ satisfying  $\|Z-W\|_F\leq C\|X-Y\|_F$, where $C$ is a fixed constant independent of $X$ and $Y$. It follows immediately that $\epi F$ is strongly $(a^{\pi})$ regular along $\gph f\big|_{\cM}$ near $(\bar X, F(\bar X))$. An application of Theorem~\ref{thm:atob} therefore guarantees that $F$ is strongly $(a)$ regular along $\lambda^{-1}(\cM)$ near $\bar X$.
  The claims about $(b_{\diamond})$ and strong $(b_{\diamond})$ properties follow similarly by using the characterization in Theorem~\ref{thm:reinterp0} and arguing regularity with respect to the retraction $\pi$. We leave the details for the reader.
\end{proof}

\subsection{Generic regularity along active manifolds}\label{sec:gen_reg}
How can one justify the use of a particular regularity condition? One approach, highlighted in the previous sections, is to verify the conditions for certain basic examples and then show that they are preserved under transverse smooth deformations.
Stratification theory adapts another viewpoint, wherein a regularity condition between two manifolds is considered acceptable if reasonable sets (e.g. semi-algebraic, subanalytic or definable) can always be partitioned into finitely many smooth manifolds so that the regularity condition holds along any two ``adjacent'' manifolds. See the survey \cite{trotman2020stratification} for an extensive discussion. 

To formalize this viewpoint, we begin with a definition of a stratification.

\begin{definition}[Stratification]
	{\rm A {\em $C^p$-stratification} ($p\geq 1$) of a set $Q\subset \E$ is a partition of $Q$ into finitely many $C^p$ manifolds, called {\em strata}, such that any two strata $\cX$ and $\cY$ satisfy  the implication:
		$$\cY\cap \cl \cX\neq \emptyset\quad \Longrightarrow\quad \cY\subset \cl \cX.$$
		A stratum $\cY$ is said to be {\em adjacent} to a stratum $\cX$ if the inclusion $\cY\subset \cl \cX$ holds. If the strata are definable in some o-minimal structure, the stratification is called {\em definable}.
	}
\end{definition}
Thus a  stratification of $Q$ is simply a partition of $Q$ into smooth manifolds so that the closure of any stratum is a union of strata. 
Stratifications such that any pair of adjacent strata are strongly (a)-regular are called Verdier stratifications.

\begin{definition}{\rm
	 A {\em $C^p$ Verdier stratification} ($p\geq 1$) of a set $Q\subset \E$ is a $C^p$ stratification of $Q$ such that any stratum  $\mathcal{X}$ is strongly (a)-regular along any stratum $\cY$ contained in $\cl \cX$.}
\end{definition}

It is often useful to refine stratifications. To this end,  a stratification is {\em compatible} with a collection of sets $Q_1,\ldots, Q_k$ if for every index $i$, every stratum $\cM$ is either contained in $Q_i$ or is disjoint from it.
The following theorem, due to Ta Le Loi \cite{le1998verdier}, shows that definable sets admit a Verdier stratification, which is compatible with any finite collection of definable sets.

\begin{thm}[Verdier stratification]\label{thm:stratification_exist}
	 For any $p\geq 1$, any definable set $Q\subset\E$ admits a definable  $C^p$ Verdier stratification. Moreover, given finitely many definable subsets $Q_1,\ldots,Q_k$, we may ensure that the Verdier stratification of $Q$ is compatible with $Q_1,\ldots,Q_k$.
\end{thm}

The analogous theorem for condition $(b_=)$ (and therefore condition $(a)$) was proved earlier; see the discussion in \cite{MR1404337}. The strong $(b_=)$ condition does not satisfy such decomposition properties. It can fail even relative to a single point of a definable set in $\R^2$, as Example~\ref{exa:bstrong_not_gen} shows. Nonetheless, as we have seen in previous sections,  it does hold in a number of  interesting settings in optimization (e.g. for cone reducible sets along the active manifold). 

\begin{example}[Strong $(b)$ is not generic]\label{exa:bstrong_not_gen}
	{\rm 
		Define the curve $\gamma(t)=(t,t^{3/2})$ in $\R^2$. Let $\cX$ be the graph of $\gamma$ and let $\cY$ be the origin in $\R^2$. Then a quick computation shows that a unit normal $u(t)\in N_{\cX}(\gamma(t))$ is given by $(-\frac{2}{3}\sqrt{t},1)/\sqrt{1+\frac{4}{9}t}$ and therefore 
		$$\left\langle  u(t),\frac{\gamma(t)}{\|\gamma(t)\|^2}\right\rangle=\frac{t^{3/2}}{3(t^2+t^3)\sqrt{1+\frac{4}{9}}t}\to \infty\qquad \textrm{as}\quad t\to 0.$$
		Therefore, the strong condition $(b_=)$ fails for the pair $(X,Y)$ at the origin.
	}
\end{example}

Applying Theorem~\ref{thm:stratification_exist}, to epigraphs immediately yields the following.

\begin{thm}[Verdier stratification of a function]
Consider a definable function $f\colon\E\to\R\cup\{\infty\}$ that is continuous on its domain. Then for any $p>0$, there exists a partition of $\dom f$ into finitely many $C^p$-smooth manifolds such that $f$ is $C^p$-smooth on each manifold $\cM$, and $f$ is strongly $(a)$-regular and $(b_=)$-regular along any manifold $\cM$. 
\end{thm}
\begin{proof}
We first form a nonvertical stratification $\{\cM_i\}$ of $\gph f$, guaranteed to exist by \cite{doi:10.1137/060670080}. Choose any integer $p\geq 2$. Restratifying using Theorem~\ref{thm:stratification_exist} yields a nonvertical $C^p$-Verdier stratification $\{\cK_j\}$ of $\gph f$. Let $\cX_j$ denote the image of $\cK_j$ under the canonical projection $(x,r)\mapsto x$. As explained in \cite{doi:10.1137/060670080}, each set $\cX_j$ is a $C^p$-smooth manifolds, the function $f$ restricted to $\mathcal{X}_j$ is $C^p$-smooth, and equality $\gph f\big|_{\cX_j}=\cK_j$ holds. 

Consider now an arbitrary stratum $\cK_j$. It remains to verify that $\epi f$  is strongly $(a)$-regular along $\cK_j$. This follows immediately from the fact that there are finitely many strata and that the inclusion $N_{\epi f}(X)\subset N_{\cK_l}(X)$ holds for any index $l$ and any $X\in \cK_l$. 
\end{proof}

In this work, we will be interested in sets that are regular along a particular manifold---the active one. Theorem~\ref{thm:stratification_exist} quickly implies that critical points of ``generic'' definable functions lie on an active manifold along which the objective function is strongly (a)-regular.

\begin{thm}[Regularity at critical points of generic functions]\label{thm:genericfunctionsregular}
Consider a closed definable function $f\colon\E\to \R\cup\{\infty\}$. Then for almost every direction $v\in\E$ in the sense of Lebesgue measure,  the perturbed function $f_{v}:=f(x)-\langle v,x\rangle$ has at most finitely many limiting critical points, each lying on a unique $C^p$-smooth active manifold and along which the function $f_v$ is strongly (a)-regular.
\end{thm}

This theorem is a special case of a more general result that applies to structured problems of the form
\begin{equation}\label{eqn:compos_crit_clarke}
\min_x ~g(x) +h(x)
\end{equation}
for definable functions $g$ and $h$. Algorithms that utilize this structure, such as the proximal subgradient method, generate a sequence that may convergence to {\em  composite Clarke critical points} $\bar x$, meaning those satisfying 
$$0\in \partial_c g(\bar x)+\partial_c h(\bar x).$$ 
This condition is typically weaker than $0\in \partial_c (g+h)(\bar x)$. Points $\bar x$ satisfying the stronger inclusion $0\in \partial g(\bar x)+\partial h(\bar x)$ will be called {\em  composite limiting critical}.

The following theorem shows that under a reasonably rich class of perturbations, the problem \eqref{eqn:compos_crit_clarke} admits no extraneous composite limiting critical points. Moreover each of the functions involved admits an active manifold along which the function is strongly $(a)$-regular. The proof is a small modification of  \cite[Theorem 5.2]{MR3461323}.

\begin{thm}[Regularity at critical points of generic functions]\label{thm:genericcomposite}
Consider closed definable functions $g\colon\E\to \R\cup\{\infty\}$ and $h\colon\E\to \R\cup\{\infty\}$ and define the parametric family of problems
\begin{equation}\label{eqn:perturb_prob}
\min_{x} ~f_{y,v}(x)=g(x)-\langle v,x\rangle+h(x+y)
\end{equation}
Define the tilted function $g_v(x)=g(x)-\langle v,x\rangle$. 
Then there exists an integer $N>0$ such that for almost all parameters $(v,y)$ in the sense of Lebesgue measure, the problem \eqref{eqn:perturb_prob} has at most $N$ composite Clarke critical points. Moreover, for any limiting composite critical point $\bar x$, there exists a unique vector 
$$\bar \lambda\in \partial h(\bar x+y) \qquad \textrm{ satisfying } -\bar \lambda\in \partial g_v(\bar x),$$ 
and the following properties are true.
\begin{enumerate}
\item The inclusions 
$\bar \lambda\in \hat \partial h(\bar x+y)$ and  $-\bar \lambda\in \hat \partial g_v(\bar x)$ hold.
\item $g_v$ admits a $C^p$ active manifold $\cM$ at $\bar x$ for $-\bar\lambda$  and $h$ admits a $C^p$ active manifold $\mathcal{K}$ at $\bar x+y$ for $\bar \lambda$, and the two manifolds intersect transversally:
$$N_{\cK}(\bar x)\cap N_{\cM}(\bar x)=\{0\}.$$
\item\label{it:claim3} $\bar x$ is either a local minimizer of $f_{y,v}$ or a $C^p$ strict active saddle point of $f_{y,v}$. 
\item\label{it:claim4} $g_v$ is strongly $(a)$-regular along $\cM$ at $\bar x$ and $h$ is strongly $(a)$-regular along $\mathcal{K}$ at $\bar x+y$.
\end{enumerate}
\end{thm}
\begin{proof}
All the claims, except for \ref{it:claim3} and \ref{it:claim4}, are proved in \cite{MR3461323}; note, that in that work, active manifolds are defined using the limiting subdifferential, but exactly the same arguments apply under the more restrictive Definition~\ref{defn:ident_man}. Claim 
\ref{it:claim3} is proved in \cite[Theorem 5.2]{davis2021proximal}\footnote{weak convexity is invoked in the theorem statement but is not necessary for the result.}; it is a direct consequence of the classical Sard's theorem and existence of stratifications. Claim~\ref{it:claim4} follows from a small modification to the proof of \cite{MR3461323}. Namely, the first-bullet point in the proof may be replaced by ``$g$ is $C^p$-smooth and strongly (a) regular on $X^j_i(\widehat{U}_i)$ and $h$ is $C^p$-smooth and strongly (a)-regular on $F^j_i(\widehat{U}_i)$''.
\end{proof}

\section{Algorithm and main assumptions}\label{sec:algos_main}
In this chapter, we introduce our main algorithmic consequences of the strong (a) and $(b_{\diamond})$ regularity properties developed in the previous sections. Setting the stage, throughout we consider a minimization problem 
\begin{align}\label{eq:mainprob}
\min_{x \in \RR^d} f(x),
\end{align}
where $f \colon \RR^d \rightarrow \RR\cup \{+\infty\}$ is a closed function. The function $f$ may enforce constraints or regularization; it may also be the population loss of a stochastic optimization problem. In order to simultaneously model algorithms which exploit such structure, we take a fairly abstract approach, assuming access to  a \emph{generalized gradient mapping} for $f$:
$$
G \colon \RR_{++} \times \dom f \times \RR^d \rightarrow \RR^d
$$ 
We then consider the following stochastic method: given $x_0 \in \RR^d$, we iterate 
\begin{align}\label{alg:perturbedGiteration}
x_{k+1} = x_k - \alpha_kG_{\alpha_k}(x_k, \perturb_k), 
\end{align}
where $\alpha_k>0$ is a control sequence and $\perturb_k$ is stochastic noise. We will place relevant assumptions on the noise $\perturb_k$ later in Section~\ref{sec:twopillarsmainresults}.  The most important example of~\eqref{alg:perturbedGiteration}, valid for locally Lipschitz functions $f$, is the stochastic subgradient method:
$$
x_{k+1} = x_k - \alpha_k (w_k + \perturb_k) \qquad \text{where } w_k \in \partial_c f(x_k),
$$
In this case,  the mapping $G$ satisfies 
\begin{align}\label{eq:subgradientmethodintro2}
G_{\alpha}(x, \nu) \in \partial_c f(x) + \nu \qquad \text{for all $x, \nu \in \RR^d$ and $\alpha > 0$}.
\end{align}
More generally, $G$ may represent a stochastic projected gradient method or a stochastic proximal gradient method---two algorithms we examine in detail in Section~\ref{sec:examplesoperators}. 

The purpose of this chapter is to understand how iteration~\eqref{alg:perturbedGiteration} is affected by the existence of ``active manifolds" $\cM$ contained within the domain of $f$. For this, we posit a tight interaction between $G$ and the active manifold $\mathcal{M}$, described in the following assumption.
\begin{assumption}[Strong (a) and aiming]\label{assumption:A}
{\rm Fix a point $\bar x \in \dom f$. We suppose that there exist constants $\localconstant, \localmu > 0$, a neighborhood $\cU$ of $\bar x$, and a $C^{3}$ manifold $\cM \subseteq \dom f$ containing $\bar x$ such that the following hold for all $\nu \in \RR^d$ and $\alpha > 0$, where we set $\cU_f := \cU \cap \dom f$.
\begin{enumerate}[label=$\mathrm{(A\arabic*)}$]
\item \label{assumption:localbound} {\bf (Local Boundedness)}
	We have $$\sup_{x \in \cU_f} \|G_\alpha(x, \perturb)\| \leq \localconstant(1+\|\perturb\|).$$
	\item\label{assumption:smoothcompatibility} {\bf (Strong (a))}
The function $f$ is $C^{2}$ on $\cM$ and for all $x \in \cU_f$, we have
\begin{align*}
\|P_{\tangentM{P_{\cM}(x)}}(G_\alpha(x, \perturb) - \nabla_\cM f(P_{\cM}(x)) - \nu)\| \leq C (1 + \|\perturb\|)^2(\dist(x, \cM) +\alpha).
\end{align*}
\item \label{assumption:aiming} {\bf (Proximal Aiming)} For $x \in \cU_f$ tending to $\bar x$, we have 
\begin{align*}
\dotp{G_\alpha(x, \perturb) - \nu, x - P_{\cM}(x)} &\geq \mu \cdot \dist(x, \cM) - (1+\|\perturb\|)^2(o(\dist(x, \cM)) + C\alpha).
\end{align*}
\end{enumerate}}
\end{assumption}

Some comments are in order. Assumption~\ref{assumption:localbound} is similar to classical Lipschitz assumptions and ensures the steplength can only scale linearly in $\|\perturb\|$. 
Assumption~\ref{assumption:smoothcompatibility} is the natural analogue of strong (a) regularity for the operator $G_{\alpha}(x, \perturb)$. It ensures that the shadow sequence $y_k = \proj_{\cM}(x_k)$ locally remains an inexact stochastic Riemannian gradient sequence with implicit retraction.
Assumption~\ref{assumption:aiming} ensures that after subtracting the noise from $G_{\alpha_k}(x_k, \perturb_k)$, the update direction $x_{k+1}-x_{k}$ locally points towards the manifold $\cM$. We will later show that this ensures the iterates $x_k$ approach the manifold $\cM$ at a controlled rate. Finally we note in passing that the power of $(1+\|\perturb\|)$ in the above expressions must be at least 2 for common iterative algorithms to satisfy  Assumption~\ref{assumption:A}; one may also take higher powers, but this requires higher moment bounds on $\|\perturb_k\|$. Before making these results precise in Section~\ref{sec:twopillarsmainresults}, we first formalize our statements about the subgradient method and introduce several examples.

The rest of the section is devoted to examples of algorithms satisfying Assumption~\ref{assumption:A}.

\subsection{Stochastic subgradient method}\label{sec:examplesoperators}

The most immediate example of operator $G$ arises from the subgradient method applied to a locally Lipschitz function $f$. In this setting, any measurable selection $s \colon \RR^d \rightarrow \RR$ of $\partial_c f(x)$ gives rise to a mapping 
\begin{align}\label{eq:subgradientG}
G_{\alpha}(x, \perturb) = s(x) + \perturb,
\end{align}
which is independent of $\alpha$. Then Algorithm~\eqref{alg:perturbedGiteration} is the classical stochastic subgradient method: 
\begin{align}\label{eq:subgradient}
x_{k+1} = x_k - \alpha_k (s(x_k) + \perturb_k).
\end{align}
Let us place the following assumption on $f$, which we will shortly show implies Assumption~\ref{assumption:A}.
\begin{assumption}[Assumptions for the subgradient mapping]\label{assumption:subgradient}
{\rm Let $f \colon \RR^d \rightarrow \RR$ be a function that is locally Lipschitz  continuous around a point $\bar x\in \RR^d$. Let $\cM \subseteq \cX$ be a $C^3$ manifold containing $\bar x$ and suppose that $f$ is $C^2$ on $\cM$ near $x$.
\begin{enumerate}[label=$\mathrm{(B\arabic*)}$]
	\item \label{assumption:subgradient:stronga} {\bf (Strong (a))}
The function $f$ is strongly $(a)$-regular along $\cM$ near $\bar x$. 
\item \label{assumption:subgradient:proximalaiming}{\bf (Proximal aiming)} There exists $\mu > 0$ such that the inequality holds 
\begin{align}\label{prop:subgradient:eq:aiming}
\dotp{v, x - P_{\cM}(x)} \geq \mu\cdot  \dist(x, \cM) \qquad \text{for all $x$ near $\bar x$ and $v \in  \partial_c f(x)$.}
\end{align}
\end{enumerate}}
\end{assumption}

Note that Corollary~\ref{cor:prox-aiming_gen} shows that the aiming condition~\ref{assumption:subgradient:proximalaiming} holds as long as $\cM$ is an active manifold for $f$ at $\bar x$ satisfying $0\in\hat \partial f(\bar x)$ and $f$ is $(b_{\leq})$-regular along $\cM$ near $\bar x$.
The following proposition follows immediately from  Corollary~\ref{cor:prox-aiming_gen}.
\begin{proposition}[Subgradient method]\label{prop:subgradient}
Assumption~\ref{assumption:subgradient} implies  Assumption~\ref{assumption:A} with the map $G$ defined in \eqref{eq:subgradientG}. 
\end{proposition}

Thus, all three properties arise from reasonable assumptions on the function $f$, as discussed in the previous sections. Moreover, for definable functions, they hold generically, as the following corollary shows. Indeed, this is a direct consequence of Theorem~\ref{thm:genericcomposite}.
\begin{cor}\label{prop:subgradientsemialgebraic}
Suppose that $f\colon \RR^d \rightarrow \RR$ is locally Lipschitz and definable in o-minimal structure. Then there exists a finite $N$ such that for a generic set of $v\in \RR^d$ the tilted function $f_{v}(x) := f(x) - \dotp{v, x}$  has at most $N$ Clarke critical points. Moreover, each limiting critical point $\bar x$ is in fact Fr{\'e}chet critical and satisfies the following.
\begin{enumerate}
\item The function $f$ and the subgradient mapping~\eqref{eq:subgradient} satisfy Assumption~\ref{assumption:A} at $\bar x$ with respect to some $C^3$ active manifold $\cM$.
\item The limiting critical point $\bar x$ is either a local minimizer or an active strict saddle point of $f$.
\end{enumerate}
\end{cor}

\subsection{Stochastic projected subgradient method}

Throughout this section let $g \colon \RR^d \rightarrow \RR$ be a locally Lipschitz function and let $\cX$ be a closed set and consider the constrained minimization problem 
$$
\min f(x) := g(x) + \delta_{\cX}(x).
$$
A classical algorithm for solving this problem is known as the stochastic projected subgradient method. Each iteration of the method updates
\begin{align}\label{eq:stochasticprojectedsubgradient}
x_{k+1} \in P_{\cX}(x_k - \alpha_k (v_k + \perturb_k)) \qquad \text{ where $v_k \in \partial_c g(x_k)$}
\end{align}
This algorithm can be reformulated as an instance of~\eqref{alg:perturbedGiteration}. Indeed, 
let $s_\cX \colon \RR^d \rightarrow \RR^d$ be a measurable selection of $P_{\cX}$, let $s_g \colon \RR^d \rightarrow \RR^d$ be a measurable selection of $\partial_c g$, and define the generalized gradient mapping
\begin{align}\label{eq:projectedsubgradientG}
G_{\alpha}(x, \perturb) := \frac{x - s_\cX(x - \alpha( s_g(x) + \perturb))}{\alpha} \qquad \text{for all $x \in \RR^d$, $\nu \in \RR^d$, $\alpha > 0$.}
\end{align}
Evidently, the update rule~\eqref{alg:perturbedGiteration} reduces to \eqref{eq:stochasticprojectedsubgradient}.

In order to ensure Assumption~\ref{assumption:A} for the stochastic projected subgradient method, we introduce the following assumptions on $g$ and $\cX$.
\begin{assumption}[Assumptions for the projected gradient mapping]\label{assumption:projectedgradient}
{\rm Let $f := g + \delta_{\cX}$, where $\cX$ is a closed set and $g \colon \RR^d \rightarrow \RR$ is a locally Lipschitz continuous function. Fix $\bar x \in \RR^d$ and let $\cM \subseteq \cX$ be a $C^3$ manifold containing $\bar x$ and suppose that $f$ is $C^2$ on $\cM$ near $\bar x$.
\begin{enumerate}[label=$\mathrm{(C\arabic*)}$]
	\item \label{assumption:projectedgradient:stronga} {\bf (Strong (a))}
The function $g$ and set $\cX$ are strongly $(a)$-regular along $\cM$ at $\bar x$. 
\item \label{assumption:projectedgradient:proximalaiming}{\bf (Proximal aiming)} There exists $\mu > 0$ such that the inequality holds 
\begin{align}\label{prop:projectedgradient:eq:aiming}
\dotp{v, x - P_{\cM}(x)} \geq \mu\cdot  \dist(x, \cM) \qquad \text{for all $x \in \cX$ near $\bar x$ and $v \in \partial_c g(x)$.}
\end{align}
\item\label{assumption:projectedgradient:bproxregularity} {\bf (Condition (b))} The set $\cX$ is $(b_{\leq})$-regular along $\cM$ at $\bar x$.\end{enumerate}}
\end{assumption}
Note that Corollary~\ref{cor:prox-aiming_gen} shows that the aiming condition~\ref{assumption:projectedgradient:proximalaiming} holds as long as $\cM$ is an active manifold for $f$ at $\bar x$ satisfying $0\in\hat \partial f(\bar x)$ and $f$ is $(b_{\leq})$-regular along $\cM$ at $\bar x$.\footnote{Corollary~\ref{cor:prox-aiming_gen} shows that there exists a constant $c>0$ such that for any $\delta>0$, the estimate
	\begin{equation}\label{eqn:prox_aim_super_dumb}
	\langle v,x-P_{\cM}(x)\rangle\geq (c-\delta\sqrt{1+\|v\|^2})\cdot\dist(x,\cM),
	\end{equation}
	holds for all $x \in \cX$ near $\bar x$ and for all $v\in \partial f(x)$. In particular, due to the inclusion $\hat \partial g(x)+\hat N_{\cX}(x)\subset \hat \partial f(x)$, we may choose any $v\in \hat \partial g(x)$ in \eqref{eqn:prox_aim_super_dumb}. Therefore, taking into account that $g$ is locally Lipschitz, we deduce that there is a constant $\mu$ such that $\langle v,x-P_{\cM}(x)\rangle\geq \mu\cdot\dist(x,\cM)$ for all $x \in \cX$ near $\bar x$ and for all $v\in \hat \partial g(x)$. Taking limits and convex hulls, the same statement holds for all $v\in \partial_c g(x)$.} The following proposition shows that Assumption~\ref{assumption:projectedgradient} is sufficient to ensure Assumption~\ref{assumption:A}; we defer the proof to Appendix~\ref{section:proof:prop:projectedgradient} since it's fairly long.
\begin{proposition}[Projected subgradient method]\label{prop:projectedgradient}
Assumption~\ref{assumption:projectedgradient} implies Assumption~\ref{assumption:A} for the map $G$ defined in \eqref{eq:projectedsubgradientG}.
\end{proposition}

Given this proposition, an immediate question is whether Assumption~\ref{assumption:projectedgradient} holds generically under for problems that are definable in an o-minimal structure. The following corollary, which is an immediate consequence of Proposition~\ref{prop:projectedgradient}, Theorem~\ref{thm:genericcomposite}, and Corollary~\ref{cor:prox-aiming_gen}, shows that the answer is yes.

\begin{cor}\label{prop:projectedgradientsemialgebraic}
Suppose that $f = g + \delta_\cX$, where $\cX \subseteq \RR^d$ is closed and $g\colon \RR^d \rightarrow \RR$ is locally Lipschitz, and both $\cX$ and $g$ are definable in an o-minimal structure. Then there exists a finite $N$ such that for a generic set of $v, w \in \RR^d$ the tilted function $f_{v, w}(x) := g(x + w) + \delta_{\cX}(x) - \dotp{v, x}$ has at most $N$  composite Clarke critical points. Moreover, each composite limiting critical point $\bar x$ is in fact Fr{\'e}chet critical and satisfies the following,
\begin{enumerate}
\item The function $f$ and the projected subgradient mapping $G$ define in \eqref{eq:projectedsubgradientG} satisfy Assumption~\ref{assumption:A} at $\bar x$ with respect to some $C^3$ active manifold $\cM$.
\item The composite limiting critical point $\bar x$ is either a local minimizer or an active strict saddle point of $f$.
\end{enumerate}
\end{cor}
In the above corollary, the qualification \emph{composite critical points}, as defined in Theorem~\ref{thm:genericcomposite}, is important, since the projected subgradient method is only known to converge to such points.

\subsection{Proximal gradient method}

Throughout this section let $g \colon \RR^d \rightarrow \RR$ be a $C^1$ function and let $h \colon \RR^d \rightarrow \RR \cup \{+\infty\}$ be a closed function. We then consider the minimization problem 
$$
\min_{x\in\RR^d} f(x) := g(x) + h(x) .
$$
A classical algorithm for solving this problem is the stochastic proximal gradient method. Each iteration of the method solves the proximal problem:
\begin{equation}\label{eqn:prox_grad}
x_{k+1} \in \argmin_{x \in \RR^d}\left\{ h(x) + \dotp{\nabla g(x_k) + \perturb_k, x - x_k} + \frac{1}{2\alpha_k}\|x - x_k\|^2\right\}.
\end{equation}
This algorithm can be reformulated as an instance of~\eqref{alg:perturbedGiteration}. Indeed, 
let $s \colon \RR_{++} \times \RR^d \rightarrow \RR^d$ be a measurable selection of the proximal map
$(x,\alpha)\mapsto \argmin_y\{h(y)+\frac{1}{2\alpha}\|y-x\|^2\}$ and consider the mapping $G$ defined by
\begin{align}\label{eq:proximalgradientG}
G_{\alpha}(x, \perturb) = \frac{x - s_{\alpha}(x - \alpha(\nabla g(x) + \perturb))}{\alpha} \qquad \text{for all $x \in \RR^d$, $\nu\in \RR^d$ and $\alpha > 0$}.
\end{align}
Evidently, the update rule~\eqref{alg:perturbedGiteration} is equivalent to \eqref{eqn:prox_grad}.

In order to ensure Assumption~\ref{assumption:A} for the stochastic proximal gradient method, we introduce the following assumptions on $g$ and $h$.
\begin{assumption}[Assumptions for the proximal gradient mapping]\label{assumption:proximalgradient}
{\rm Let  $f := g + h$, where $g \colon \RR^d \rightarrow \RR$ is $C^1$ and $h \colon \RR^d \rightarrow \RR \cup \{+\infty\}$ is closed. Denote $\cX := \dom h$ and let $\cM \subseteq \cX$ be a $C^3$ manifold containing some point $\bar x$ and suppose that $f$ is $C^2$ on $\cM$ near $\bar x$.
\begin{enumerate}[label=$\mathrm{(D\arabic*)}$]
\item \label{assumption:proximalgradient:Lipschitz1} {\bf (Lipschitz gradient/boundedness)}
	The gradient $\nabla g$ Lipschitz near $\bar x$. Moreover, there exists $C > 0$ such that $\|\nabla g(x)\| \leq C(1 + \|x\|)$ for all $x \in \cX$.
\item \label{assumption:proximalgradient:Lipschitz2} {\bf (Lipschitz proximal term)} The function $h$ is Lipschitz on $\cX$.
	\item \label{assumption:proximalgradient:stronga}{\bf (Strong (a))}
The function $h$ is strongly $(a)$-regular along $\cM$ at $\bar x$. 
\item \label{assumption:proximalgradient:proximalaiming} {\bf (Proximal Aiming)} There exists $\mu > 0$ such that the inequality holds 
\begin{align}\label{prop:proximalgradient:eq:aiming}
\dotp{v, x - P_{\cM}(x)} \geq \mu\cdot  \dist(x, \cM) - (1+\|v\|)o(\dist(x, \cM))
\end{align}
for all $x \in \dom h$ near $\bar x$ and $v \in \partial f(x)$.
\end{enumerate}}
\end{assumption}

Note that Corollary~\ref{cor:prox-aiming_gen} shows that the aiming condition~\ref{assumption:proximalgradient:proximalaiming} holds as long as $\cM$ is an active manifold for $f$ at $\bar x$ satisfying $0\in\hat \partial f(\bar x)$ and $f$ is $(b_{\leq})$-regular along $\cM$ at $\bar x$.
The following proposition shows that Assumption~\ref{assumption:proximalgradient} is sufficient to ensure Assumption~\ref{assumption:A}. The proof of the Proposition appears in Appendix~\ref{section:proof:prop:proximalgradient}
\begin{proposition}[Proximal gradient method]\label{prop:proximalgradient}
If assumption~\ref{assumption:proximalgradient} holds at $\bar x \in \dom f$, then $f$ and $G$ satisfy Assumption~\ref{assumption:A} at $\bar x$.
\end{proposition}

The following corollary, which is an immediate consequence of Proposition~\ref{prop:proximalgradient} and Theorem~\ref{thm:genericfunctionsregular}, shows that assumption~\ref{assumption:proximalgradient} is automatically true for definable problems.

\begin{cor}\label{prop:proximalgradientsemialgebraic}
Suppose that $f = g + h_0 + \delta_\cX$, where $\cX \subseteq \RR^d$, $g$ is a $C^1$ function with Lipschitz gradient, the function $h_0 \colon \RR^d\rightarrow \RR$ is Lipschitz on $\cX$, and we define $h := h_0 + \delta_{\cX}$.
Suppose that $g, h_0,$ and $\cX$ are  definable in an o-minimal structure. Then there exists a finite $N$ such that for a full measure set of $v, w \in \RR^d$, the tilted function $f_{v, w} := g(x+w) +h_0(x+w) +  \delta(x) - \dotp{v, x} $ has at most $N$  composite Clarke critical points $\bar x$. Moreover, each composite limiting critical point $\bar x$ is in fact composite Fr{\'e}chet critical and satisfies the following.
\begin{enumerate}
\item The function $f$ and the proximal gradient mapping~\eqref{eq:proximalgradientG} satisfy Assumption~\ref{assumption:A} at $\bar x$ with respect to some active manifold $\cM$.
\item The critical point $\bar x$ is either a local minimizer or an active strict saddle point of $f$.
\end{enumerate}
\end{cor}

Thus, we find that Assumption~\ref{assumption:A} is satisfied for common iterative mappings, under reasonable assumptions, and is even automatic for certain generic classes of functions. In the next several sections, we turn our attention to the algorithmic consequences of theses assumptions.

\section{The two pillars}\label{sec:twopillarsmainresults}

Assumption~\ref{assumption:A} at a point $\bar x$ guarantees two useful behaviors, provided the iterates $\{x_k\}$ of iteration~\eqref{alg:perturbedGiteration} remain in a small ball around $\bar x$. First $x_k$ must approach the manifold $\cM$ containing $\bar x$ at a controlled rate, a consequence of the proximal aiming condition. Second the shadow $y_k = P_{\cM}(x_k)$ of the iterates along the manifold form an approximate Riemannian stochastic gradient sequence with an implicit retraction. Moreover, the approximation error of the sequence decays with $\dist(x_k, \cM)$ and $\alpha_k$, quantities that quickly tend to zero.

The formal statements of our results crucially require local arguments and frequently refer to the following stopping time:
given an index $k \geq 1$ and a constant $\delta > 0$, define
\begin{align}\label{def:stoppingtime}
\tau_{\discrete, \delta} := \inf\{j \geq k \colon x_j  \notin B_{\delta}(\bar x)\}.
\end{align}
Note that the stopping time implicitly depends on $\bar x$, a point at which Assumption~\ref{assumption:A} is satisfied. In the statements of our result, the point $\bar x$ will always be clear from the context. Second, we make the following standing assumption on $\alpha_k$ and $\perturb_k$. We assume they are in force throughout the rest of the sections.

\begin{assumption}[Standing assumptions]\label{assumption:zero}
{\rm~Assume the following.
\begin{enumerate}[label=$\mathrm{(E\arabic*)}$]
\item The map $G$ is measurable.
\item There exist constants $c_1, c_2 > 0$ and $\gamma \in (1/2, 1]$ such that 
$$
\frac{c_1}{k^\gamma} \leq \alpha_k \leq \frac{c_2}{k^\gamma}.
$$
\item $\{\perturb_k\}$ is a martingale difference sequence w.r.t.\ to the increasing sequence of $\sigma$-fields 
$$
\cF_k = \sigma(x_j \colon j \leq k \text{ and } \perturb_j  \colon j<k),
$$
and there exists a function $q \colon \RR^d \rightarrow \RR_+$ that is bounded on bounded sets with 
$$
\EE[\perturb_k \mid \cF_k] = 0 \qquad \text{ and } \qquad  \EE[\|\perturb_k\|^4\mid \cF_k]  < q(x_k).
$$
We let $\EE_k[\cdot ] = \EE[ \cdot \mid \cF_k]$ denote the conditional expectation.
\item The inclusion $x_k \in \dom f$ holds for all $k \geq 1$.
\end{enumerate}}
\end{assumption}
All items in Assumption~\ref{assumption:zero} are standard in the literature on stochastic approximation methods and mirror those found in~\cite[Assumption C]{davis2020stochastic}. The only exception is the fourth moment bound on $\|\perturb_k\|$, which stipulates that $\nu_k$ has slightly lighter tails. This bound appears to be necessary for the setting we consider. We now turn to the first pillar.

\subsection{Pillar I: Aiming towards the manifold}

The following proposition ensures the sequence $x_k$ approaches the manifold. 
The proof appears in Section~\ref{sec:proof:prop:gettingclosertothemanifold}.
\begin{proposition}\label{prop:gettingclosertothemanifold}
Suppose that $f$ satisfies Assumption~\ref{assumption:A} at $\bar x$. Let $\gamma \in (1/2, 1]$ and assume $c_1 \geq 32/\mu$ if $\gamma = 1$. Then for all $k_0 \geq 1$ and sufficiently small $\delta > 0$, there exists a constant $C$, such that the following hold with stopping time $\tau_{k_0, \delta}$ defined in~\eqref{def:stoppingtime}: 
\begin{enumerate}
\item \label{eq:prop:gettingclosertothemanifoldbound1}  There exists a random variable $V_{k_0, \delta}$ such that 
\begin{enumerate}
 \item \label{eq:prop:gettingclosertothemanifoldbound1:as} The limit holds: $$\frac{k^{2\gamma - 1}}{\log(k+1)^2}\dist^2(x_{k}, \cM)1_{\tau_{k_0, \delta} > k} \xrightarrow{\text{a.s.}}  V_{k_0, \delta}.$$
\item \label{eq:prop:gettingclosertothemanifoldbound1:sum} The sum is almost surely finite: $$\sum_{k=1}^\infty \frac{k^{\gamma - 1}}{\log(k+1)^2}\dist(x_{k}, \cM)1_{\tau_{k_0, \delta} > k} < +\infty.$$ 
\end{enumerate} 
\item \label{eq:prop:gettingclosertothemanifoldbound1p5}  We have
\begin{enumerate}
\item \label{eq:prop:gettingclosertothemanifoldbound1p5:expectedsquareddistance}The expected squared distance satisfies:
 $$
\EE[\dist^2(x_k, \cM)1_{\tau_{k_0, \delta} > k}]  \leq C\alpha_k \qquad \text{for all $k \geq 1$}.
$$
\item \label{eq:prop:gettingclosertothemanifoldbound1p5:saddle} The tail sum  is bounded:
$$
\EE\left[  \sum_{i=k}^\infty  \alpha_i \dist(x_i , \cM)1_{\tau_{k_0, \delta} > i}\right] \leq C\sum_{i=k}^\infty \alpha_i^2 \qquad \text{for all $k \geq 1$.}
$$
\end{enumerate}
\end{enumerate}
\end{proposition}

We note that Part~\ref{eq:prop:gettingclosertothemanifoldbound1:sum} of the proposition holds not only almost surely, but also in expectation, which is  a stronger statement in general. Now we turn our attention to Pillar II: the shadow iteration.

\subsection{Pillar II: The shadow iteration}

Next we study the evolution of  the shadow $y_k = P_{\cM}(x_k)$  along the manifold, showing that $y_k$ is locally an inexact Riemannian stochastic gradient sequence with error that asymptotically decays as $x_k$ approaches the manifold. Consequently, we may control the error using Proposition~\ref{prop:gettingclosertothemanifold}. The proof appears in Section~\ref{sec:proof:prop:shadow}
\begin{proposition}\label{prop:shadow}
Suppose that $f$ satisfies Assumption~\ref{assumption:A} at $\bar x$. Then for all $k_0 \geq 1$ and sufficiently small $\delta > 0$, there exists a constant $C$, such that the following hold with stopping time $\tau_{k_0, \delta}$ defined in~\eqref{def:stoppingtime}: there exists a sequence of $\cF_{k+1}$-measurable random vectors $E_k \in \RR^d$ such that
\begin{enumerate}
\item \label{prop:shadow:part:sequence} The shadow sequence 
\begin{align*}
y_k = \begin{cases}
P_{\cM}(x_k) & \text{if $x_k \in B_{2\delta}(\bar x)$} \\
\bar x &  \text{otherwise.}
\end{cases}
\end{align*}
satisfies $y_k \in B_{4\delta}(\bar x) \cap \cM$ for all $k$ and the recursion holds:
\begin{align}\label{eqn:shadow:eq:iteration}
\boxed{y_{k+1} = y_k - \alpha_k\nabla_{\cM} f(y_k) - \alpha_k P_{\tangentM{y_k}}(\perturb_k) + \alpha_k E_k \qquad \text{for all $k \geq 1.$}} 
\end{align}
Moreover, for such $k$, we have $\EE_k[P_{\tangentM{y_k}}(\perturb_k)] = 0$.
\item \label{prop:shadow:part:error}  Let $\gamma \in (1/2, 1]$ and assume that $c_1 \geq 32/\mu$ if $\gamma = 1$.  
\begin{enumerate}
\item\label{prop:shadow:part:error:part:upperbound} We have the following bounds for $k_0 \leq k \leq \tau_{k_0, \delta} -1$:
\begin{enumerate}
\item \label{prop:shadow:part:error:part:upperbound:1} 
$\|E_k\|1_{\tau_{k_0, \delta} > k} \leq C(1+\|\perturb_k\|)^2(\dist(x_k, \cM) + \alpha_k)1_{\tau_{k_0, \delta} > k} $
\item \label{prop:shadow:part:error:part:upperbound:2}$\max\{\EE_k [\|E_k\|]1_{\tau_{k_0, \delta} > k}, \EE_k [\|E_k\|^2]1_{\tau_{k_0, \delta} > k}\} \leq C$.
\item \label{prop:shadow:part:error:part:upperbound:3}$\EE[\|E_k\|^2]1_{\tau_{k_0, \delta} > k} \leq C\alpha_k$
\end{enumerate}
\item \label{prop:shadow:part:error:part:random} The following sums are finite
\begin{enumerate}
\item \label{prop:shadow:part:error:part:random:1} $\sum_{k=1}^\infty \frac{k^{\gamma - 1}}{\log(k+1)^2}\max\{\|E_k\|1_{\tau_{k_0, \delta} > k}, \EE_k[\|E_k\|]1_{\tau_{k_0, \delta} > k}\}  < +\infty$
\item \label{prop:shadow:part:error:part:random:2}$\sum_{k=1}^\infty \frac{k^{\gamma - 1}}{\log(k+1)^2}\max\{\|E_k\|^21_{\tau_{k_0, \delta} > k}, \EE_k[\|E_k\|^2]1_{\tau_{k_0, \delta} > k}\}  < +\infty$
\end{enumerate}
\item \label{prop:shadow:part:error:part:random:saddle} The tail sum is bounded
$$
\EE\left[1_{\tau_{k_0, \delta} = \infty}  \sum_{i=k}^\infty  \alpha_i\|E_k\|\right] \leq C\sum_{i=k}^\infty \alpha_i^2 \qquad \text{for all $k \geq 1$}.
$$

\end{enumerate}
\end{enumerate}
\end{proposition}

With the two pillars we separate our study of the sequence $x_k$ into two orthogonal components: In the tangent/smooth directions, we study the sequence $y_k$, which arises from an inexact gradient method with rapidly decaying errors and is amenable to the techniques of smooth optimization. In the normal/nonsmooth directions, we steadily approach the manifold, allowing us to infer strong properties of $x_k$ from corresponding properties for $y_k$.

\section{Avoiding saddle points}\label{sec:avoidingsaddlepointschapter3}

In this section, we ask whether $x_k$ can converge to points $\bar x$ at which $\nabla^2_{\cM}f(\bar x)$ has at least one strictly negative eigenvalue. We call such points \emph{strict saddle points}, and when $\cM$ is in addition an active manifold for $f$, then we call such points \emph{active strict saddle points}.
We use a well-known technique in the stochastic approximation literature: isotropic noise injection~\cite{pemantle1990nonconvergence,AIHPB_1996__32_3_395_0,benaim1996dynamical,benaim1999dynamics}.

Let us briefly describe this technique. Fix a point $p \in\RR^d$ and consider a $C^2$ mapping $F_p \colon \RR^d \rightarrow \RR^d$ with an unstable zero at $p$, meaning $\nabla F_p(p)$ has an eigenvalue with a strictly positive real part. Then a well-known result of Pemantle~\cite{pemantle1990nonconvergence}  states that, with probability 1, the following perturbed iteration cannot converge to $p$:
\begin{align}\label{eq:pemantleiterationintro}\left\{
	\begin{aligned}
	&\textrm{Sample } \xi_\discrete \sim \unif(B_1(0))\\
	& \textrm{Set } Y_{k+1} = Y_k + \alpha_k F_p(Y_k) + \alpha_k \xi_k 
	\end{aligned}\right\}.
\end{align}
As stated, the result of~\cite{pemantle1990nonconvergence} does not shed light on the iteration~\eqref{alg:perturbedGiteration}. Nevertheless, in light of \eqref{eqn:shadow:eq:iteration}, the shadow iteration $y_k$ does satisfy an iteration similar to~\eqref{eq:pemantleiterationintro} with mapping 
$$
F_p(y) = -\nabla (f \circ P_{\cM})(y),
$$
which under reasonable assumptions is locally $C^2$ near $p$ and satisfies $F_p(y) = -\nabla_{\cM} f(y)$ and $\nabla F_p(y) = -\nabla_\cM^2 f(y)$ for all $y \in \cM$ near $p$. Moreover, if $p$ is an active strict saddle of $f$, then  $\nabla_{\cM}^2 f(p)$ has a strictly negative eigenvalue, so $p$ is an ``unstable zero" of $F_p$. Thus, we might reasonably expect $y_k$ to converge to $p$ only with probability zero. If this is the case, we can then lift the argument to $x_k$, showing that if $x_k$ converges to $p$, then so does $y_k$---a probability zero event. This is the strategy we will apply in what follows, taking into account the additional error term $E_k$ in the shadow iteration~\eqref{eqn:shadow:eq:iteration}, a key technical issue that we have so far ignored.

In order to formalize the above strategy, we prove the following extension of the main result of~\cite{pemantle1990nonconvergence} which takes into account the relationship between $x_k$ and $y_k$ described above. The proof, which we defer to Section~\ref{sec:proofofpemantlething}, draws on  the techniques of~\cite{pemantle1990nonconvergence,brandiere1998some,AIHPB_1996__32_3_395_0,benaim1996dynamical,benaim1999dynamics}.

\begin{thm}[Nonconvergence]\label{thm: nonconvergence to unstable points}
Fix $c_1, c_2 > 0$ and let $S \subseteq \RR^d$. Suppose for any $p\in S$, there exists a ball $\ball{p}{\epsilon_p}$ centered at $p$ and a $C^2$ mapping $F_p \colon \ball{p}{\epsilon_p} \rightarrow \RR^d$ that vanishes at $p$ and has a symmetric Jacobian $\nabla F_p(p)$ that has at least one positive eigenvalue. Suppose $\{ X_\discrete\}_{\discrete=1}^{\infty}$ is a stochastic process and for any $\discrete_0$, $p \in S$, and $\delta > 0$ define the stopping time: 
$$
\tau_{\discrete_0, \delta}(p) = \inf\left\{ \discrete \geq \discrete_0 \colon X_\discrete \notin \ball{p}{\delta}\right\}.
$$
Suppose that for any $p \in S$, $k_0 \geq 1$, and all sufficiently small $\delta_p \leq \epsilon_p$ the following hold: there exists $c_3,c_4 > 0$ possibly depending on $p$, but not on $\delta_p$ and $\epsilon_p$, such that on the event $\Omega_{0} = \{\tau_{k_0, \delta_p}(p) = \infty\}$, we have
\begin{enumerate}
\item\label{item:localiteration} \textbf{(Local iteration.)} There exists a process $\{Y_\discrete\colon  \discrete \ge \discrete_0\} \subseteq \ball{p}{\epsilon_p/2}$ satisfying
\begin{equation}\label{eqn: dynamics}
	Y_{\discrete+1} = Y_\discrete + \stepsize_\discrete F_p(Y_\discrete) +  \stepsize_\discrete \xi_\discrete + \stepsize_\discrete \error_\discrete
\end{equation}
for error sequence $\{\error_\discrete\}$, noise sequence $\{\xi_\discrete\}$, and deterministic stepsize sequence $\{\stepsize_\discrete\}$ that are square summable, but not summable.
\item\label{item:noise} \textbf{(Noise Conditions.)} Let $\cF_k$ be the sigma algebra generated by $X_{\discrete_0}, \ldots, X_{\discrete}$ and $Y_{\discrete_0}, \ldots, Y_{\discrete}$. Define $W_p$ to be the subspace of eigenvectors of $\nabla F_p(p)$ with positive eigenvalues. Then we have
\begin{enumerate}
\item $\expect{\xi_\discrete \mid \cF_\discrete} = 0$.
\item \label{item:noise:ub} $\limsup_{k}\EE[\|\xi_k\|^4 \mid \cF_k] \leq c_3$.
\item \label{item:noise:lb} $
\expect{|\dotp{\xi_\discrete, w}| \mid \cF_\discrete} \ge c_4  \qquad \text{for $k \geq \discrete_0$ and all unit norm $w \in W_p$.}
$
\end{enumerate}
\item \label{item:error} \textbf{(Error Conditions.)} 
\begin{enumerate}
\item We have $\limsup_{k} \EE[1_{\Omega_0}\|E_k\|^4 \mid\cF_k] < \infty$. 
\item \label{label:stochasticprocess:averaging} For all $n \geq k_0$, we have $\EE\left[1_{\Omega_0}  \sum_{k=n}^\infty  \alpha_k \|E_{k}\|\right] = O_{k_0}\left(\sum_{k=n}^\infty \alpha_k^2\right)$.
\end{enumerate}
\end{enumerate}
	Then $P(\lim_{\discrete\rightarrow \infty} X_\discrete \in S)=0$.
\end{thm}

Looking at the theorem, recursion condition~\eqref{eqn: dynamics} is clearly modeled on the shadow sequence of Proposition~\ref{prop:shadow}. Moreover, the error condition~\ref{label:stochasticprocess:averaging} on $E_k$ precisely matches~\ref{prop:shadow:part:error:part:random:saddle}. Finally, the noise $\xi_k$ is modeled on $P_{\tangentM{y_k}}(\perturb_k)$ in the shadow iteration, which is mean zero and has bounded fourth moment. Condition~\ref{item:noise:lb} is not automatic for all noise distributions and requires that $\perturb_k$ has nontrivial mass in all directions of negative curvature for $f$.

Given Theorem~\ref{thm: nonconvergence to unstable points}, we now ask: can $x_k$ converge to critical points $\bar x$ at which $\nabla^2_{\cM} f(\bar x)$ has a strict negative eigenvalue? In the following theorem we show that the answer is no, provided that we choose the noise $\perturb_k$ according to the following assumption: 
\begin{assumption}[Uniform noise]\label{assumption:uniform}
{\rm There exists $r > 0$ such that $\nu_k \sim \text{Unif}(B_r(0))$ for all $k$.}
\end{assumption}
The proof of the theorem appears in Section~\ref{section:proof:thm:avoidance}.
 
\begin{thm}[Nonconvergence to strict saddle point]\label{thm:avoidance}
Let $S \subseteq \RR^d$ and suppose that Assumption~\ref{assumption:A} holds at each point $\bar x \in S$, where each manifold is $C^4$. Let $\cM$ be the manifold associated to a point $\bar x\in S$ and suppose that $\nabla_\cM^2f(\bar x)$ has a strictly negative eigenvalue.  Suppose that $\nu_k$ satisfies Assumption~\ref{assumption:uniform}. In addition, suppose that  $\gamma \in (\frac{1}{2},1)$. Then 
	\begin{equation}
		P\left(\lim_{\discrete\rightarrow \infty} x_\discrete \in S\right) =0.
	\end{equation} 
\end{thm}

Note that the theorem applies to arbitrary sets $S$, making no assumptions on countability/isolatedness. Second the result does not preclude the {\em limit points} of $x_k$ from lying in $S$. Thus, the result is useful only when $x_k$ is known to converge. 

We now examine two applications of the above theorem for the projected and proximal subgradient methods. The following corollary provides sufficient conditions for the projected subgradient method to avoid active strict saddle points. We place the proof in Appendix~\ref{appendix:cor:activestrictsaddleregular}.
\begin{cor}[Projected subgradient methods]\label{cor:activestrictsaddleregular}
Suppose that $f = g + \delta_{\cX}$, where $g\colon \RR^d \rightarrow \RR$ is locally Lipschitz and $\cX \subseteq \RR^d$ is closed. Let $S$ consist of points $x$ satisfying $0\in \hat{\partial} f(x)$ and that are $C^4$ active strict saddle points of $f$. Suppose the following hold for all $x \in S$ with associated active manifold $\cM_x$: \begin{enumerate}
\item The function $g$ and the set $\cX$ are strongly $(a)$-regular along $\cM_x$ at $x$.
\item The function $g$ is weakly convex around $x$ or $(b_{\leq})$-regular along $\cM_x$ at $x$.
\item The set $\cX$ is prox-regular at $x$ or $(b_{\leq})$-regular along $\cM_x$ at $x$.
\end{enumerate}
Suppose that $\nu_k$ satisfies Assumption~\ref{assumption:uniform}. Then the iterates of the stochastic projected subgradient method~\eqref{eq:stochasticprojectedsubgradient} satisfy 
$$
P\left(\lim_{k\rightarrow \infty} x_k \in S\right) = 0.
$$
\end{cor}

Next we analyze the the proximal gradient method. Recall that the paper~\cite{davis2021proximal} showed that randomly initialized proximal gradient methods avoid active strict saddles of weakly convex functions. The following Corollary shows that the same behavior holds for perturbed proximal gradient methods beyond the weakly convex class. We place the proof in Appendix~\ref{appendix:cor:activestrictsaddleregularproximalgradient}.
\begin{cor}[Proximal gradient methods]\label{cor:activestrictsaddleregularproximalgradient}
Suppose that $f = g + h$, where $h \colon \RR^d \rightarrow \RR\cup \{\infty\}$ is closed and Lipschitz on its domain $\cX := \dom h$ and $g\colon \RR^d \rightarrow \RR$ is $C^1$ with Lipschitz continuous gradient on $\cX$. Let $S$ consist of points $x$ satisfying $0\in \hat{\partial} f(x)$ and that are $C^4$ active strict saddle points of $f$. Suppose that for all $x \in S$ with associated active manifold $\cM_x$, the function $f$ is strong (a)-regular and $(b_{\leq})$-regular along $\cM_x$ at $x$. Suppose that $\nu_k$ satisfies Assumption~\ref{assumption:uniform}. Then the iterates of the stochastic proximal gradient method~\eqref{eq:stochasticprojectedsubgradient}  satisfy 
$$
P\left(\lim_{k\rightarrow \infty} x_k \in S\right) = 0.
$$
\end{cor}

\subsection{Consequences for generic semialgebraic functions}
The results we have presented so far show that the perturbed projected subgradient and the proximal gradient method cannot converge to Fr{\'e}chet active strict saddle points, provided that $x_k$ converges and various regularity properties hold. Although the convergence of $x_k$ and the required regularity properties may seem stringent, they are in a precise sense generic. Indeed, the genericity of the regularity properties was already addressed in Section~\ref{sec:gen_reg}. Convergence also holds generically: it is known that all limit points of the stochastic subgradient method, the stochastic projected subgradient method, and the stochastic proximal method are (composite) Clarke critical points, as long as $f$ is a semialgebraic function~\cite[Corollary 6.4.]{davis2020stochastic}. Thus, since generic semialgebraic functions have only finitely many (composite) Clarke critical points and one can show (with small effort) that the set of limit points of each  algorithm is connected, it follows that the entire sequence $x_k$ must converge on generic problems (if the sequence remains bounded). Thus we have the following three corollaries, whose proofs we place in Appendix~\ref{sec:genericsaddleavoidance}.

\begin{cor}[Subgradient method on generic semialgebraic functions]\label{thm:globalsemialgebraic}
Let $f \colon \RR^d \rightarrow \RR$ be a locally Lipschitz semialgebraic function. Then for a full measure set of $v$ the following is true for the tilted function $f_v(x) := f(x) - \dotp{v, x}$: Let $\{x_k\}_{k \in \NN}$ be generated by the subgradient method~\ref{eq:subgradient} on $f_v$. Suppose that $\nu_k$ satisfies Assumption~\ref{assumption:uniform}. Then 
on the event $\{x_k\}_{k \in \NN}$ is bounded, almost surely we have only two possibilities 
\begin{enumerate}
\item $x_k$ converges to a local minimizer $\bar x$ of $f_v$.
\item $x_k$ converges to a Clarke critical point of $f_{v}$
\end{enumerate}
Thus, if $f$ is Clarke regular, the sequence $x_k$ must converge to a local minimizer of $f_v$. 
\end{cor}

\begin{cor}[Projected subgradient method on generic semialgebraic functions]\label{thm:globalsemialgebraic2}
Let $f = g+ \delta_{\cX}$, where $\cX \subseteq \RR^d$ semialgebraic and closed and $g\colon \RR^d \rightarrow \RR$ is locally Lipschitz and semialgebraic. Then for a full measure set of $v, w \in \RR^d$ the following is true for the tilted function $f_{v, w}(x) := g(x + w) + \delta_{\cX}(x) - \dotp{v, x}$. Let $\{x_k\}_{k \in \NN}$ be generated by the projected subgradient method~\ref{eq:projectedsubgradientG}. Suppose that $\nu_k$ satisfies Assumption~\ref{assumption:uniform}. Then 
on the event $\{x_k\}_{k \in \NN}$ is bounded, almost surely we have only two possibilities 
\begin{enumerate}
\item $x_k$ converges to a local minimizer $\bar x$ of $f_{v,w}$.
\item $x_k$ converges to a composite Clarke critical point of $f_{v, w}$.
\end{enumerate}
Thus, if $g$ and $\cX$ are Clarke regular, the sequence $x_k$ converges to a local minimizer of $f_{v, w}$. 
\end{cor}

\begin{cor}[Proximal gradient method on generic semialgebraic functions]\label{thm:globalsemialgebraic3}
Suppose that $f = g + h_0 + \delta_\cX$, where $\cX \subseteq \RR^d$, $g$ is a $C^1$ function with Lipschitz gradient on $\cX$, the function $h_0 \colon \RR^d\rightarrow \RR$ is Lipschitz on $\cX$, and we define $h := h_0 + \delta_{\cX}$.
Then for a full measure set of $v, w \in \RR^d$ the following is true for the tilted function $f_{v, w} := g(x+w) +h_0(x+w) + \delta(x) - \dotp{v, x}  $. Let $\{x_k\}_{k \in \NN}$ be generated by the proximal gradient method~\ref{eq:proximalgradientG}. Suppose that $\nu_k$ satisfies Assumption~\ref{assumption:uniform}. Then 
on the event $\{x_k\}_{k \in \NN}$ is bounded, almost surely we have only two possibilities 
\begin{enumerate}
\item $x_k$ converges to a local minimizer $\bar x$ of $f_{v,w}$.
\item $x_k$ converges to a composite Clarke critical point of $f_{v, w}$.
\end{enumerate}
Thus, if $h_0$ and $\cX$ are Clarke regular, the sequence $x_k$ converges to a local minimizer of $f_{v, w}$. 
\end{cor}

In short, the main conclusion of the above three theorems is
\begin{quote}
\centering
On generic regular semialgebraic functions, perturbed subgradient/proximal methods converge only to local minimizers
\end{quote}
We note in passing that the results hold verbatim if one replaces the word ``semialgebraic" with ``definable in an $o$-minimal structure," throughout.

\section{Proofs of the two pillars}\label{sec:proofoftwopillars}

Throughout this section, we let $\EE_k[\cdot ] = \EE[ \cdot \mid \cF_k]$ denote the conditional expectation. We now present the proofs of the two pillars.

\subsection{Proof of Proposition~\ref{prop:gettingclosertothemanifold}: aiming towards the manifold}\label{sec:proof:prop:gettingclosertothemanifold}
Throughout the proof, we let $C$ denote a constant depending on $k_0$ and $\delta$, which may change from line to line. Choose $\delta \leq \min\{1, \frac{c_1\mu}{12\gamma}\}$, satisfying $B_{\delta}(\bar x) \subseteq \mathcal{U}$ where $\mathcal{U}$ is the neighborhood in which Assumption~\ref{assumption:A} holds.
Define $Q := \max\{\sup_{x \in B_{\delta}} q(x), 1\}$. By shrinking $\delta$ slightly, we can assume that the little $o$ term in~\ref{assumption:aiming} satisfies 
$$
o(\dist(x, \cM)) \leq \frac{\mu}{4(1+ Q)}\dist(x, \cM) \qquad  \text{for all $x \in B_{\delta}(\bar x).$}
$$
Now define: $D_k := \dist(x_k, \cM)$ for all $k \geq 0$. We prove a recurrence relation satisfied by the sequence $D_k$. To that end, define $v_k = G_{\alpha_k}(x_k, \perturb_k)$ and observe that in the event $A_k := \{\tau_{k_0, \delta} > k\}$, we have
\begin{align}\label{eq:stayinball1}
D_{k+1}^2 &\le \norm{x_{k+1} - P_{\cM}(x_k)}^2\notag\\
	&=\norm{x_k - \stepsize_k \subgrad_k -P_{\cM}(x_k)}^2 \notag\\
	&=\norm{x_k - P_{\cM}(x_k)}^2  - 2\stepsize_k \dotp{\subgrad_k, x_k - P_{\cM}(x_k)} + \stepsize_k^2 \norm{\subgrad_k}^2\notag\\
&\leq D_k^2  - 2\stepsize_k \mu D_k + 2\alpha_k (1+\|\perturb_k\|)^2o(D_k) \notag \\
&\hspace{20pt}- 2\alpha_k\dotp{\perturb_k, x_k - \proj_{\cM}(x_k)}+ \underbrace{C(1+\|\perturb_k\|)^2}_{:=B_k}\alpha_k^2,
\end{align}
	where the second inequality follows from the proximal aiming and local boundedness properties of $G$; see Assumption~\ref{assumption:A}. 
This inequality will allow us to prove all parts of the result.

Indeed, let us prove Part~\ref{eq:prop:gettingclosertothemanifoldbound1}. To that end, first note that the bound $\EE_k[\|\nu_k\|^2]1_{A_k} \leq q(x_k)1_{A_k} \leq Q$ implies that there exists $C > 0$ such that 
$$
\EE_k[B_k]1_{A_k} \leq C,
$$
meaning the conditional expectation is bounded for all $k$. Moreover, by our choice of $\delta$, 
$$
\EE_k[(1+\|\perturb\|)^2o(D_k)1_{A_k}] \leq \frac{\mu}{2}D_k1_{A_k}.
$$
Thus, for each $k$, we have
\begin{align}
\EE_k[D_{k+1}^21_{A_{k+1}}] &\leq \EE_k[D_{k+1}^21_{A_k}]\notag\\
&\leq D_k^21_{A_k}  - \stepsize_k \mu D_k1_{A_k} + \EE_k[B_k]1_{A_k}\alpha_k^2 -2\alpha_k\dotp{\EE_k[\perturb_k], x_k - \proj_{\cM}(x_k)}1_{A_k} \notag\\
&\leq D_k^21_{A_k}  - \stepsize_k \mu D_k1_{A_k} + C\alpha_k^2  \notag  \\
&\leq (1- (\stepsize_k/2) \mu) D_k^21_{A_k}  - (\stepsize_k/2) \mu D_k1_{A_k} + C\alpha_k^2, \label{eq:decrease}
\end{align}
where the first inequality follows from $1_{A_{k+1}} \leq 1_{A_k}$; the second inequality follows from $\cF_k$-measurability of $A_k$; and the fourth inequality follows since $D_k1_{A_k} \geq D_k^21_{A_k}$ (recall $\delta \leq 1)$. Now apply  Lemma~\ref{lem:fastsummability}  with the sequences $X_k := D_k^21_{A_k}, Y_k := \alpha_k \mu D_k1_{A_k},$ and $Z_k :=  C\alpha_k^2$ and deduce that $(k^{2\gamma - 1}/\log(k+1)^2)D_k^2$ almost surely converges to a finite valued random variable and the following sum is finite:
$$
\sum_{k=1}^\infty \frac{k^{2\gamma - 1}\alpha_k}{\log(k+1)^2} \mu D_k1_{A_k} < + \infty.
$$
Recalling that $\alpha_k \geq c_1/k^\gamma$, we get the claimed summability result. 

Next we prove Part~\ref{eq:prop:gettingclosertothemanifoldbound1p5}. To that end, take expectation of~\eqref{eq:decrease} and use the law of total expectation to deduce that for some $C > 0$, we have
\begin{align*}
\EE[D_{k+1}^21_{A_k}] &\leq (1-\mu\alpha_k/2)\EE[D_k^21_{A_k}] - (\stepsize_k/2) \mu \EE[D_k1_{A_k}]  + C \alpha_k^2 \\
&\leq (1-\mu c_1 k^{-\gamma}/2)\EE[D_k^21_{A_k}] - (\stepsize_k/2) \mu \EE[D_k1_{A_k}] + Ck^{-2\gamma} 
\end{align*}
To prove part~\ref{eq:prop:gettingclosertothemanifoldbound1p5:expectedsquareddistance}, simply apply Lemma~\ref{lem:sequencelemmasquared} applied with sequence $s_k = \EE[D_k^21_{A_k}]$ and constants $c = \mu c_1/2$ and $C$. To prove part~\ref{eq:prop:gettingclosertothemanifoldbound1p5:saddle}, sum the above inequality from $n$ to infinity to get 
\begin{align*}
\sum_{k=n}^\infty (\stepsize_k/2) \mu \EE[D_k1_{A_k}] \leq \EE[D_n^21_{A_n}] + C\sum_{k=n}^\infty \alpha_k^2 &\leq Cn^{-\gamma} + C\sum_{k=n}^\infty \alpha_k^2,
\end{align*}
where the second inequality follows from Part~\ref{eq:prop:gettingclosertothemanifoldbound1p5:expectedsquareddistance}.
Noting that $n^{-\gamma} = O(\sum_{k=n}^\infty \alpha_k^2)$ proves the result.

\subsection{Proof of Proposition~\ref{prop:shadow}: the shadow iteration}\label{sec:proof:prop:shadow}

Throughout the proof we let $C$ denote a constant depending on $k_0$ and $\delta$, but not on $k$, which may change from line to line. We assume $\delta$ is small enough that the conclusions of Proposition~\ref{prop:gettingclosertothemanifold} hold; that $B_{4\delta}(\bar x) \subseteq \mathcal{U}$ where $\mathcal{U}$ is the neighborhood in which Assumption~\ref{assumption:A} holds; and that $P_{\cM}$ and $\nabla P_{\cM}$ are Lipschitz continuous on $B_{4\delta}(\bar x)$. Write $\tau = \tau_{k_0, \delta}$ and fix index $k\geq 1$. Finally, recall that $\proj_{\cM}$ is $C^2$ on $\mathcal{U}$ and $\nabla \proj_{\cM}(x) = \proj_{\tangentM{x}}$ for all $x \in \cM$.  

Let us first prove that $y_k \in B_{4\delta}(\bar x)$. Clearly, we need only consider the case $x\in B_{2\delta}(\bar x)$. In this case,
$$
\|y_k - \bar x\| \leq \|y_k - x_k\| + \|x_k - \bar x\| \leq 2\|x_k - \bar x\| \leq 4\delta,
$$
where the final inequality follows since $\bar x \in \cM$. Therefore, we always have $\norm{y_k-\bar x}\le 4\delta$. 

Next, let us define the error sequence $E_k$ in the shadow iteration. To that end, denote $T_k := \tangentM{y_k}$ and
$$
w_k: = y_{\discrete} -\alpha_k \nabla_\cM f(y_k) - \alpha_k P_{T_\discrete}(\perturb_k)
$$ 
Then with error sequence $E_k := (y_{k+1} - w_k)/\alpha_k$, the claimed recursion is trivially true. Thus, in the remainder of the proof, we bound $E_k$.

Turning to the bound, we first note that throughout the proof, we must separate the analysis into two cases:  $x_{k+1} \in B_{2\delta}(\bar x)$ and $x_{k+1} \notin B_{2\delta}(\bar x)$.  In the second case, the following preliminary observation will be useful: 
\begin{claim}\label{claim:shadowlowerbound}
Suppose that in the event $\{\tau > k\}$ it holds that $x_{k+1} \notin B_{2\delta}(\bar x)$. Then there exists $C > 0$ such that 
\begin{align}\label{eq:shadowlowerbound}
\|y_{k+1} - y_k\| \leq 4\delta \leq C \|x_{k+1} - x_k\|.
\end{align}
\end{claim}
\begin{proof}
First notice that 
$$
\|x_{k+1} - x_k\|  \geq \|x_{k+1} - \bar x\| - \|x_{k} - \bar x\| \geq 2\delta - \delta\geq \delta.
$$
Therefore, the result trivially holds since $\|y_{k+1} - y_k\| \leq 4\delta$. 
\end{proof}

With the preliminaries set, we now bound $\|E_k\|$. To that end, in what follows we assume we are in the event $\{\tau > k\}$ where $k \geq k_0$. In this event, our strategy will be to bound the terms $R_1$ and $R_2$ in the following decomposition:
\begin{align}
\|E_k\| &= \|(y_{k+1} - w_k)/\alpha_k\|  \notag\\
&\leq \underbrace{\|y_{k+1} - y_k - P_{T_k}(y_{k+1} - y_k)\|/\alpha_k}_{:= R_1} + \underbrace{\|P_{T_k}(y_{k+1} - y_k)/\alpha_k +  \nabla f_\cM(y_k) +  P_{T_\discrete}(\perturb_k)\|}_{:= R_2}. \label{eq:EKboundfirst}
\end{align}
In our bounds of these terms, we frequently use the following bound: there exists $C > 0$ such that 
\begin{align}\label{eq:successiveiteratediff1}
\|x_{k+1} - x_k\| \leq \alpha_k\|G_{\alpha_k}(x_k, \perturb_k)\| \leq C(1+\|\perturb_k\|)\alpha_k.
\end{align}
We now bound $R_1$ and $R_2$ separately. 

The following claim bounds $R_1$.
\begin{claim}\label{claim:R1bound}
There exists $C > 0$ such that
\begin{align}\label{eq:R1bound}
R_11_{\tau > k} \leq C(1 + \|\nu_k\|)^2\alpha_k1_{\tau > k}.
\end{align}
\end{claim}
\begin{proof}
We consider two cases. First suppose $x_{k+1} \in B_{2\delta}(\bar x)$. Let $C > 0$ be a local Lipschitz constant of $\nabla P_{\mathcal{M}}$ and $P_{\cM}$. Then it follows that vector $y_{k+1} - y_k = P_{\cM}(x_{k+1}) - P_{\cM}(x_k)$ is nearly tangent to the manifold at $y_k$: 
\begin{align*}
\|y_{k+1} - y_k- \proj_{T_\discrete}(y_{k+1} - y_{\discrete})\|\leq  C\|y_{k+1} - y_k\|^2 \leq  C^3 \|x_{k+1} - x_k\|^2.
\end{align*}
Thus, taking into account~\eqref{eq:successiveiteratediff1}, we have for some $ C> 0$, the bound:
$$
R_1 \leq C(1 + \|\nu_k\|)^2\alpha_k,
$$
as desired. Now suppose that $x_{k+1} \notin B_{2\delta}(\bar x)$. Therefore, there exists $C > 0$ such that 
\begin{align*}
\|y_{k+1} - y_k- \proj_{T_\discrete}(y_{k+1} - y_{\discrete})\| \leq 2\|y_{k+1} - y_k\| \leq C\|x_{k+1} - x_k\| \leq \frac{C^2}{\delta}\|x_{k+1} - x_k\|^2,
\end{align*}
where the first inequality follows since $\|P_{T_k}\| \leq 1$ and the second and third inequalities follow from Claim~\ref{claim:shadowlowerbound}. Thus taking into account~\eqref{eq:successiveiteratediff1}, we again have for some $C > 0$, the bound: 
$$
R_1 \leq  C(1 + \|\nu_k\|)^2\alpha_k,
$$
Thus, putting together both bounds on $R_1$, the result follows.
\end{proof}

The following claim bounds $R_2$.

\begin{claim}\label{claim:R2bound}
There exists $C > 0$ such that  
\begin{align}\label{eq:R2bound}
R_21_{\tau > k} \leq C (1 + \|\nu_k\|)^2(\dist(x_k, \cM) + \alpha_k)1_{\tau > k}.
\end{align}
\end{claim}
\begin{proof}
To bound $R_2$, we first simplify:
\begin{align}\label{eq:R2intermediate}
R_2 &= \|P_{T_k}(y_{k+1} - y_k)/\alpha_k +  \nabla_\cM f(y_k) +  P_{T_\discrete}(\perturb_k)\|\notag \\
&\leq \|\proj_{T_\discrete}(y_{k+1} - x_{k+1})/\alpha_k\| + \|P_{T_k}(x_{k} - y_k)/\alpha_k \| +  \|P_{T_k}(x_{k+1} - x_k)/\alpha_k +  \nabla_\cM f(y_k) +  P_{T_\discrete}(\perturb_k)\| \notag\\
&\leq \|\proj_{T_\discrete}(y_{k+1} - x_{k+1})/\alpha_k\| + C(1+\|\perturb_k\|)^2(\dist(x_k, \cM) +\alpha),
\end{align}
where the second inequality follows from by Assumption~\ref{assumption:A} and the inclusion $x_{\discrete} - y_k \in N_{\cM}(y_k)$, which implies that $\proj_{T_\discrete}(x_{k} - y_{\discrete}) = 0$. We now bound the term $\|\proj_{T_\discrete}(y_{k+1} - x_{k+1})/\alpha_k\|$.

First suppose that $x_{k+1} \in B_{2\delta}(\bar x)$ and note that $y_{k+1} \in B_{4\delta}(\bar x)\cap \cM \subseteq \mathcal{U} \cap \cM$. Let $C' > 0$ be a local Lipschitz constant of $\nabla P_{M}$ and $P_{\cM}$. Then for some $C > 0$ larger than $C'$, we have 
\begin{align*}
\|\proj_{T_\discrete}(y_{k+1} - x_{k+1})/\alpha_k\| &\leq \|(\proj_{T_{k+1}} - \proj_{T_k})(y_{k+1} - x_{k+1})/\alpha_k\| \\
&\leq C'\|y_{k+1} - y_k\|\dist(x_{k+1}, \cM)/\alpha_k \\
&\leq (C')^2\|x_{k+1} - x_k\| (\dist(x_k, \cM) + \|x_{k+1} - x_k\|)/\alpha_k \\
&\leq C^3(1 + \|\perturb_k\|)\dist(x_k, \cM) + C^4(1 + \|\nu_k\|)^2\alpha_k,
\end{align*}
where the first inequality follows from $x_{k+1} - y_{k+1} \in N_{\cM}(y_{k+1})$, which implies $P_{T_{k+1}}(y_{k+1} - x_{k+1}) = 0$; the second inequality follows from Lipschitz continuity of $\nabla P_{\cM}(y) = P_{\tangentM{y}}$ in $y$; the third inequality follows from Lipschitz continuity of $P_{\cM}$ and Lipschitz continuity of $\dist(\cdot, \cM)$; and the fourth inequality follows from~\eqref{eq:successiveiteratediff1}.
Plugging this bound into~\eqref{eq:R2intermediate}, yields that for some $C > 0$, we have
\begin{align*}
R_2 \leq C (1 + \|\nu_k\|)^2(\dist(x_k, \cM) + \alpha_k), 
\end{align*}
as desired.

Now suppose that $x_{k+1} \notin B_{2\delta}(\bar x)$. 
Then, there exists $C > 0$ such that 
\begin{align*}
\|\proj_{T_\discrete}(y_{k+1} - x_{k+1})/\alpha_k\| &\leq \|\proj_{T_\discrete}(y_{k+1} - x_k)\|/\alpha_k + \|\proj_{T_\discrete}(x_k - x_{k+1})\|/\alpha_k\\
&\leq 2\delta/\alpha_k + \|x_{k} - x_{k+1}\|/\alpha_k \\
&\leq (1+C)\|x_k - x_{k+1}\|/\alpha_k \\
&\leq \frac{(1+C)C}{\delta\alpha_k}\|x_k - x_{k+1}\|^2 \\
&\leq \frac{(1+C)C^3}{\delta}(1+\|\perturb_k\|)^2 \alpha_k
\end{align*}
where first inequality follows from the triangle inequality; the second inequality follows since $x_k \in B_{2\delta}(\bar x)$ and $y_{k+1} = \bar x$; the third and fourth third inequalities follow from Claim~\ref{claim:shadowlowerbound}; and the fifth follows from~\eqref{eq:successiveiteratediff1}. Thus, in this case, we find that there exists $C > 0$ with
\begin{align*}
R_2 \leq C (1 + \|\nu_k\|)^2(\dist(x_k, \cM) +\alpha_k).
\end{align*}
Therefore, putting together both bounds on $R_2$, the result follows.
\end{proof}

Now we prove Part~\ref{prop:shadow:part:error:part:upperbound}. Beginning with subpart~\ref{prop:shadow:part:error:part:upperbound:1}, we find that by Claim~\ref{claim:R1bound} and~\ref{claim:R2bound}, we have that for some $C > 0$, the bound
\begin{align}
\|E_k\|1_{\tau > k} \leq R_11_{\tau > k} + R_21_{\tau > k} \leq  C(1+\|\perturb_k\|)^2(\dist(x_k, \cM) + \alpha_k)1_{\tau > k}, \label{eq:EKboundfinal}
\end{align}
as desired. Turning to Part~\ref{prop:shadow:part:error:part:upperbound:2}, first note that that $\dist(x_k, \cM)1_{\tau > k} \leq \delta$. Thus, the bound will follow if the conditional expectation of $(1+\|\perturb_k\|)^4$ is bounded whenever $x_k \in B_{\delta}(\bar x)$. This holds by assumption, since 
$$
\EE_k[\|\nu_k\|^4]1_{\tau > k} \leq  \sup_{x \in B_{\delta}(\bar x)} q(x) < \infty.
$$
Finally, we prove Part~\ref{prop:shadow:part:error:part:upperbound:3}. Again using the boundedness of the conditional fourth moment of $\|\nu_k\|1_{\tau > k}$, we find that there exists a $C > 0$ such that 
\begin{align}\label{eq:conditionalEKstuff}
 \EE_k[\|E_k\|^2 1_{\tau > k}] \leq C \dist^2(x_k, \cM) 1_{\tau > k} + C \alpha_k^2 1_{\tau > k},
\end{align}
where the first inequality follows from Jensen's inequality and the second inequality follows from~\eqref{eq:EKboundfinal}. Consequently, there exists $C' > 0$ such that
$$
\EE[\|E_k\|^21_{\tau > k}] = \EE[\EE_k\|E_k\|^21_{\tau > k}] \leq C\EE[  \dist^2(x_k, \cM)1_{\tau > k}] + C\alpha_k^2 
\leq C'\alpha_k,
$$
where the third inequality follows from Part~\ref{eq:prop:gettingclosertothemanifoldbound1p5:expectedsquareddistance} of Proposition~\ref{prop:gettingclosertothemanifold}. This prove Part~\ref{prop:shadow:part:error:part:upperbound}.  

Now we prove Part~\ref{prop:shadow:part:error:part:random}, beginning with Part~\ref{prop:shadow:part:error:part:random:1}. To that end, define $F_k = \frac{k^{\gamma - 1}}{\log(k+1)^2} \|E_k\|1_{\tau > k}$. Recall that by the conditional Borel-Cantelli theorem (Lemma~\ref{lem:condborelcantelli}), the sequence $F_k$ is summable whenever $\EE_k[F_k]$ is summable. 
Thus, we first upper bound $\EE_k[F_k]$ by a summable sequence: there exists $C > 0$ such that 
\begin{align*}
\EE_k[F_k] &\leq C \frac{k^{\gamma - 1}}{\log(k+1)^2}(\dist(x_k, \cM) + \alpha_k)1_{\tau > k}  \\
&\leq C\frac{k^{\gamma - 1}}{\log(k+1)^2}\dist(x_k, \cM)1_{\tau > k} + C\frac{c_2}{k\log(k+1)^2},
\end{align*}
where the first inequality follows from~\eqref{eq:conditionalEKstuff} and the second inequality follows by definition of $\alpha_k$. By Part~\ref{eq:prop:gettingclosertothemanifoldbound1} of Proposition~\ref{prop:gettingclosertothemanifold}, it follows that we have upper bounded $\EE_k[F_k]$ by a summable sequence. Therefore, it follows that $F_k$ is summable, as desired. This proves part~\ref{prop:shadow:part:error:part:random:1}. 

Now we prove part~\ref{prop:shadow:part:error:part:random:2}. The conditional expectation is summable by Part~\ref{prop:shadow:part:error:part:upperbound:3}, since 
$$
\sum_{k=k_0}^\infty \frac{k^{\gamma - 1}}{\log(k + 1)^2}\EE[\|E_k\|^21_{\tau > k}] \leq C\sum_{k=k_0}^\infty \frac{k^{-1}}{\log(k + 1)^2} < + \infty.
$$
By conditional Borel-Cantelli theorem (Lemma~\ref{lem:condborelcantelli}), we also have that $$\sum_{k=k_0}^\infty \frac{k^{\gamma - 1}}{\log(k + 1)^2}\|E_k\|^21_{\tau > k} < +\infty,$$ as desired. 

Now we prove Part~\ref{prop:shadow:part:error:part:random:saddle}. To that end, note that there exists $C > 0$ such that 
\begin{align*}
\EE[ \alpha_k\|E_k\|1_{\tau > k}] = \EE[ \alpha_k\EE_k[\|E_k\|1_{\tau > k}]] 
\leq C\EE[ \alpha_k\dist(x_k, \cM)1_{\tau > k} + \alpha_k^21_{\tau > k}].
\end{align*}
where the inequality follows from~\eqref{eq:conditionalEKstuff}. Thus, the result follows by Part~\ref{eq:prop:gettingclosertothemanifoldbound1p5:saddle} of Proposition~\ref{prop:gettingclosertothemanifold}.

\section{Proofs of the main theorems}\label{sec:proofsofstochastic}

In this section, we prove the remaining theorems.
\subsection{Proof of Theorem~\ref{thm: nonconvergence to unstable points}: nonconvergence of stochastic process}\label{sec:proofofpemantlething}

We begin by recalling and slightly reframing Proposition 3 in~\cite{pemantle1990nonconvergence}. This result provides a Lyapunov function, which we will use to show that each local process $Y_k$ escapes a local neighborhood of each $p \in S$.
\begin{proposition}[Lyapunov Function]\label{prop: geometric part}
Fix $p \in \RR^d$ and suppose $F \colon \RR^d \rightarrow \RR^d$ is a $C^2$ mapping that is zero at $p$ and has a symmetric Jacobian $\nabla F(p)$. Suppose that $\nabla F(p)$ has at least one positive eigenvalue and let $W$ denote the subspace of eigenvectors of $\nabla F(p)$ with positive eigenvalues. Then, there exists a matrix $A \in \RR^{d\times d}$ with $\range(A^T) = W$, a ball $\cB$ centered at $p$, and a $C^2$ mapping $\Phi \colon \cB \rightarrow \RR^d$ with $\Phi(p) = p$ and $\nabla \Phi(p) = I_d$ such that the function $\eta \colon \cB \rightarrow \RR$ defined as 
$$
\eta(v) = \|A(\Phi(v) - p)\|_2
$$
satisfies the following condition: There exists $c,c'>0$ such that 
$$
\eta(v+\epsilon F(v)) \ge (1+c\epsilon) \eta(v) - c'\epsilon^2 \qquad \text{for $v$ in $\cB$ and all sufficiently small $\epsilon$}.
$$
In particular, we have 
$$
\eta'(v; F(v)) \geq c\eta(v) \qquad \text{for all $v \in \cB$.}
$$
\end{proposition}

Turning to the proof of Theorem~\ref{thm: nonconvergence to unstable points}, we begin with a covering argument: For any $p\in S$, choose $\epsilon_p$ small enough that both the conditions of Theorem~\ref{thm: nonconvergence to unstable points} and Proposition~\ref{prop: geometric part} hold in $B_{\epsilon_p}(p)$ for $F_p$. Let $\delta_p \leq \epsilon_p$ and $c_3, c_4, c_5>0$ be the associated constants.  Clearly, the union $\cup_{p \in S}\xkball$ is an open cover of set $S$, and therefore there exists a countable index set $\Lambda\subset S$ such that $S\subset \cup_{p \in \Lambda}\xkball$. Therefore, to prove Theorem~\ref{thm: nonconvergence to unstable points}, it suffices to show that 
\begin{align}\label{eqn: markov property}
	P\left(X_\discrete \in \xkball, \forall \discrete \ge k_0\right) =0 \qquad \text{for all $k_0\ge K_p$.}
\end{align}

To this end, fix $p \in \Lambda$ and $k_0 \geq K_p$. Let $F = F_p$ denote the local mapping in Condition~\ref{item:localiteration} of Theorem~\ref{thm: nonconvergence to unstable points}. In addition, let $\eta = \eta_p$, denote the mapping associated to $F$, guaranteed to exist by Theorem~\ref{prop: geometric part}.\footnote{Note that strictly speaking we should extend $F$ to all $\RR^d$, for example, by a partition of unity~\cite[Lemma 2.26]{lee2012manifold}. Since the argument that follows is local, we omit this discussion for simplicity.} Furthermore, recall the stopping time $\tau_{k_0} = \tau_{k_0, \delta_p}(p)$, defined as $\tau_{k_0, \delta_p}(p) = \inf\{\discrete \ge k_0 \colon X_\discrete \notin \xkball\}$. Note that~\eqref{eqn: markov property} holds if $P(\tau_{k_0} = \infty) = 0$.

Our strategy is as follows. We first prove that on the event $\{\tau_{k_0} = \infty\}$, we have $\eta(Y_k) \rightarrow 0$ almost surely. Then we show that $P(\{\tau_{k_0} = \infty\}\cap \{\eta(Y_k) \rightarrow 0\}) = 0$. This will imply that $P(\{\tau_{k_0} = \infty\}) = 0$ and the proof will be complete. These two claims are subjects of the following two subsections.

\subsubsection{Claim: On the event $\{\tau_{k_0} = \infty\}$, we have $\eta(Y_k) \rightarrow 0$}

To prove this claim, note that the following hold for almost all sample paths in the event $\{\tau_{\discrete_0} = \infty\}$:
\begin{enumerate}
	\item The sequence $Y_\discrete$ is bounded. 
	\item Define $\beta_k = \sum_{i=0}^{k-1}\stepsize_i$. Then for each $T>0$, the limit holds:
	\begin{align}\label{eqn: small tail}
		\lim_{n\rightarrow \infty}\left(\sup_{k\colon 0 \le \beta_k - \beta_n \le T} \norm{\sum_{i=n}^{k-1}\stepsize_i \cdot (\xi_i + E_k)}\right)=0.
	\end{align}
Indeed, note that by Condition~\ref{label:stochasticprocess:averaging} of the Theorem, it suffices to show $M_k=\sum_{i=0}^k \alpha_i \xi_i$ converges almost surely, since then it is a Cauchy sequence. To prove that $M_k$ converges, note that $\sum_{i}\alpha_i^2<\infty$ and $\limsup \EE[\norm{\xi_k}^2\mid \cF_k]< \infty$, so $M_k$ is a martingale. Moreover, 
\begin{align}
	\sup_{k\ge 0} \expect{\norm{M_k}}^2 \leq \sup_{k\ge 0} \expect{\norm{M_k}^2} \le  c_3^{\frac{1}{2}} \sum_{i\ge 0}\alpha_i^2 <\infty.
\end{align}
Standard martingale theory then shows that $M_k$ converges almost surely~(Theorem 4.2.11 in~\cite{durrett2019probability}). Therefore, \eqref{eqn: small tail} holds almost surely.
	\end{enumerate}

These conditions match those of~\cite[Theorem 1.2]{benaim1996dynamical}. Consequently, by this result it holds that the set of limit points of $Y_\discrete$ is almost surely invariant under the mapping $\Theta_t \colon \ykball \rightarrow \RR^d$, defined as the time-$t$ map of the ODE $\dot \gamma(t) = F(\gamma(t))$. Thus, for any $x'$ in limit set of $Y_\discrete$, we have $\Theta_t(x') \in \overline \ykball$ for all $t \geq 0$. Consequently, by Proposition~\ref{prop: geometric part}, we have
\begin{align}
	\eta'(\Theta_t(x'); F(\Theta_t(x'))) \ge c \eta(\Theta_t(x')) \qquad \text{for all $t \geq 0$.}
\end{align}
Therefore, by integrating $\eta'$ with respect to $t$, we have for all $t \geq 0$, the bound 
\begin{align*}
\eta(\Theta_t(x')) = \eta(\Theta_0(x')) + \int_{0}^t \eta'(\Theta_s(x'); F(\Theta_s(x')))ds \geq \eta(\Theta_0(x')) + \int_{0}^t c \eta(\Theta_s(x'))ds. 
\end{align*}
Thus, by Gronwall's inequality~\cite{gronwall1919note} it holds that 
$$
\eta(\Theta_t(x')) \ge e^{ct}\eta(\Theta_0(x')) = e^{ct}\eta(x') \qquad \text{for all $t \geq 0$.}
$$
Now observe that since $\Theta_t(x') \in \overline  \ykball$, the quantity $\eta(\Theta_t(x'))$ is bounded for all $t \geq 0$. Consequently, we must have $\eta(x')=0$. Thus, we have shown that for all limits points $x'$ of $Y_k$, we have $\eta(x') = 0$. Since $\eta$ is continuous in $\overline \ykball$, we must therefore have $\eta(Y_\discrete) \rightarrow 0$.

\subsubsection{Claim: We have $P(\{\tau_{k_0} = \infty \} \cap \{ \eta(Y_k) \rightarrow 0\}) = 0$.}

We begin by stating the following straightforward extension of~\cite[Theorem 4.1]{brandiere1998some}.
\begin{lemma}\label{lem:repulsion}
Let $\{\zeta_k\}_k$ be a nonnegative sequence of random variables adapted to a filtration $\{\cF_k\}$ satisfying the following recurrence almost surely on an $\cF_{\infty}$-measurable set $\Omega_0$:
\begin{align*}
\zeta_{k+1} \geq \zeta_k + \alpha_k (e_{k+1} + r_{k+1} + \hat r_{k+1}) \qquad \text{for all $k \geq k_0$.}
\end{align*}
where $\{\alpha_k\}$ is a square-summable, but not summable sequence. Assume that $\{e_k\}_k$, $\{r_k\}$, and $\{\hat r_k\}_k$ are $\cF_k$ measurable and satisfy 
\begin{align*}
\EE[e_{k+1} \mid \cF_k] = 0; \qquad && &&\qquad &&\sum_{k=1}^\infty r_k^2 < + \infty \\
\limsup_{k} \EE[e_{k+1}^2 \mid \cF_k] < \infty \qquad &&  &&\qquad && \liminf_k \EE[|e_{k+1}| \mid \cF_k] > 0, 
\end{align*}
almost surely on $\Omega_0$.
Assume that for $n \geq k_0$, we have
$$
\EE\left[1_{\Omega_0}\sum_{k=n}^\infty \alpha_k |\hat r_{k+1}|\right] = O\left(\sum_{k = n}^\infty \alpha_k^2\right).
$$
Then we have $P( \Omega_0 \cap \{\zeta_k \rightarrow 0\}) = 0$.
\end{lemma}
\begin{proof}
Without loss of generality we may assume $k_0 = 0.$ Following \cite[Theorem 4.1]{brandiere1998some} (itself based on \cite[Page 401]{AIHPB_1996__32_3_395_0}) it suffices to work in the case where there exist fixed constants $\mu$ and $C > 0$ such that almost surely on the whole probability space, we have
\begin{align*}
\EE[e_{k+1} \mid \cF_k] = 0 \qquad \text{ and } \qquad \limsup_{k} \EE[e_{k+1}^2 \mid \cF_k ] < C \\
\liminf_{k} \EE[|e_{k+1}| \mid \cF_k ] > \mu > 0 \qquad  \text{ and } \qquad\sum_{k=1}^\infty r_{k}^2 < C.
\end{align*}

Now define the nonnegative residual sequence:
$$
\alpha_k U_{k+1} = \zeta_{k+1} - \zeta_k - \alpha_k (e_{k+1} + r_{k+1} + \hat r_{k+1})
$$
Notice that for all $k \geq 0$, we have 
$$
\zeta_{k} = \left[ \zeta_0 + \sum_{j = 0}^k \alpha_j(e_{j+1} + r_{j+1} + \hat r_{j+1} +  U_{j+1})\right] \qquad \text{ on $G := \Omega_0 \cap \{ \zeta_k \rightarrow 0\}$.}
$$
Therefore, on $G$, we have
$$
- \zeta_0 = \left[\sum_{j = 0}^\infty \alpha_j(e_{j+1} + r_{j+1} + \hat r_{j+1} +  U_{j+1})\right].
$$
Then as argued the proof of~\cite[Theorem 4.1]{brandiere1998some} it suffices by Theorem A of~\cite{AIHPB_1996__32_3_395_0} (included as Lemma~\ref{lem:nonmeasruable} in the Appendix) to show that 
$$
\EE\left[1_{G}\sum_{k=n}^\infty\alpha_k |U_{k+1} + \hat r_{j+1}|\right] = o\left(\left(\sum_{k=n}^\infty \alpha_k^2\right)^{1/2}\right),
$$
 Clearly, it suffices to bound the series $\EE\left[1_G\sum_{j=K}^\infty \alpha_j U_{j+1}\right]$ (which consists of nonnegative terms), since by assumption, we have
$$
\EE\left[1_{G}\sum_{k=n}^\infty\alpha_k |\hat r_{j+1}|\right] = O\left(\sum_{j=n}^\infty \alpha_j^2\right) = o\left(\left(\sum_{j=n}^\infty \alpha_j^2\right)^{1/2}\right).
$$
To that end, note that for all $k, n \geq 0$, we have 
\begin{align*}
\zeta_{n + k} =  \left[ \zeta_n+ \sum_{j = n}^{n + k} \alpha_j(e_{j+1} + r_{j+1} + \hat r_{j+1} + U_{j+1})\right]
\end{align*}
Hence on $G$, we may let $k$ tend to infinity, yielding:
\begin{align*}
-\zeta_{n} = \sum_{j = n}^{\infty} \alpha_j(e_{j+1} + r_{j+1} + \hat r_{j+1} + U_{j+1}) .
\end{align*}
Thus, on the event $G$, we have  
\begin{align*}
 \sum_{j=n}^\infty \alpha_j U_{j+1} = - \zeta_n - \sum_{j = n}^{\infty} \alpha_j(e_{j+1} + r_{j+1} + \hat r_{j+1} )
\end{align*}
Therefore, we find that 
\begin{align*}
\EE\left[1_G\sum_{j=n}^\infty \alpha_j U_{j+1}\right]  &\leq - \EE\left[\zeta_n \right] -  \EE\left[1_{G}\sum_{j = n}^{\infty} \alpha_j(e_{j+1} + r_{j+1} + \hat r_{j+1} )\right] \\
&\leq  \abs{\EE\left[1_{G}\sum_{j = n}^{\infty} \alpha_j(e_{j+1} + r_{j+1} )\right]} + o\left(\left(\sum_{j=n}^\infty \alpha_j^2\right)^{1/2}\right).
\end{align*}
where the second inequality follows from nonnegativity of $\zeta_n$ and our assumptions on $\hat r_{j+1}$. Thus, to complete the bound of $\EE[1_G\sum_{j=K}^\infty \alpha_j U_{j+1}] $ we must show that 
$$
\abs{\EE\left[1_{G}\sum_{j = n}^{\infty} \alpha_j(e_{j+1} + r_{j+1} )\right]} = o\left(\left(\sum_{j=n}^\infty \alpha_j^2\right)^{1/2}\right).
$$

The above bound follows by the exact same argument as \cite[Theorem 4.1]{brandiere1998some}, which we reproduce for completeness: First let $G_n = \EE[1_G \mid \cF_n]$, recall that $G$ is $\cF_\infty$ measurable and that $G_n$ converges to $1_{G}$ almost surely in $L^p$ for every $p \geq 1$, e.g., $\EE[(G_n - 1_G)^2] \rightarrow 0$. Turning to the bound, we have
\begin{align*}
&\abs{\EE\left[1_{G}\sum_{j = n}^{\infty} \alpha_j(e_{j+1} + r_{j+1} )\right] }\\
&\leq \abs{\EE\left[(1_{G} - G_n)\sum_{j = n}^{\infty} \alpha_j(e_{j+1} + r_{j+1} )\right]} + \abs{\EE\left[G_n\sum_{j = n}^{\infty} \alpha_j(e_{j+1} + r_{j+1} )\right]}\\
&\leq \EE[(1_G - G_n)^2]^{1/2}\underbrace{\left(\EE\left[\left(\sum_{j = n}^{\infty} \alpha_j(e_{j+1} + r_{j+1})\right)^2\right]\right)^{1/2}}_{=:R_1} + \underbrace{\EE\left[\sum_{j = n}^{\infty} \alpha_j|r_{j+1}|\right]}_{=:R_2}.
\end{align*}
The proof will be complete if $R_1 = O\left(\sum_{j=n}^\infty \alpha_j^2\right)^{1/2}$ and $R_2 = o\left(\sum_{j=n}^\infty \alpha_j^2\right)^{1/2}$. Let us first bound $R_2$: 
\begin{align*}
R_2 \leq \left(\sum_{j=n}^\infty \alpha_j^2\right)^{1/2} \EE\left[\sum_{j = n}^{\infty} r_{j+1}^2\right]^{1/2} = o\left(\left(\sum_{j=n}^\infty \alpha_j^2\right)^{1/2}\right),
\end{align*}
where the last inequality follows from the bound $\sum_{k=1}^\infty r_{k+1}^2 < C$. Now we bound $R_1$:
\begin{align*}
R_1 &\leq \left(\EE\left[\left(\sum_{j = n}^{\infty} \alpha_je_{j+1}\right)^2\right]\right)^{1/2} + \left(\EE\left[\left(\sum_{j = n}^{\infty} \alpha_jr_{j+1}\right)^2\right]\right)^{1/2} \\
&\leq \left(\EE\left[\sum_{j = n}^{\infty} \alpha_j^2 \EE[e_{j+1}^2 \mid \cF_k] \right]\right)^{1/2} + \left(\sum_{j=n}^\infty \alpha_j^2\right)^{1/2} \left(\EE\left[\sum_{j = n}^{\infty} r_{j+1}^2\right]\right)^{1/2} = O\left(\left(\sum_{j=n}^\infty \alpha_j^2\right)^{1/2}\right).
\end{align*}
Therefore, the proof is complete.
\end{proof}

Now we apply the above Lemma. To that end, we state a few simplifications and facts to be used below. First, throughout the proof, we let $C$ be a positive constant that changes from line to line. Second, we simplify notation and let $\tau$ denote $\tau_{k_0, \delta}$. Third, we recall the bound $\frac{c_1}{k^\gamma} \le \stepsize_k \le \frac{c_2}{k^\gamma}$.  Fourth, the function $\eta$ is weakly convex and Lipschitz continuous on $B_{\epsilon_p}(p)$. Fifth, the Jacobian $ \nabla \Phi$ is $\text{Lip}_{\nabla \Phi}$-Lipschitz in $B_{\epsilon_p}(p)$. Sixth, we note that for sufficiently large $k$, we have the following on $\{\tau = \infty\}$: $Y_k + \alpha_kF(Y_k) \in B_{\epsilon_p}(p)$. We may assume without loss of generality that these assertions hold for all $k \geq 1$. Finally, we note that by shrinking $\epsilon_p$, if necessary, we can assume that on the event $\{\tau = \infty\}$, we have
\begin{align}\label{eq:a3}
s_{\min}(A)\liminf_k\inf_{w \in W\cap \sphere^{d-1}}\EE[|\dotp{w,  \xi_k}|\mid \cF_k] -  \epsilon_p\limsup_k\EE[\|\xi_k\| \mid \cF_k] \|A\|_{\rm op}\text{Lip}_{\nabla \Phi} \notag \\
\geq c_4s_{\min}(A) - \epsilon_p c_3^{1/4}\|A\|_{\rm op}\text{Lip}_{\nabla \Phi} > 0
\end{align}
where $c_4$ and $c_3$ are independent of $\epsilon_p$ and $\delta_p$, $A$ is defined in Proposition~\ref{prop: geometric part}, and $s_{\min}(A)$ denotes the minimal nonzero singular value of $A$.

Now let $s \colon B_{\epsilon_p}(p) \rightarrow \RR^d$ be a selection of $\partial \eta$ defined as follows: for all $y \in B_{\epsilon_p}(p)$, 
\begin{itemize}
\item If $\eta(y) \neq 0$, then $\eta$ is differentiable at $Y$, so set $s(y) = \nabla \eta(y).$
\item If $\eta(y) = 0$, then $\eta$ is nondifferentiable, so we choose subgradient
$$
s(Y) = \nabla \Phi(y)^\top A^\top u \in \partial \eta(y)
$$
where $u \in \sphere^{d-1}$ satisfies $\|A^\top u\| = \|A\|_{\rm op} > 0$. \end{itemize}
Next, consider the event $\Omega_0 = \{\tau = \infty\}$. Then by the boundedness of $s(Y_k + \alpha_k F(Y_k))$ and the weak convexity of $\eta$ on $B_{\epsilon_p}(p)$, there exists $C > 0$ such that 
\begin{align}
\eta(Y_{k+1}) &\geq \eta(Y_{k} + \alpha_k F(Y_k)) +\dotp{s(Y_k + \alpha_k F(Y_k)),\alpha_k E_k + \alpha_k \xi_k} - C\| \alpha_k E_k + \alpha_k \xi_k\|^2 \notag \\
&\geq \eta(Y_{k} + \alpha_k F(Y_k)) + \dotp{s(Y_k + \alpha_k F(Y_k)),\alpha_k \xi_k} - C\| \alpha_k E_k + \alpha_k \xi_k\|^2- C\alpha_k\|E_k\| \notag \\
&\geq (1+c \alpha_k) \eta(Y_k) + \dotp{s(Y_k + \alpha_k F(Y_k)),\alpha_k \xi_k} - C\| \alpha_k E_k + \alpha_k \xi_k\|^2 - C\alpha_k\|E_k\|- C\alpha_k^2.
\end{align}
Now define four sequences:
$$
\zeta_k := \eta(Y_{k}); \quad e_{k+1} := \dotp{s(Y_k + \alpha_k F(Y_k)), \xi_k}; \quad r_{k+1} := - C\alpha_k\left(1+\|  E_k + \xi_k\|^2 \right);  \quad \hat r_{k+1} := -C\|E_k\|
$$
and observe that on $\Omega_0$, we have 
$$
\zeta_{k+1} \geq \zeta_k + \alpha_k (e_{k+1} + r_{k+1} + \hat r_{k+1}).
$$

Now we must verify the assumptions of the Lemma. We begin with $\hat r_{k+1}$. To that end, observe that
$$
\EE\left[1_{\Omega_0} \sum_{k=n}^\infty  \alpha_k \hat r_{k+1}\right] = O\left(\sum_{k=n}^\infty \alpha_k^2\right), 
$$ 
by our assumption on $\|E_k\|$. 
Next we prove square summability of $r_{k+1}$ on $\Omega_0$: Indeed, observe
$$
  r_{k+1}^2 \leq C\alpha_k^2 (\|\xi_k\|^4 + \|E_k\|^4+1). 
$$
Moreover both $\limsup_{k}\EE_k[\|\xi_k\|^4\mid \cF_k] < \infty$ and $\limsup_{k}\EE_k[\|E_k\|^4\mid \cF_k] < \infty$ are bounded on $\Omega_0$. Therefore, by conditional Borel-Cantelli Lemma~\ref{lem:condborelcantelli}, we have 
$$
\sum_{k=1}^\infty r_{k+1}^2 < + \infty.
$$
almost surely on $\Omega_0$. 

Finally we prove that $e_k$ has the desired properties. First note that  we have 
$$
\EE[e_{k+1} \mid \cF_k] = 0 \qquad  \text{ and } \qquad \limsup_{k} \EE[e_{k+1}^2 \mid \cF_k] < \infty.
$$
on $\Omega_0$. Indeed, this follows since $\limsup_{k} \EE[\|\xi_k\|^4 \mid\cF_k] < \infty$ almost surely and and $Y_k + \alpha_k F(Y_k) \in B_{\epsilon_p}(p)$ on $\Omega_0$. Next, since $\eta$ is globally Lipschitz on $B_{\epsilon_p}(p)$, we have that $s(Y_k + \alpha_k F(Y_k))$ is uniformly bounded. Thus, 
$$
\limsup_{k} \EE[e_{k+1}^2 \mid \cF_k]  \leq \limsup_{k} \EE[\|s(Y_k + \alpha_k F(Y_k))\|^2\|\xi_k\|^2 \mid \cF_k] < \infty,
$$
on $\Omega_0$, as desired.

Now we prove that $\liminf \EE[|e_{k+1}| \mid \cF_k]$ is positive on $\Omega_0$. To that end, recall that the mapping $\Phi$ satisfies $\nabla \Phi(p) = I_d$.  Turning to the proof, there are two cases to consider. First suppose that $\eta(Y_k + \alpha_k F(Y_k)) \neq 0$. Then $\eta$ is differentiable at $Y_k + \alpha_k F(Y_k)$. Now define $u_k := \frac{A(\Phi(Y_k+\alpha_kF(Y_k)) - p)}{\norm{A(\Phi(Y_k+\alpha_kF(Y_k))-p)}}$ and note that  
\begin{align*}
s(Y_k + \alpha_k F(Y_k)) = \nabla \eta(Y_k + \alpha_k F(Y_k)) &= \nabla \Phi(Y_k +\alpha_kF(Y_k))^\top A^\top u_k\\
&= A^\top u_k + (\nabla \Phi(Y_k +\alpha_kF(Y_k)) - \nabla \Phi(p))^\top A^\top u_k \\
&\in A^\top u_k +  \epsilon_p\|A\|_{\rm op}\text{Lip}_{\nabla \Phi}B_1(0),
\end{align*}
where the inclusion follows since $Y_k + \alpha_k F(Y_k) \in B_{\epsilon_p}(p)$. Let $s_{\min}(A)$ denote the minimal nonzero singular value of $A$ and notice that since $u_k \in \sphere^{d-1} \cap \range(A)$, we have that $w_k := A^T u_k $ satisfies and 
$$
w_k \in W \qquad \text{ and } \qquad\|w_k\| \geq s_{\min}(A) > 0. 
$$
Therefore, it follows that on the event $\Omega_0$, we have
\begin{align*}
 \EE[|e_{k+1}| \mid \cF_k] &=  \EE[|\dotp{s(Y_k +  \alpha_k F(Y_k)),  \xi_k}| \mid \cF_k] \\
&\geq \EE[|\dotp{w_k,  \xi_k}|\mid \cF_k] -  \epsilon_p\EE[\|\xi_k\| \mid \cF_k] \|A\|_{\rm op}\text{Lip}_{\nabla \Phi} \\
&\geq s_{\min}(A)\inf_{w \in W\cap \sphere^{d-1}}\EE[|\dotp{w,  \xi_k}|\mid \cF_k] -  \epsilon_p\EE[\|\xi_k\| \mid \cF_k] \|A\|_{\rm op}\text{Lip}_{\nabla \Phi} 
\end{align*}
We now consider the case $\eta(Y_k + \alpha_k F(Y_k)) = 0$. In this case, there exists $u_k\in \sphere^{d-1} $ such that $\|A^\top u_k\| = \|A\|_{\rm op}$ and 
\begin{align*}
s(Y_k + \alpha_k F(Y_k)) =  \nabla \Phi(Y_k + \alpha_k F(Y_k))^\top A^\top u_k &\in A^\top u_k +  \epsilon_p\|A\|_{\rm op}\text{Lip}_{\nabla \Phi}B_1(0),
\end{align*}
Recall $\range(A^\top)=W$. Thus, we have that the vector  $w_k :=A^\top u_k$ is in $W$ and  $\norm{w_k} = \norm{A}_{\rm op}>0$. Thus, for all $v \in \RR^d$, we have 
\begin{align*}
 |\dotp{s(Y_k + \alpha_k F(Y_k)), v}| &= \dotp{\nabla \Phi(Y_k + \alpha_k F(Y_k))^\top A^\top u_k, v} \ge \dotp{w_k, v} - \epsilon_p\text{Lip}_{\nabla \Phi}\norm{A}_{\rm op}\norm{v}. 
\end{align*}
Taking $v= \xi_k$, we obtain
\begin{align*}
	\EE[|\dotp{s(Y_k + \alpha_k F(Y_k)), \xi_k}| \mid \cF_k] \geq  \|A\|_{\rm op}\inf_{w \in W\cap \sphere^{d-1}}\EE[|\dotp{w,  \xi_k}|\mid \cF_k] -  \epsilon_p\EE[\|\xi_k\| \mid \cF_k] \|A\|_{\rm op}\text{Lip}_{\nabla \Phi} 
\end{align*}
Thus, putting both cases together, we find that on the event $\Omega_0$, we have
$$
\liminf_{k} \EE[|e_{k+1}| \mid \cF_k ] \geq s_{\min}(A)\liminf_k\inf_{w \in W\cap \sphere^{d-1}}\EE[|\dotp{w,  \xi_k}|\mid \cF_k] -  \epsilon_p\limsup_k\EE[\|\xi_k\| \mid \cF_k] \|A\|_{\rm op}\text{Lip}_{\nabla \Phi} > 0,
$$
where the last inequality follows from~\eqref{eq:a3}.  

\subsection{Proof of Theorem~\ref{thm:avoidance}: nonconvergence to saddle points}\label{section:proof:thm:avoidance}

In this section, prove Theorem~\ref{thm:avoidance} by verifying that the iterates $\{x_k\}_{k \in \NN}$ satisfy the conditions of Theorem~\ref{thm: nonconvergence to unstable points}. We begin with some notation. To this end, observe that there exists $\epsilon > 0$ such that the function $f_{\cM} \colon B_{2\epsilon}(\bar x) \rightarrow \RR$, defined as the composition 
\begin{align}
f_{\cM}  := f\circ P_{\cM}
\end{align}
is $C^2$ and satisfies
$$
\nabla f_{\cM}(x) = \nabla_\cM f(x) \qquad \text{ and } \qquad \nabla^2 f_{\cM}(x) = \nabla_\cM^2 f(x) 
$$
for all $x \in B_{2\epsilon}(\bar x) \cap \cM$. Moreover, we may also assume that the projection map $P_{\cM} \colon B_{2\epsilon}(\bar x) \rightarrow \RR^d$ is $C^2$, in particular, Lipschitz with Lipschitz Jacobian. Throughout the proof, we assume that $\delta \leq \epsilon/4$ is small enough that conclusions of Propositions~\ref{prop:gettingclosertothemanifold} and~\ref{prop:shadow} are valid; we shrink $\delta$ several further times throughout the proof. In addition, we let $C$ denote a constant depending on $k_0$ and $\delta$, which may change from line to line.

Now, denote stopping time~\eqref{def:stoppingtime} by $\tau := \tau_{k_0, \delta}$ and the noise bound by $Q := \sup_{x \in B_{\delta}(\bar x)} q(x)$. Observe that by Proposition~\ref{prop:shadow}, the shadow sequence $y_k $ satisfies $y_k \in B_{4\delta}(x_k) \cap \cM \subseteq B_{\epsilon}(\bar x) \cap \cM$ and recursion holds:
$$
y_{k+1} = y_k - \alpha_k\nabla f_{\cM}(y_k) - \alpha_k P_{\tangentM{y_k}}(\perturb_k) + \alpha_k E_k.
$$
In addition, defining $$f^\ast := \inf_{x \in B_{\epsilon}(\bar x)}  f_{\cM}(x),$$ we have the bound $f^\ast1_{\tau > k} \leq f(y_k)1_{\tau > k}$ for all $k$. We now turn to the proofs.

To that end, fix a point $p \in S$ with associated manifold $\cM$ and neighborhood $\cU$. Let $\epsilon_p$ be small enough that $B_{\epsilon_p}(\bar x) \subseteq \cU$ and define the $C^2$ mapping $F_{p} \colon \ball{p}{\epsilon_p} \rightarrow \RR^d$ by:
$$
F_{p}(y) = - \nabla f_\cM(y),
$$
where $f_{\cM} := f \circ P_{\cM}$. Note that the mapping $F$ is indeed $C^2$, since $\cM$ is a $C^4$ manifold, and hence, $f_{\cM}$ is $C^3$. Moreover, since $\nabla F(p) = -\nabla_{\cM}^2 f(p)$, the mapping $F_p$ has at least one eigenvector with positive eigenvalue. In addition, the subspace $W_p$ spanned by such eigenvectors is contained in $\tangentM{p}$.

Turning to the proof, define $X_\discrete = x_\discrete$ for all $k \geq 1$. We now construct the sequences $Y_k, \xi_k,$ and $E_k$ and show they satisfy the assumptions of the theorem. Beginning with $Y_k$, recall that by Proposition~\ref{prop:shadow}, for all  $k \geq 1$ and all sufficiently small $\delta > 0$, the sequence 
\begin{align}\label{eq:Ykdefinition}
Y_{k}  := \begin{cases}
\proj_{\cM}(X_k) &\text{if $x_k \in B_{2\delta}(\bar x)$} \\
p & \text{otherwise.}
\end{cases},
\end{align}
satisfies $Y_k \in B_{4\delta}(\bar x) \cap \cM$ and the recursion 
\begin{align*}
Y_{k+1} &= Y_k - \alpha_k\nabla f_\cM(y_k) - \alpha_k \xi_k + \alpha_k E_k \qquad \text{for all $k \geq 1.$}
\end{align*}
where $\xi_k := P_{\tangentM{Y_k}}(\perturb_k)$ and $E_k$ is an error sequence. Moving to $E_k$, let us show that the error sequence satisfies the assumptions of the theorem. To that end, Proposition~\ref{prop:shadow} shows that  for $\delta$ sufficiently small, there exists $C > 0$ such that for all $n \geq k_0$, we have
$$
\EE\left[1_{\tau_{k_0, \delta} = \infty}  \sum_{k=n}^\infty  \alpha_k\|E_k\|\right] \leq C\sum_{k=n}^\infty \alpha_k^2.
$$
Moreover, by the Part \ref{prop:shadow:part:error:part:upperbound:1}  from Proposition~\ref{prop:shadow}, the sequence $\|E_k\|1_{\tau_{k_0, \delta} > k}$ is bounded above by a bounded sequence that almost surely converges to zero:
\begin{align*}
\|E_k\|1_{\tau_{k_0, \delta} > k} \leq C(1+\|\perturb_k\|)^2(\dist(x_k, \cM) + \alpha_k)1_{\tau_{k_0, \delta} > k} \leq C(1+r)^2(\delta + \alpha_k),
\end{align*}
Thus, on the event $\{\tau_{k_0, \delta} = \infty \}$, we have 
$$
\limsup_k 1_{\Omega_0}\EE[\|E_k\|^4\mid \cF_k] \leq \limsup_k \EE[\|E_k\|^41_{\tau_{k_0, \delta} > k}\mid \cF_k] \leq \left(C(1+r)^2(\delta + \alpha_k)\right)^4.
$$
Therefore, $Y_k$ and $E_k$ satisfy the  conditions~\ref{item:localiteration} and~\ref{item:error} of Theorem~\ref{thm: nonconvergence to unstable points} for all sufficiently small $\delta_p$ satisfying $\delta_p \leq \epsilon_p/8$.

To conclude the proof, we now show that Condition~\ref{item:noise} of Theorem~\ref{thm: nonconvergence to unstable points} is satisfied. To that end, clearly $\|\xi_\discrete\| = \|\proj_{T_\discrete}(\perturb_\discrete)\| \leq r =: c_3$ for all $k \geq \discrete_0$. In addition, we have that 
$$
\expect{\xi_k \mid \cF_k} = \proj_{T_\discrete}(\expect{\nu_k \mid X_{\discrete_0}, \ldots, X_{\discrete}}) = 0.
$$
Indeed, this follows from two facts: first $Y_{\discrete}$ is a measurable function of $X_\discrete$; and second the noise sequence $\perturb_\discrete$ is mean zero and independent of $X_{\discrete_0}, \ldots, X_{\discrete}$. Finally, we must show that $\xi_\discrete$ has positive correlation with the unstable subspace $W_p$. 

To prove correlation with the unstable subspace, recall that there exists $C' > 0$ such that the mapping $x \mapsto \proj_{\tangentM{x}}$ is $C'$-Lipschitz mapping on $\cM \cap  \ball{p}{\epsilon_p}$. In addition, we have that $W_p \subseteq \tangentM{p}$. Therefore, since $Y_\discrete \in   \cM\cap \ball{p}{\epsilon_p}$ for all $\discrete \geq \discrete_0$, we have the following bound for all $w \in W \cap \sphere^{d-1}$:
\begin{align*}
\EE[|\dotp{\xi_k, w}| \mid \cF_k] &= \EE[|\dotp{\perturb_\discrete, \proj_{\tangentM{Y_k}} w}| \mid \cF_k] \\
&\geq\EE[|\dotp{\perturb_\discrete, w}| \mid \cF_k] -r\|(\proj_{\tangentM{Y_k}} -\proj_{\tangentM{p}})w\|\\
&\geq rc_d - rC'\|Y_k - p\|,
\end{align*}
where $c_d$ is a constant dependent only on $d$ since $\perturb_k \sim \unif(B_r(0))$. By slightly shrinking $\epsilon_p$ if needed, we can ensure that $\inf_{x \in  \ball{p}{\epsilon_p}} \{rc_d - rC'\|x - p\|\} > (1/2)rc_d =: c_4$, as desired.

\bibliographystyle{plain}
	\bibliography{reference}
	\appendix
	\section{Appendix}
	\subsection{Proof of Proposition~\ref{prop:square_root}}
	Since $X$ is a $C^3$ manifold, the projection $P_Y$ is $C^2$-smooth. Therefore, there exist constants $\epsilon,L>0$ satisfying
	\begin{equation}\label{eqn:proj_lip}
	\|P_Y(y+h)-P_Y(y)-\nabla P_Y(y)h\|\leq L\|h\|^2
	\end{equation}
	for all $y\in B_{\epsilon}(\bar x)$ and $h\in \epsilon \mathbb{B}$.
	Fix now two points $x \in X$ and $y \in Y$ and a unit vector  $v \in N_X(x)$. Clearly, we may suppose $v\notin N_{Y}(y)$, since otherwise the claim is trivially true. Define the normalized vector $w := -\frac{P_{\tangent{Y}{y}}( v)}{\|P_{\tangent{Y}{y}}( v)\|}$. 
Noting the equality $\nabla P_Y(y)=P_{T_Y(y)}$ and appealing to \eqref{eqn:proj_lip}, we deduce the estimate
	\begin{align*}
	\|P_{Y}(y - \alpha w) -  (y - \alpha w)\| \leq L\| \alpha w\|^2 = L\alpha^2,
	\end{align*}
	for all $y\in B_{\epsilon}(\bar x)$ and $\alpha\in (0,\epsilon)$.
 	Shrinking $\epsilon>0$, prox-regularity yields the estimate
	\begin{align*}
	\dotp{v, P_{Y}(y - \alpha w)-x} \leq \frac{\rho}{2}\|x-P_{Y}(y - \alpha w)\|^2,
	\end{align*}
	for some constant $\rho>0$.
	Therefore, we conclude 
	\begin{align*}
	\alpha\|P_{\tangent{Y}{y}}v\| = -\alpha\dotp{v, w}  
	&= \dotp{v,  x-y} + \dotp{v, P_{Y}(y - \alpha w)-x} + \dotp{v,  (y - \alpha w)- P_{Y}(y - \alpha w)} \\
	&\leq \|x - y\| + \frac{\rho}{2}\|x-P_{Y}(x - \alpha w)\|^2 + L\alpha^2.
	\end{align*}
	Note that the middle term is small: 
	$$
	\| P_{Y}(y - \alpha w)-x\|^2 \leq 2\|P_{Y}(y - \alpha w) - (y - \alpha w) \|^2 + 2\|y - \alpha w - x\|^2 \leq 2L^2\alpha^4 + 4\|y - x\|^2 + 4\alpha^2.
	$$
	Thus, we have 
	\begin{align*}
	\alpha \|P_{\tangent{Y}{y}}v\| &\leq  \|x - y\| + \rho L^2\alpha^4 + 2\rho \|x - y\|^2 + 2\rho\alpha^2 + L\alpha^2.
	\end{align*}
	Dividing both sides by $\alpha$ and setting $\alpha = \sqrt{\|x - y\|}$ completes the proof of \eqref{eqn:sqrt_estimate}.

\subsection{Proof of Proposition~\ref{prop:projectedgradient}: the projected gradient method}\label{section:proof:prop:projectedgradient}
Choose $\epsilon > 0$ small enough that the following hold for all $x \in B_{\epsilon}(\bar x) \cap \cX$. First \eqref{prop:projectedgradient:eq:aiming} holds. 
Second we require that for some $L > 0$, we have
\begin{align}
	\|P_{\tangentM{P_{\cM}(x)}}(s_g(x) - \nabla_{\cM} g(P_{\cM}(x))\| &\leq L\dist(x, \cM),\label{eqn:dumbo0}\\
	\|P_{\tangentM{z}}(u)\| &\leq L \|x-z\|, \label{eqn:dumbo}
\end{align}
for all $u \in N_{\cX}(x)$ of unit norm and all $z\in B_{\epsilon}(\bar x)\cap\cM$, a consequence of \ref{assumption:projectedgradient:stronga}. Third, given an arbitrary $\delta\in (0,1)$ we may choose $\epsilon>0$ so small so that
\begin{equation}\label{eqn:blah_prox}
	\dotp{z, x - x'} \geq -o(\|x-x'\|) 
\end{equation}
for all $z \in N_{\cX}(x)$ of unit norm, and $x'\in \cM\cap B_{\epsilon}(\bar x)$---a consequence of \ref{assumption:projectedgradient:bproxregularity}. 
We will fix $x\in B_{\epsilon/2}(\bar x)\cap \cX$ and arbitrary $\alpha > 0$ and $\nu \in \RR^d$, and choose an arbitrary $y\in\proj_{\cM}(x)$. Define 
\begin{align*}
	w = G_\alpha(x, \perturb) - \perturb-s_g(x)  \qquad \text{ and } 
	\qquad x_+ = s_\cX(x - \alpha(s_g(x)+ \perturb)). 
\end{align*}
Note the inclusion $w\in N_{\cX}(x^+)$. Next, to verify \ref{assumption:localbound},  we compute 
\begin{align*}
	\alpha\|G_{\alpha}(x, \perturb)\|&=\|x-s_{\cX}(x-\alpha(s_g(x)+\nu))\|\\
	&\leq \dist_{\cX}(x-\alpha(s_g(x)+ \perturb))+\alpha\|s_g(x)+ \perturb\|\\
	&\leq 2\alpha\|s_g(x)+ \perturb\|=  O(\alpha (1 +\|\perturb\|)).
\end{align*}
Thus there exists a constant $C>0$ satisfying
$$
\max\{\|w\|,\|G_{\alpha}(x, \perturb)\|\} \leq C(1 + \|\perturb\|)  \qquad \text{and} \qquad \|x_+ - x\| \leq C(1 + \|\perturb\|)\alpha.
$$
We will use these estimates often in the proof. Finally, we let $C$ be a constant independent of $x, \alpha$ and $\nu$, which changes from line to line.

\noindent\underline{Assumption~\ref{assumption:smoothcompatibility}:}
Suppose first that $x_+\in B_{\epsilon}(\bar x)$. Using \eqref{eqn:dumbo}, we compute
\begin{align}\label{eq:projgradientverdierbound1}
	\|P_{\tangentM{P_{\cM}(x)}}w\|&\leq L\|w\| \|x_+ - P_{\cM}(x)\|\notag \\
	&\leq L\|w\|(\|x_+ - x\| + \dist(x, \cM)) \notag \\
	&\leq \localconstant(1+ \|\perturb\|)^2\alpha + \localconstant(1 + \|\perturb\|)\dist(x, \cM).
\end{align}
On the other hand, if $x_+\notin B_{\epsilon}(\bar x)$, then we compute \begin{align}\label{eq:projgradientverdierbound3}
	\|P_{\tangentM{P_{\cM}(x)}}w\| \leq \|w\| \leq \frac{2}{\epsilon}\|w\|\|x_+ - x\| \leq \localconstant(1 + \|\perturb\|)^2\alpha.
\end{align}
In either case, Assumption~\ref{assumption:smoothcompatibility} now follows since from \eqref{eqn:dumbo0} we have
\begin{align*}
	\|P_{\tangentM{P_{\cM}(x)}}(s_g(x) - \nabla_\cM g(P_{\cM}(x))\|	&\le \localconstant \dist(x, \cM),
\end{align*}
as we had to show.

\noindent\underline{Assumption~\ref{assumption:aiming}:} We write the decomposition
\begin{equation}\label{eqn:decomp}
	\langle G_{\alpha}(x,\nu)-\nu, x-y\rangle=\underbrace{\langle s_g(x),x-y\rangle}_{R_1}+\underbrace{\langle w,x_+-y\rangle}_{R_2}+\underbrace{\langle w,x-x_+\rangle}_{R_3}.
\end{equation}
The aiming condition \ref{assumption:projectedgradient:proximalaiming} ensures 
\begin{equation}\label{eqn:dumb3a}
	\begin{aligned}
		R_1& \geq \mu\cdot\dist(x,\cM).
	\end{aligned}
\end{equation}
We next look at two cases. Suppose first $x_+\in B_{\epsilon}(\bar x)$ and therefore $\|x_+ - x\| \geq \epsilon/2$.
Using the inclusion $w\in N_{\cX}(x_+)$ and Assumption~\ref{assumption:projectedgradient:bproxregularity}, we compute 
\begin{align}
	R_2\geq -\|w\|\cdot o(\|x_+-y\|)&\geq -\|w\|\cdot (o(\|y-x\|)+\|x-x_+\|)\notag\\
	&\geq -C(1+\|\nu\|)^2(o(\dist(x,\cM))+\alpha)\label{eqn:dumb1a}.
\end{align}
Next, the Cauchy–Schwarz inequality implies
\begin{equation}\label{eqn:dumb2a}
	|R_3|=\|w\| \|x-x_+\|\leq C(\alpha(1+\|\nu\|)^2).
\end{equation}
Combining \eqref{eqn:decomp}-\eqref{eqn:dumb2a} yields the claimed bound \ref{assumption:aiming}.

Suppose now on the contrary that $x_+\notin B_{\epsilon}(\bar x)$ and therefore $\|x-y\|\leq \|x-x_+\|$. We thus deduce $R_2+R_3=\langle w,x-y\rangle\geq -\|w\|\|x-y\|\geq -C\alpha(1+\|\nu\|)^2$ holds. Combining this estimate with \eqref{eqn:decomp} and \eqref{eqn:dumb3a} verifies the claim \ref{assumption:aiming}.

\subsection{Proof of Proposition~\ref{prop:proximalgradient}: the proximal gradient method}\label{section:proof:prop:proximalgradient}

Let $\epsilon \in (0,1)$ be small enough such that the following hold for all $x \in B_{\epsilon}(\bar x) \cap \dom f$. First, \eqref{prop:proximalgradient:eq:aiming} holds and therefore:
\begin{equation}\label{eqn:anglea}
	\dotp{\nabla g(x)+v, x - P_{\cM}(x)} \geq \mu\cdot  \dist(x, \cM) - (1+\|v\|)o(\dist(x, \cM)),
\end{equation}
for all $v \in \hat\partial h(x)$.
Second we require that for some $L > 0$, we have
\begin{equation}\label{eqn:basic_est_needed}
	\|P_{\tangentM{P_{\cM}(x)}}(u - \nabla_{\cM} h(P_{\cM}(x))\| \leq L\sqrt{1+ \|u\|^2}\cdot\dist(x, \cM)
\end{equation}
for all $u \in \partial h(x)$, a consequence of strong (a) regularity. Third, we assume that $\nabla_{\cM} f$ is $L$-Lipschitz on $B_{\epsilon}(\bar x) \cap \cM$. Fourth, we assume that $\nabla g(\cdot)$ is $L$-Lipschitz. Shrinking $\epsilon$ we may moreover assume $\epsilon\leq \frac{\mu}{4L}$.
Finally, we may also assume that the assignments $P_{\cM}$ is $L$-Lipschitz  on $B_{\epsilon}(\bar x)$ and that the map $x\mapsto P_{T_{\cM}(P_{\cM}(x))}(\cdot)$ is $L$-Lipschitz on $B_{\epsilon}(\bar x)$ with respect to the operator norm.

Fix $x \in B_{\epsilon/2}(\bar x)\cap \dom f$ and $\nu \in \RR^d$ and set $y:=P_{\cM}(x)$. We define the vectors 
\begin{align*}
	w = G_\alpha(x, \perturb) - \nabla g(x) - \perturb  \qquad \text{ and } 
	\qquad x_+ = s_\alpha(x - \alpha(\nabla g(x)+ \perturb)). 
\end{align*}
\begin{claim}\label{eq:claiminitialboundproxgradient}
	We have $w \in \hat \partial h(x_+)$ and there exists a constant $C$ independent of $x, \perturb, \alpha$, such that the following bounds hold:
	$$
	\max\{\|G_{\alpha}(x, \perturb)\|, \|w\|\} \leq C(1 + \|\perturb\|) ; \qquad \text{and} \qquad \|x_+ - x\| \leq C(1 + \|\perturb\|)\alpha.
	$$
\end{claim}
\begin{proof}
	Beginning with the inclusion, first-order optimality conditions imply that $w$ is a Fr{\'e}chet subgradient:
	\begin{align*}
		w = \frac{x - \alpha (\nabla g(x) + \perturb) - x_{+}}{\alpha} \in \hat \partial h(x_+),
	\end{align*}
	as desired. First, we bound $\|x_+ - x\|$: Let $v = \nabla g(x) + \perturb$ and observe from the very definition of $x^+$ that there exists $ C> 0$ such that 
	\begin{align*}
		\frac{1}{2\alpha}\|x_+ - x\|^2  &\leq h(x) - h(x_+) - \dotp{v, x_+ - x} \leq \localconstant\|x_+ - x\| + \|v\|\|x_+ -x\|.
	\end{align*}
	Consequently, we have 
	$
	\|x_+ - x\| \leq (2\localconstant +2 \|v\|)\alpha \leq 2(2\localconstant + \|\perturb\|)\alpha,
	$
	as desired.
	Second, the bound on $G_{\alpha}(x, \perturb)$ follows trivially from the computation
	$$
	\|G_{\alpha}(x, \perturb)\| = \|x_+ - x\|/\alpha \leq 2(2\localconstant + \|\perturb\|).
	$$
	Finally, we bound $\|w\|$ using the estimate
	$$
	\|w\| = \|x - x_+\|/\alpha + \|\nabla g(x) + \perturb\|  \leq 4(2\localconstant + \|\perturb\|),
	$$
	as desired.
\end{proof}

We will use the estimates in the claim often in the proof. Finally, we let $C$ be a constant independent of $x, \alpha$ and $\nu$, which changes from line to line.

\noindent\underline{Assumption~\ref{assumption:smoothcompatibility}:} 
First suppose $x_+\in B_{\epsilon}(\bar x)$. Using the triangle inequality, we write 
\begin{align*}
	&\|P_{\tangentM{P_{\cM}(x)}}(G_\alpha(x, \perturb) - \nabla_{\cM}f(P_{\cM}(x)) - \nu)\|\\
	&=\|P_{\tangentM{P_{\cM}(x)}}(w+\nabla g(x)- \nabla g(P_{\cM}(x))-\nabla_{\cM} h(P_{\cM}(x)))\|\\
	&\leq \underbrace{\|P_{\tangentM{P_{\cM}(x)}}(w-\nabla_{\cM} h(P_{\cM}(x_+)))\|}_{R_1}+\underbrace{\|\nabla g(x)- \nabla g(P_{\cM}(x))\|}_{R_2}+\underbrace{\|\nabla_{\cM} h(P_{\cM}(x)))-\nabla_{\cM} h(P_{\cM}(x_+)))\|}_{R_3}.
\end{align*}
Taking into account that the assignment $x\mapsto P_{T_{\cM}(P_{\cM}(x))}(\cdot)$ is Lipschitz with respect to the operator norm, the estimate \eqref{eqn:basic_est_needed} implies
\begin{align*}
	R_1&\leq L\sqrt{1+\|w\|^2}\cdot\dist(x_+,\cM)+L\|x-x_+\|\|w-\nabla_{\cM}h(P_{\cM}(x_+))\|\\
	&\leq C(1+\|\nu\|)\dist(x_+,\cM)+L(1+\|\nu\|)^2\alpha\\
	&\leq C(1+\|\nu\|)(\dist(x,\cM)+C\|x-x_+\|)+L(1+\|\nu\|)^2\alpha\\
	&\leq C(1+\|\nu\|)\dist(x,\cM)+C(1+\|\nu\|)^2\alpha.
\end{align*}
Moreover, clearly we have
$R_2\leq C\dist(x,\cM)$
and 
$R_3\leq C\|x-x_+\|\leq (1+\|\nu\|)\alpha.$ Condition~\ref{assumption:smoothcompatibility} follows immediately.

Now suppose that $x_+ \notin B_{\epsilon}(\bar x)$, and therefore $\|x_+ - x\| \geq \epsilon/2$. Then, we may write
\begin{align*}
	\|P_{\tangentM{P_{\cM}(x)}}(G_{\alpha}(x, \perturb)- \perturb - \nabla f_{\cM}(P_{\cM}(x)))\| &\leq \|G_{\alpha}(x, \perturb)\| +  \|\perturb\| +  \|\nabla f_{\cM}(P_{\cM}(x))\|\\
	&\leq \frac{2}{\epsilon}(\|G_{\alpha}(x, \perturb)\| +  \|\perturb\| +  \|\nabla f_{\cM}(P_{\cM}(x))\|)\|x - x_+\| \\
	&\leq C(1+\|\perturb\|)^2\alpha,
\end{align*}
as desired.

\noindent\underline{Assumption~\ref{assumption:aiming}:} We begin with the decomposition 
\begin{align*}
	\langle G_{\alpha}(x,\nu)-\nu, x-y\rangle=&\underbrace{\langle \nabla g(x_+)+w,x_+-P_{\cM}(x_+)\rangle}_{R_1}\\
	&+\underbrace{\langle \nabla g(x)-\nabla g(x_+),x-y\rangle}_{R_2}+\underbrace{\langle \nabla g(x_+)+w, (x-P_{\cM}(x))-(x_+-P_{\cM}(x_+))\rangle}_{R_3}.
\end{align*}
We now bound the two terms on the right in the case $x_+\in B_{\epsilon}(\bar x)$. Using \eqref{eqn:angle}, we estimate
\begin{align*}
	R_1&\geq \mu\cdot  \dist(x_+, \cM) - (1+\|v\|)o(\dist(x_+, \cM))\\
	&\geq \mu\cdot  (\dist(x, \cM)-\|x-x_+\|) - (1+\|v\|)(o(\dist(x, \cM)) +\|x-x_+\|)\\
	&\geq \mu  \cdot\dist(x, \cM)-(1+\|\nu\|)^2(o(\dist(x, \cM))+C\alpha).
\end{align*}
Next, we compute
$$|R_2|\leq \|\nabla g(x)-\nabla g(x_+)\|\cdot\dist(x,\cM)\leq 2L\epsilon\cdot \dist(x,\cM)\leq \frac{\mu}{2} \dist(x,\cM).$$
Next using Lipschitz continuity of the map $I-P_{\cM}$ on $B_{\epsilon}(\bar x)$,  we deduce
$$|R_3|\leq (1+L)\|\nabla g(x_+)+w\|\cdot\|x-x_+\|\leq  C(1+\|\nu\|)^2\alpha.$$
The claimed proximal aiming condition follows immediately.

Let us look now at the case $x_+ \notin B_{\epsilon}(\bar x)$, and therefore $\dist(x,\cM)\leq \frac{\epsilon}{2}\leq  \|x-x^+\|$.
Then we compute \begin{align*}
	\dotp{G_{\alpha}(x, \perturb) - \perturb, x - P_{\cM}(x)} &\geq -\dist(x,\cM)\cdot \|G_{\alpha}(x, \perturb) - \perturb\|\\
	&= \dist(x, \cM) - \dist(x, \cM)(1 + \|G_{\alpha}(x, \perturb) - \perturb\|) \\
	&\geq\dist(x, \cM) - C\|x - x_+\| (1 + C(1+\|\perturb\|)) \\
	&\geq \dist(x, \cM) -C(1+\|\perturb\|)^2 \alpha,
\end{align*}
as desired. The proof is complete.

\subsection{Proof of Corollary~\ref{cor:activestrictsaddleregular}: avoiding active strict saddle via projected subgradient method}\label{appendix:cor:activestrictsaddleregular}

By Proposition~\ref{prop:projectedgradient} we need only show that Assumption~\ref{assumption:projectedgradient} holds. To that end, note that Assumptions\ref{assumption:projectedgradient:stronga} and ~\ref{assumption:projectedgradient:bproxregularity} hold by assumption. Next we prove ~\ref{assumption:projectedgradient:proximalaiming}. Note that if $g$ satisfies $(b_{\leq})$ along $\cM$, then~\ref{assumption:projectedgradient:proximalaiming} holds by Corollary~\ref{cor:prox-aiming_gen}. Next, suppose that $g$ is weakly convex around $x$. In this case, since each $x \in S$ is Fr{\'e}chet critical and $\cM_x$ is an active manifold, it follows by Proposition~\ref{prop: sharpness} that for some $\mu > 0$, we have
$$
g(y) - g(\proj_{\cM_x}(y)) \geq \mu\dist(y,\cM),
$$ 
near $x$. Consequently, for all $v \in \partial_c g(x)$, we have 
$$
\dotp{v, y - \proj_{\cM_x}(y)} \geq g(y) - g(\proj_{\cM_x}(y)) - O(\|x - y\|^2) \geq (\mu/2)\dist(x, \cM), 
$$
for all $y$ near $x$, verifying~\ref{assumption:projectedgradient:proximalaiming}.

 \subsection{Proof of Corollary~\ref{cor:activestrictsaddleregularproximalgradient}: avoiding active strict saddle via proximal gradient method}\label{appendix:cor:activestrictsaddleregularproximalgradient}
 
 By Proposition~\ref{prop:proximalgradient}, we need only show that Assumption~\ref{assumption:proximalgradient} holds. Note that~\ref{assumption:proximalgradient:Lipschitz1},~\ref{assumption:proximalgradient:Lipschitz2}, and~\ref{assumption:proximalgradient:stronga} hold by assumption. Thus, we need only verify~\ref{assumption:proximalgradient:proximalaiming}, which is immediate from $(b_{\leq})$-regularity and Corollary~\ref{cor:prox-aiming_gen}.

\subsection{Proofs of Corollaries~\ref{thm:globalsemialgebraic},~\ref{thm:globalsemialgebraic2}, and~\ref{thm:globalsemialgebraic3}: saddle point avoidance for generic semialgebraic problems.}\label{sec:genericsaddleavoidance}

We first claim that the collection of limit points for all three methods is a connected set of composite Clarke critical points. To that end, note that by \cite[Theorem 6.2/Corollary 6.4]{davis2020stochastic}, we know that for each method, on the event the sequence $x_k$ is bounded, all limit points are composite Clarke critical. We claim that the set of limit points is in fact connected. Indeed, by~\cite[Lemma 5(iii)]{palm}, this will follow if 
$$
\lim_{k \rightarrow 0}\|x_{k+1} - x_k\| = \lim_{k \rightarrow 0}\|\alpha_kG_{\alpha_k}(x_k, \perturb_k)\| = 0.
$$ 
This in turn follows from \cite[Lemma A.4, A.5, and A.6]{davis2020stochastic}, which shows that $G_{\alpha_k}(x_k, \perturb_k) = w_k + \xi_k$, where $w_k$ is bounded and $\sum_{k=1}^\infty \alpha_k\xi_k$ exists almost surely. Consequently, we have $\|\alpha_kG_{\alpha_k}(x_k, \perturb_k)\| = \alpha_k\|w_k + \xi_k\| \rightarrow 0$ almost surely, as desired.

Next we claim that the sequence $x_k$ converges for all three methods. Indeed, by Corollaries~\ref{prop:subgradientsemialgebraic},~\ref{prop:projectedgradientsemialgebraic}, and~\ref{prop:proximalgradientsemialgebraic}, it follows that each of the set of composite Clarke critical points for all three problems is finite for generic semialgebraic problems. Therefore, since the set of limit points of $x_k$ is connected and discrete, it follows that on the event the sequence $x_k$ is bounded, it must converge to a composite Clarke critical point. 

To wrap up the proof, suppose that $x_k$ converges to a composite limiting critical point. Then by  Corollaries~\ref{prop:subgradientsemialgebraic},~\ref{prop:projectedgradientsemialgebraic}, and~\ref{prop:proximalgradientsemialgebraic} for any  of the three methods, every composite limiting critical point of $f$ is a composite Fr{\'e}chet critical point which is either a local minimizer or an active strict saddle point at which Assumption~\ref{assumption:A} holds along the active manifold. By Theorem~\ref{thm:avoidance}, the sequence $x_k$ can converge to the such active strict saddle points only with probability zero. Therefore, the limit point must be a local minimizer, as desired.

\subsection{Sequences and Stochastic Processes}\label{sec:stochasticprocess}

\subsubsection{Lemmas from other works.}

\begin{lem}[Robbins-Siegmund\cite{robbins1971convergence}]\label{lem: Robbins-Siegmund}
Let $A_k, B_k,C_k,D_k \ge 0$ be non-negative random variables adapted to the filtration  $\{\cF_{k}\}$ and satisfying 
$$
\EE[A_{k+1}\mid \mathcal{F}_k] \le (1+B_k)A_k + C_k - D_k.
$$
Then on the event $\{\sum_{k}B_k <\infty, \sum_{k}C_k < \infty\}$, there is a random variable $A_\infty<\infty$ such that $A_k \xrightarrow{\text{a.s.}} A_\infty$ and $\sum_k D_k <\infty$ almost surely. 
\end{lem}

\begin{lem}[Conditional Borel-Cantelli~{\cite{condborellcantelli}}]\label{lem:condborelcantelli}
Let $\{X_n \colon n \ge 1\}$ be a sequence of nonnegative random variables defined on the probability space $(\Omega, \cF, \mathbb{P})$ and  $\{\cF_n \colon n \ge 0\}$ be a sequence of sub-$\sigma$-algebras of $\cF$. Let $M_n = \expect{X_n \mid \cF_{n-1}}$ for $n\ge 1$. If $\{\cF_n \colon n \ge 0\}$ is nondecreasing, i.e., it is a filtration, then $\sum_{n=1}^{\infty} X_n < \infty$ almost surely on $\{\sum_{n=1}^{\infty}M_n <\infty\}$.
\end{lem}

\begin{lem}[{\cite[Theorem A]{AIHPB_1996__32_3_395_0}}]\label{lem:nonmeasruable}
Let $\{\cF_k\}$ be a filtration and let $\{\epsilon_k\}$ be a sequence of random variables adapted to $\{\cF_k\}$ satisfying for all $k$ the bound
$$
\EE[\epsilon_{k+1}^2 \mid \cF_k] < \infty \qquad \text{ and } \qquad \EE[\epsilon_{k+1}\mid \cF_k] = 0.
$$
Let $\{\Phi_k\}_k$ be another sequence of random variables adapted to $\{\cF_k\}$. Let $\{c_k\}$ be a deterministic sequence that is square summable but not summable. Suppose that the following hold almost surely on an event $H$: 
\begin{itemize}
\item We have the Marcinkiewick-Zygmund conditions: 
\begin{align*}
 \limsup_{k}~ \EE[\epsilon_{k+1}^2 \mid \cF_k] < \infty \qquad \text{ and } \qquad  \liminf_{k}~ \EE[|\epsilon_{k+1}| \mid \cF_k] > 0.
\end{align*}
\item There exists sequences of random variables $\{r_k\}$ and $\{R_k\}$, adapted to $\{\cF_k\}$ such that $\Phi_k = r_k + R_k$ and 
\begin{align*}
\sum_{k} \|r_k\|^2 < \infty \qquad \text{ and } \qquad \EE\left[ 1_{H}\sum_{k = K}^\infty c_k|R_k|\right] = o\left(\left(\sum_{k=K}^\infty c_k^2\right)^{1/2}\right).
\end{align*}
\end{itemize}
Then on $H$ the series $\sum_{k=1}^\infty c_k(\Phi_k + \epsilon_k)$ converges almost surely to a finite random variable $L$.  Moreover,  for any $p \in \NN$ and any $\cF_p$-measurable random variable $Y$ we have 
$$
P(H \cap (L = Y)) = 0.
$$
\end{lem}

\begin{lem}[{\cite[Exercise 5.3.35]{dembo2016lecture}}]\label{lem: L2martingaleconvergence}
	Let  $M_k$ be an $L^2$ martingale adapted to a filtration $\{\cF_k\}$ and let $b_k \uparrow \infty$ be a positive deterministic sequence. Then if 
	$$\sum_{k\ge 1}b_k^{-2}\expect{(M_k - M_{k-1})^2 \mid \cF_{k-1}} < +\infty,$$
	we have $b_n^{-1}M_n \xrightarrow{\text{a.s.}} 0$.
\end{lem}

\begin{lem}[Kronecker Lemma]\label{lem:kronecker}
Suppose $\{x_k\}_k$ is an infinite sequence of real number such that the sum
$
\sum_{k=1}^{\infty} x_k 
$ 
exists and is finite. Then for any divergent positive nondecreasing sequence $\{b_k\}$, we have 
$$
\lim_{K \rightarrow \infty}\frac{1}{b_K}\sum_{k=1}^K b_k x_k = 0.
$$
\end{lem}

\subsubsection{Lemmas proved in this work}
We will use the following two Lemmas on sequences. The proof of the following Lemma may be found in Appendix~\ref{sec:proof:lem:fastsummability}. 
\begin{lem}\label{lem:fastsummability}
Fix $k_0 \in \NN, c > 0$, and $\gamma \in (1/2, 1]$. Suppose that $\{X_k\}, \{Y_k\},$ and $\{Z_k\}$ are nonnegative random variables adapted to a filtration $\{\cF_k\}$. Suppose the relationship holds:
$$
\EE[X_{k+1}\mid\mathcal{F}_k] \leq (1-ck^{-\gamma})X_k - Y_k + Z_k \qquad \text{for all $k \geq k_0$}.
$$
Assume furthermore that $c \geq 6$ if $\gamma = 1$. Define the constants $a_k:=\frac{k^{2\gamma-1}}{\log^2(k+1)}$. Then there exists a random variable $V < \infty$ such that  on the event $\{\sum_{k=1}^\infty a_{k+1}Z_k < + \infty\}$, the following is true:
\begin{enumerate}
\item The limit holds
$$
 a_kX_k\xrightarrow{\text{a.s.}} V.
$$
\item The sum is finite
$$
  \sum_{k=1}^\infty a_{k+1}Y_k <  +\infty.
$$
\end{enumerate}
\end{lem} 

The proof of the following Lemma may be found in Appendix~\ref{sec:proof:lem:sequencelemmadistance}.
\begin{lem}\label{lem:sequencelemmadistance}
	Fix $k_0 \in \NN$, $c, C > 0$, and $\gamma \in (1/2, 1]$. Suppose that $\{s_k\}_{k}$ is a nonnegative sequence satisfying 
	\begin{align*}
	s_{k} \le  \frac{c}{12\gamma} \qquad \text{ and } \qquad s_{k+1}^2 \leq s_k^2 - c k^{-\gamma} s_k + Ck^{-2\gamma},  \qquad \text{for all $k \geq k_0$,}
	\end{align*}
Then, there exists a constant $C_{\ub}$ depending only on $c, C, \gamma$ and $k_0$ such that  
	$$
	s_{k} \leq C_{\ub}k^{-\gamma}, \qquad \forall k \ge 1.
	$$
\end{lem}

The proof of the following Lemma may be found in Appendix~\ref{sec:proof:lem:sequencelemmasquared}.
\begin{lem}\label{lem:sequencelemmasquared}
	Fix $k_0 \in \NN$, $c, C > 0$, and $\gamma \in (1/2, 1]$. Suppose that $\{s_k\}_{k}$ is a nonnegative sequence satisfying 
	\begin{align*}
 \qquad s_{k+1} \leq (1-ck^{-\gamma})s_k  + Ck^{-2\gamma},  \qquad \text{for all $k \geq k_0$,}
	\end{align*}
Assume furthermore that $c \ge 16$ if $\gamma=1$. Then, there exists a constant $C_{\ub}$ depending only on $c, C, \gamma$ and $k_0$ such that  
	$$
	s_{k} \leq C_{\ub}k^{-\gamma}, \qquad \forall k \ge 1.
	$$
\end{lem}

\subsection{Proof of Lemma~\ref{lem:fastsummability}}\label{sec:proof:lem:fastsummability}
\begin{proof}
For all $k \geq 0$, define $a_k := \frac{k^{2\gamma - 1}}{\log(k+1)^2}$ and observe that 
\begin{align*}
\EE[a_{k+1}X_{k+1}\mid \mathcal{F}_k]  \leq a_{k+1} (1-ck^{-\gamma}) X_k - a_{k+1} Y_k + a_{k+1} Z_k \qquad \text{for all $k \geq k_0$}.
\end{align*}
Thus, the result will follow from Robbins-Siegmund Lemma~\ref{lem: Robbins-Siegmund} if $a_{k+1}(1-ck^{-\gamma}) \leq a_{k}$ for all sufficiently large $k$. To that end, notice that for sufficiently large $k$, we have 
$$
\left(\frac{k+1}{k}\right)^{2\gamma-1} \leq 1 + \frac{2(2\gamma-1)}{k}.
$$
Therefore,
\begin{align*}
\frac{a_{k+1}}{a_k} \leq 1 + \frac{2(2\gamma - 1)}{k} \qquad \text{for all sufficiently large $k$.}
\end{align*}
Now we deal separately with the cases $\gamma < 1$ and $\gamma = 1$. First suppose  that $\gamma < 1$. Then there exists a constant $C' > 0$ such that 
$$
\frac{1}{1-ck^{-\gamma}} \geq 1+ \frac{C'}{k^{\gamma}},\qquad \text{for all sufficiently large $k$.}
$$
Consequently, $a_{k+1}/a_k \leq (1-ck^{-\gamma})^{-1}$ for all sufficiently large $k$, as desired.

Now assume that $\gamma = 1$. Then we compute
$$
\frac{1}{1-ck^{-1}} \geq 1+ \frac{c}{2k},\qquad \text{for all sufficiently large $k$.}
$$
Consequently, $a_{k+1}/a_k \leq (1 - ck^{-1})^{-1}$ for all large $k$, provided that  $c \geq 6$.
\end{proof}

\subsubsection{Proof of Lemma~\ref{lem:sequencelemmadistance}}\label{sec:proof:lem:sequencelemmadistance}

	\begin{proof}
It suffices to exhibit $C_\ub'>0$ such that 
	$$
s_{k} \leq C_{\ub}'k^{-\gamma} \qquad \text{ for all sufficiently large $k \geq k_0$}.
$$
To that end, choose $k_1$ large enough that the following two bounds hold: 
\begin{enumerate}
\item 
$C_\ub' :=\max \left\{  \frac{c k_1^\gamma}{12\gamma}, 2\sqrt{C}, \frac{4C}{c}\right\} =  \frac{c k_1^\gamma}{12\gamma} 
$
\item$\paren{\frac{k_1+1}{k_1}}^{2\gamma}\le \min\left\{2, 1+ \frac{3\gamma}{k_1}\right\}$.
\end{enumerate}
Then by assumption, we have
\begin{align*}
 k^{2\gamma}s_{k+1}^2 &\leq k^{2\gamma}s_k^2 -  c k^{\gamma}s_k + C  \qquad \text{for all $k \geq k_1$.}
\end{align*}
Denoting $t_k :=  k^\gamma s_k$, we obtain the following bound for all $k \geq k_1$:
\begin{align}\label{eq: tkupdate}
	t_{k+1}^2 \le   \paren{\frac{k+1}{k}}^{2\gamma}(t_k^2 - c t_k + C)\le  \paren{\frac{k_1+1}{k_1}}^{2\gamma}(t_k^2 - c t_k + C).
\end{align}
Thus the claim will follow if $t_k \le C_\ub'$ for all $k \geq k_1$. We prove the claim by induction. First the case $k=k_1$ holds by definition of $C_\ub'$. 
Now suppose $t_k \le C_\ub'$ for some $k \geq k_1$ and consider two cases

First suppose $t_k \in [0, \frac{1}{2} C_\ub']$. By~\eqref{eq: tkupdate} and definition of $C_\ub'$, we have 
	\begin{align*}
		t_{k+1}^2 &\le  \paren{\frac{k_1+1}{k_1}}^{2\gamma}(t_k^2 + C)\\
		&\le \paren{\frac{k_1+1}{k_1}}^{2\gamma}\paren{\frac{1}{4}C_\ub'^2 + \frac{1}{4}C_\ub'^2}\\
		&\le C_\ub'^2.
	\end{align*}

Second, suppose $t_k \in [\frac{1}{2} C_\ub', C_\ub']$. By~\eqref{eq: tkupdate} and definition of $C_\ub'$, we have 
	\begin{align*}
	t_{k+1}^2 &\le \paren{\frac{k_1+1}{k_1}}^{2\gamma}(t_k^2 - c t_k + C) \\ 
	&\le \paren{\frac{k_1+1}{k_1}}^{2\gamma}\left(C_{\ub}'^2 -\frac{c C_\ub'}{2}+ C\right)\\
	&\le \paren{\frac{k_1+1}{k_1}}^{2\gamma}\left(C_{\ub}'^2 -\frac{cC_\ub'}{4}\right)\\
	&= C_\ub' \paren{\frac{k_1+1}{k_1}}^{2\gamma}\paren{C_\ub' -\frac{c}{4}}
	\end{align*}
We claim that $\paren{\frac{k_1+1}{k_1}}^{2\gamma}\paren{C_\ub' -\frac{c}{4}}\le C_\ub'$. Indeed, we have 
\begin{align*}
	&\paren{\frac{k_1+1}{k_1}}^{2\gamma}\paren{C_\ub' -\frac{c}{4}}\\
	&\le \paren{1+ \frac{3\gamma}{k_1}}\paren{C_\ub' -\frac{c}{4}}\\
	&\le C_\ub' + \frac{3\gamma C_\ub'}{k_1} - \frac{c}{4}\\
	&\le C_\ub' + \frac{c}{4k_1^{1-\gamma}} - \frac{c}{4}\\	
	&\leq  C_\ub',
\end{align*}
as desired. This completes the induction.
\end{proof}

\subsubsection{Proof of Lemma~\ref{lem:sequencelemmasquared}}\label{sec:proof:lem:sequencelemmasquared}
\begin{proof}
It suffices to exhibit $C_\ub>0$ such that 
	$$
s_{k} \leq C_{\ub}k^{-\gamma} \qquad \text{ for all sufficiently large $k \geq k_0$}.
$$
To that end, choose $k_1$ large enough that the following two bounds hold: 
\begin{enumerate}
	\item  $\left(\frac{k+1}{k}\right)^{\gamma} \le 1+\frac{2\gamma}{k}\le 2$ for all $k \ge k_1$.
	\item $k_1^{1-\gamma} \ge \frac{16 \gamma}{c}$ if $\gamma \in (\frac{1}{2},1)$. 
\end{enumerate} 
Now let $t_k = s_k k^\gamma$, then we rewrite the above inequality as
\begin{align}\label{eq: tkupdateconvergence}
t_{k+1} \le \left(\frac{k+1}{k}\right)^\gamma \left[(1-ck^{-\gamma})t_k + \frac{C}{k^\gamma}\right], \qquad \text{for all $k \ge k_0$}.
\end{align}
Let $ C_{\ub} = \max\{s_{k_1} k_1^\gamma, 4C, \frac{8C}{c} \}$. By definition of $ C_{\ub}$, we know that 
$$
t_{k_1} = s_{k_1}k_1^\gamma \le  C_{\ub}.
$$ 
For the induction step, we consider two cases.\\
First suppose $t_k \in [0, \frac{1}{4}  C_{\ub}]$. By~\eqref{eq: tkupdateconvergence} and definition of $ C_{\ub}$, we have 
\begin{align*}
	t_{k+1} &\le  \paren{\frac{k_0+1}{k_0}}^{\gamma}(t_k + C)\\
	&\le \paren{\frac{k_0+1}{k_0}}^{\gamma}\paren{\frac{1}{4} C_{\ub} + \frac{1}{4} C_{\ub}}\\
	&\le C_{\ub}.
\end{align*}

Second, suppose $t_k \in [\frac{1}{4}  C_{\ub,k_0,\bar x},  C_{\ub,k_0,\bar x}]$. By~\eqref{eq: tkupdateconvergence} and definition of $\tilde C_{\ub,k_0,\bar x}$, we have 
\begin{align*}
	t_{k+1} &\le \paren{\frac{k+1}{k}}^{\gamma}\left(t_k-\frac{ct_k}{k^\gamma} + \frac{C}{k^\gamma}\right) \\ 
	&\le \paren{\frac{k+1}{k}}^{\gamma}\left( C_{\ub} -\frac{cC_{\ub}}{4k^\gamma}+ \frac{C}{k^\gamma}\right)\\
	&\le \paren{\frac{k+1}{k}}^{\gamma}\left( C_{\ub} -\frac{c  C_{\ub}}{8k^\gamma}\right)\\
	&=  C_{\ub}\paren{\frac{k+1}{k}}^{2\gamma}\paren{1 -\frac{c}{8k^\gamma}}
\end{align*}
We claim that $\paren{\frac{k+1}{k}}^{\gamma}\paren{1 -\frac{c}{8k^\gamma}}\le 1$. Indeed, we have 
\begin{align*}
	&\paren{\frac{k+1}{k}}^{\gamma}\paren{1 -\frac{c}{8k^\gamma}}\\
	&\le \paren{1+ \frac{2\gamma}{k}}\paren{1 -\frac{c}{8k^\gamma}}\\
	&\le 1 + \frac{2\gamma }{k} - \frac{c}{8k^\gamma}
\end{align*}
When $\gamma=1$, $1 + \frac{2\gamma }{k} - \frac{c}{8k^\gamma}$ by our assumption on $c$. When $\gamma \in (\frac{1}{2},1)$, $1 + \frac{2\gamma }{k} - \frac{c}{8k^\gamma} \le 1$ by our choice of $k_0$. 
This completes the induction.
\end{proof}
\end{document}